\documentclass[11pt]{amsart}
\usepackage{amsfonts,amssymb,amscd}
\usepackage{oldgerm}
\usepackage{amsfonts}
\usepackage{amssymb}
\usepackage{graphics}
\usepackage{booktabs}
\usepackage{tikz-cd}

\usepackage{multicol}

\usepackage{booktabs, tabularx}
\usepackage[margin=0.8in]{geometry}

\usepackage{xfrac}

\usepackage[english]{babel}
\usepackage[T1]{fontenc}
\usepackage{ucs}
\usepackage[utf8x]{inputenc}
\usepackage{booktabs}

\usepackage{xfrac}
\usepackage{epic,eepic,pstricks,graphics,color}

\usepackage{ucs}
\usepackage[utf8x]{inputenc}

\usepackage{epic,eepic,pstricks,graphics,color}

\newcommand{\C}{\mathbb{C}}
\newcommand{\R}{\mathbb{R}}
\newcommand{\N}{\mathbb{N}}

\newcommand{\p}{\mathbb{P}}

\newcommand{\fol}{\mathcal{F}}

\newcommand{\RR}{\mathbb{R}}
\newcommand{\CC}{\mathbb{C}}
\newcommand{\PP}{\mathbb{P}}

\newcommand{\OO}{\mathrm{O}}
\newcommand{\Aut}{\mathrm{Aut}}
\newcommand{\Diag}{\mathrm{Diag}}
\newcommand{\Span}{\mathrm{Span}}

\newcommand{\SLR}{{\rm SL}(2,\mathbb{R})}
\newcommand{\SLC}{{\rm SL}(2,\mathbb{C})}
\newcommand{\slR}{\mathfrak{sl}({2},\mathbb{R})}

\newcommand{\rb}{\right)}
\newcommand{\lb}{\left(}

\newcommand{\dd}[1]{\frac{\partial}{\partial #1}}
\newcommand{\D}[1]{\frac{d}{d #1}}

\newcommand{\fg}{\mathfrak{g}}

\newcommand{\ad}{\mathrm{ad}}
\newcommand{\Ad}{\mathrm{Ad}}

\newcommand{\Tr}{\mathrm{Tr}}

\def\picill#1by#2(#3)#4
{\vbox to #2
{\hrule width #1 height 0pt depth 0pt
\vfill\special{illustration #3 scaled #4}}}

\newtheorem{teo}{Theorem}[section]
\newtheorem{proposition}[teo]{Proposition}
\newtheorem{lema}[teo]{Lemma}
\newtheorem{lemma}[teo]{Lemma}
\newtheorem{obs}[teo]{Remark}

\newtheorem{definition}[teo]{Definition}
\newtheorem{coro}[teo]{Corollary}

\date{\today}

\begin{document}

\title[Geodesic completeness of metrics on Lie groups]{Geodesic completeness of pseudo and holomorphic Riemannian metrics on Lie groups}

\author{Ahmed Elshafei, \, \, \, Ana Cristina Ferreira \, \, \, \& \, \, \, Helena Reis}
\address{}
\thanks{}


\begin{abstract}
This paper is devoted to geodesic completeness of left-invariant metrics for real and complex Lie groups. We start by establishing the Euler-Arnold formalism in the holomorphic setting. We study the real Lie group $\SLR$ and reobtain the known characterization of geodesic completeness and, in addition, present a detailed study where we investigate the maximum domain of definition of every single geodesic for every possible metric. We investigate completeness and semicompleteness of the complex geodesic flow for left-invariant holomorphic metrics and, in particular, establish a full classification for the Lie group $\SLC$.
\end{abstract}

\maketitle

\section{Introduction}

The study of left-invariant pseudo-Riemannian metrics on Lie groups and their quotients is an important topic in differential geometry with applications in cosmology,  and general relativity, for the case of Lorentzian signature. Even though the formalism regarding pseudo-Riemanannian geometry is very similar to the Riemannian one, a fundamental issue in this respect is that pseudo-Riemannian metrics
often fail to be geodesically complete, even in the compact case.   In mathematical physics, incompleteness might initially have been regarded as a failure of the model, however,  singularities have become so common (consider, for instance, the classical Schwarzschild spacetime or the Penrose-Hawkings singularity theorems) that it is reasonable to expect incompleteness under general physical assumptions.  In fact, the completeness or incompleteness of a metric becomes a fundamental property of the model (cf. \cite{Sanchez} for a survey on geodesic completeness in the context of mechanical systems).

The notion of geodesic flow on Lie groups can be traced back to Euler's work on the motion of  rigid bodies in $\RR^3$ with a fixed
point. Euler showed that these motions can be described as geodesics on the special orthogonal group ${\rm SO}(3)$. Two centuries later, Arnold showed, in
his seminal paper \cite{Arnold-paper}, that the motion of a rigid body with a fixed point can indeed be described as the geodesics of
a left-invariant pseudo-Riemannian metric on a Lie group. Another important result to be noted is that, only a few years later, Marsden \cite{Marsden} proved that compact pseudo-Riemannian Lie groups (and more generally, homogeneous spaces) are geodesically complete.

The technique introduced by Arnold in \cite{Arnold-paper} lends itself well to investigate the completeness of left-invariant pseudo-Riemannian
metrics. More precisely, Arnold established that these geodesics are in one-to-one correspondence with integral curves of a certain quadratic
(homogeneous) vector field on the corresponding Lie algebra (cf. Section \ref{sectionEA}). We will thus refer to this method as the
``Euler-Arnold formalism'' and to the quadratic vector field in question as the ``Euler-Arnold vector field''. The Euler-Arnold
formalism has a great advantage in the study of completeness of geodesics. In fact, the geodesics become described as the integrals
curves of a quadratic (homogeneous) vector field on $\R^n$. The study of the completeness of a vector field is played on the behavior
at infinity.  Since $\R^n$ (resp. $\C^n$) has a simple compactification, the projective space $\R\p(n)$ (resp. $\C\p(n)$) that allows us
to have control on this behavior at infinity. Also, the fact that the vector field is algebraic (in fact, polynomial) is important.
Indeed, there is a vast literature on algebraic differential equations that can be brought to bear. For example, certain ideas from
\cite{RR_Applications} will be used in the course of this work.

The task of classifying metrics on an arbitrary (non-compact) Lie group remains a daunting one, even in the semisimple case (see Section \ref{sec:SLC}  for more details). However,
for the lowest dimensional simple Lie group $\SLR$ a complete classification can be indeed achieved. As a
matter of fact, the problem of studying (in)complete metrics on $\SLR$ has a history that goes back a few decades. A first study of completeness
of left-invariant pseudo-Riemannian metrics on $\SLR$ was developed by Guediri-Lafontaine \cite{G-L} in 1995. Later, in  2008,
Bromberg and Medina \cite{B-M} provided the complete classification. A more geometric approach to the same
problem was given by Tholozan \cite{Tholozan} in 2014, as part of his Ph.D. thesis.

In this paper, we present an alternative approach to the problem of classification of geodesically complete left-invariant metrics on $\SLR$. This dynamical
approach is build on ideas of \cite{RR_Applications} to estimate the size of the domains of definition of the geodesics, which we followed
 to ensure the conditions under which such a domain is strictly contained in $\R$. This allowed us to reobtain the
classification of complete left-invariant pseudo-Riemannian metrics previously provided in \cite{B-M}. In fact, this allowed us to go further and we also present a detailed study of completeness or incompleteness for every single geodesic of the metrics in question.

The structure of the paper can be essentially divided into three parts. The first part (Sections \ref{Sec: holomorphic-Riem},
\ref{Sec: holomorphic-Lie}, \ref{sectionEA}) is a discussion on (the not so well-known) holomorphic-Riemannian geometry of complex manifolds with special attention to Lie groups,
leading up to the description of the Euler-Arnold formalism. The second part (Sections \ref{sec: LP_equations}, \ref{Sec:Case1},
\ref{Sec:Case3}, \ref{sec:Case2and4}) concerns the classification of completeness of left-invariant pseudo-Riemannian metrics on $\SLR$
and the dynamical behavior of their geodesic flows. The third and final part of the paper (Section \ref{sec:SLC}) is concerned
with left-invariant holomorphic  metrics on $\SLC$.

Sections \ref{Sec: holomorphic-Riem} and~\ref{Sec: holomorphic-Lie} present the preliminary material necessary in the sequel in order
to provide a self-contained exposition. Since the study of pseudo-Riemannian geometry and of real Lie groups is more familiar, we present
our material in the holomorphic setting, bearing in mind that most of the constructions can find direct analogues in the real context and
pointing out the differences where needed.

In Section~\ref{sectionEA}, we develop the Euler-Arnold formalism for completeness of holomorphic metrics. There are several approaches in the literature to
establish this result; in our work, we used the methods (albeit in the real case) proposed by Tholozan \cite{Tholozan}  since they are more adapted to the complex setting. In the case of an orthogonal Lie group, the Euler-Arnold equations are more tractable and can be formulated in terms of a Lax-pair. The added feature that Lax equations (or, equivalently, the Euler-Arnold equations) come with two first integrals allowed us to establish the following general result.

\begin{teo}\label{Teo: 1.1}
Let $G$ be a Lie group equipped with a bi-invariant pseudo-Riemannian metric. Then $G$ can be endowed with a complete left-invariant pseudo-Riemannian metric of every possible signature.
\end{teo}

In Sections \ref{sec: LP_equations}, \ref{Sec:Case1}, \ref{Sec:Case3}, \ref{sec:Case2and4} we focus on $\SLR$. In Section~\ref{sec: LP_equations}, we present two important lemmas which allow us to control the
orthogonality and bracket relations on $\slR$. This outlines the special character of $\SLR$ and makes possible to have an explicit
description of the Euler-Arnold differential system. The study is conducted in four subcases corresponding to the normal forms of
the isomorphism which relates the metric in question to the Killing form. Sections \ref{Sec:Case1}, \ref{Sec:Case3} and \ref{sec:Case2and4}
are devoted to the detailed study of geodesic (in)completeness for each of the subcases established in the previous section. For
example, it was known that every single metric whose associated isomorphism has a (unique real) eigenvalue with algebraic multiplicity
3 but geometric multiplicity $1$ admits an incomplete geodesic. We show in this paper that there is no geodesic for
such a metric whose domain of definition is a finite interval $]a,b[$ for some $a,b \in \R$ with $a<b$. To be more precise, we proved the
following.

\begin{teo}
Consider a metric whose associated isomorphism has a (unique real) eigenvalue with algebraic multiplicity 3 but geometric multiplicity~$1$. Then there exists an invariant plane for the geodesic flow over which all geodesics are complete. Furthermore, all the
other geodesics are $\R^+$ or $\R^-$ complete, although not complete.
\end{teo}

This contrasts with the case where the isomorphism associated with the metric has two complex (non-real) eigenvalues. In this case
the domain of definition of almost all geodesics is a finite interval, $]a,b[$ for some $a,b \in \R$ with $a<b$. In fact, we proved the following.

\begin{teo}
Consider a metric whose associated isomorphism has two complex (non-real) eigenvalues. Then there exists an invariant plane for the geodesic
flow over which all geodesics are complete. Furthermore, there exists an invariant plane over which some of the geodesics are $\R^+$ or
$\R^-$ complete and the others have a finite interval as domain of definition. All the remaining geodesics are incomplete with a finite
interval as domain of definition.
\end{teo}

The final section, Section~\ref{sec:SLC}, is devoted to the study of $\SLC$.
We start by presenting brief considerations why a direct and similar approach to that of $\SLR$ is unfeasible.  We then turn our attention to the complex Lie group $\SLC$ and obtain a complete characterization for completeness of holomorphic Riemannian left-invariant metrics. More precisely, we have the following theorem.

\begin{teo}
Let $q$ be the holomorphic metric on $\mathfrak{sl}(2,\mathbb{C})$ defined by $q(X,Y) = B(\Phi X, Y)$, where $B$ is the Killing form
and $\Phi$ a $B$-self-adjoint isomorphism. Then, $q$ is a complete holomorphic metric if and only if $\Phi$ has an eigenvalue whose
eigenspace has dimension at least 2.
\end{teo}

As a corollary, we show that every complex simple Lie group can be endowed with an incomplete holomorphic metric.

As seen from
the example of $\SLC$,  complete holomorphic metrics are ``rare''
 and, therefore, the problem of classifying semicomplete holomorphic metrics becomes more interesting.
A vector field or, equivalently, a differential system, is said to be semicomplete if its solutions admit a maximal domain of definition (that
need not coincide with $\C$), see \cite{Reb}. If a vector field is complete, then it is necessarily semicomplete. Although the converse does not
hold, a semicomplete vector field can be completed in the sense that it can be realized by a complete vector field on a suitable space that
need not be a Lie group any more (see \cite{Palais}, \cite{FRR}, \cite{Guillot} for further details). In fact, this new space is not necessarily Hausdorff. It is perhaps relevant to understand these completions as they are important in the theory of transformation groups and may be relevant in applications to physics or other fields of geometry.

Finally, as a unexpected outcome of our work concerning the general theory of semicomplete vector fields, we provide an example of
a family with $2$-complex parameters of iso-spectral quadratic semicomplete vector fields on $\C^3$, up to
linear conjugation with the terminology of \cite{Guillot}, see Theorem~\ref{Thm:semicomplete}. Note that this family is rather
different from the $2$-parameter family provided in \cite{Guillot-CRAS} and sits closer to the framework of ``quadratic elliptic
foliations on the complex projective plane'' as in \cite{gautier}.


\section{Holomorphic-Riemannian geometry}\label{Sec: holomorphic-Riem}

The foundations of pseudo-Riemannian geometry are well-known and there is extensive literature on the subject (see, for instance, \cite{ON}). In this section.  we will briefly establish preliminary notions and results of its holomorphic counterpart, bearing in mind that most of definitions and results apply analogously in the smooth real setting.  We will follow the approach of LeBrun \cite{LeBrun} and point out the differences where relevant. The topic of holomorphic-Riemannian geometry has been studied in more recent years in the works of Biswas, Dumitrescu and Zeghib (cf.  \cite{B-D, D,D-Z}).

\smallskip

Let $M$ be a complex manifold of (complex) dimension $n$ with holomorphic tangent and cotangent bundles denoted, respectively, by $T_0 M$ and $T_0^\ast M$ .

A holomorphic Riemannian metric (or, simply, holomorphic metric) on $M$ is a holomorphic covariant symmetric 2-tensor i.e. a holomorphic section $g: M \longrightarrow T^\ast_0 M \otimes T^\ast_0 M$ which is symmetric and non-degenerate.

In terms of local coordinates, a holomorphic metric $g$ can be expressed as follows. If  $\mathcal{U}$ is a domain of $M$ such that we have a chart $\mathcal{U} \longrightarrow \mathbb{C}^n$ with coordinates $(z_1,\cdots, z_n)$, then $g|_\mathcal{U}$ can be written as
$$
g|_\mathcal{U} = \sum_{i\leq j} g_{ij}dz^i\otimes dz^j
$$
with $\bar{\partial}g_{ij}=0$ (i.e. $g$ is holomorphic), $g_{ij}=g_{ji}$ (i.e. $g$ is symmetric).

Clearly, such a tensor carries no signature, but it is still possible to define the notion of non-degenerate metric by prescribing that the map $T_0 M \longrightarrow T_0^\ast M$ given by $X\longmapsto g(X, -)$ is an isomorphism between the tangent and the cotangent holomorphic bundles. As in the real setting, this condition is translated by $\mathrm{det}[g_{ij}]\neq 0$.

A holomorphic-Riemannian manifold  is thus a pair $(M,g)$ where $M$ is a complex manifold and $g$ is a holomorphic Riemannian metric on $M$.

Notice that the notion of holomorphic metric should not be confused with that of Hermitian metric. Observe also that the real part of a holomorphic metric is a pseudo-Riemannian metric of signature $(n,n)$. Indeed, it is always possible to find a local basis of holomorphic vectors $(X_1, \cdots, X_n)$ such that $g(X_i, X_j) = \delta_{ij}$ and then $(X_1, \cdots, X_n, iX_1,\cdots, i X_n)$ is a local basis for the underlying real manifold  which is orthonormal for $\mathrm{Re}\, g$ with signature $(n,n)$.

It is a well-known fact that a compact Lorenztian manifold has Euler characteristic equal to zero, and this is simply due to the fact that there must exist a nowhere vanishing (timelike) vector field.  The existence of a holomorphic metric for compact complex manifolds is also restrictive in terms of its topology, cf. \cite{D}. The existence of such a metric fixes an isomorphism between the holomorphic tangent and the holomorphic cotangent bundles, which implies that the canonical bundle $K = \Lambda^n (T^\ast_0 M)$ of $M$ is isomorphic to its dual, the anticanonical bundle $K^\ast$, and this yields the vanishing of the first Chern class of $M$, since $c_1(M) = c_1(K) = -c_1(K^\ast)$.

\smallskip

With the notion of holomorphic metric understood, we can now introduce the homolorphic Levi-Civita connection.
It can easily be proved, just like in the real setting that there exists a unique affine connection, which we will call the Levi-Civita connection, that is torsion-free and preserves the holomorphic metric, cf. \cite{LeBrun}. Moreover, this connection is also defined by the Kozsul formula as in pseudo-Riemannian geometry.



\smallskip

Given a holomorphic-Riemannian manifold $(M,g)$, an isometry is defined to be a biholomorphic map $\phi: M \longrightarrow M$ which preserves $g$, that is, such that the pullback satisfies $\phi^\ast g = g$. If the biholomorphism $\phi$ is defined only locally (between two open sets of $M$) we say that $\phi$ is a local isometry.

As in the real case, the set of local isometries is a pseudo-group for composition.
A similar proof to that of the classical result of pseudo-Riemannian geometry shows that the following three statements are equivalent for a holomorphic vector field $X$.
\begin{enumerate}
\item[(i)]  The (local) flow of $X$ preserves $g$, i.e., it acts by isometries;
\item[(ii)]  $\nabla X$ is a pointwise $g$-skew symmetric endomorphism of $T_0 M$;
\item[(iii)] $\mathcal{L}_X g =0$, where $\mathcal{L}$ represents the Lie derivative.
\end{enumerate}

A vector field which satisfies any of the above condition is called a Killing vector field.  Note that the set of Killing vector fields forms a Lie algebra with respect to the standard Lie bracket of holomorphic vector fields.
\smallskip

A notion of parallel transport can also be obtained in a classical fashion.  Consider $\gamma: A \longrightarrow M$, $A \subseteq \mathbb{C}$,  an immersed (i.e. $\dot{\gamma}(t) \neq 0$, for all $t \in A$) holomorphic curve, and let $\nabla_{\dot{\gamma}}$  denote the covariant derivative along $\gamma$.





\begin{definition}
The curve $\gamma$ is called a geodesic if $\dot{\gamma}$ is parallel along $\gamma$, that is, if $\nabla_{\dot{\gamma}}\dot{\gamma} =0$.
\end{definition}

Using local coordinates and considering the Christoffel symbols $\Gamma_{ij}^k$,  the curve $\gamma$ is a geodesic of $(M,g)$ if and only if it satisfies the equations
$$\ddot{\gamma}^k(t) + \dot{\gamma}^j(t) \dot{\gamma}^i(t) \Gamma_{ij}^k(\gamma(t))=0,$$
which are perfectly analogous to the geodesic equations of pseudo-Riemannian geometry.

\smallskip

Using the theorem of existence and uniqueness of ordinary differential equations, we know that given  a point $p$ in $M$ and a vector $v$ in the holomorphic tangent space of $M$ at $p$, there exists a ball $B(0,\delta)$ centered at $0\in \mathbb{C}$ with radius $\delta >0$ such that $\gamma: B(0,\delta) \longrightarrow M$ is the unique geodesic verifying $\gamma(0) = p$ and $\dot{\gamma}(0)= v$.

\smallskip

We can now introduce the notion of completeness and semicompleteness for geodesics.

\begin{definition}\label{def: geod-semicomplete}
Let $(M,g)$ be a holomorphic Riemannian manifold and $\nabla_{\dot{\gamma}}\dot{\gamma} =0$ be the geodesic equation of $(M,g)$. We say that the geodesic equation is semicomplete on an open set $U\subseteq M$ if for every $p \in U$ and $v \in T_p M$, there exists a neighborhood $V_p$ of $0\in \mathbb{C}$
and a curve $\gamma: V_p \longrightarrow U$ which satisfies the following conditions:
\begin{itemize}
\item[(i)] $\gamma(0) = p$, $\gamma'(0) = v$ and $\nabla_{\dot{\gamma}}\dot{\gamma} =0$;
\item[(ii)] for every sequence $\{t_i\}_{i\in \mathbb{N}}\subset V_p\subset \mathbb{C}$ which converges to a point $\hat{t}$ in the boundary of $V_p$, the sequence $\{\gamma(t_i)\}_{i\in \mathbb{N}}$ escapes every compact set of $U$.
\end{itemize}
\end{definition}

\begin{definition}
For a holomorphic Riemannian manifold $(M,g)$, the geodesic equation is said to be complete if there is a map $\Gamma\!\!: \mathbb{C}\times M \longrightarrow M$ such that for every $p$ in $M$, the curve $\gamma(t) = \Gamma(t,p)$ satisfies condition (i) of Definition \ref{def: geod-semicomplete}.

\end{definition}

The notion of geodesic completeness is clearly perfectly analogous to that of the real setting. Our definition of geodesic semicompleteness here is
a particular case of the general definition for (germs of) vector fields on complex manifolds, which was introduced by Rebelo in \cite{Reb}. In fact,
Definition~\ref{def: geod-semicomplete} can be rephrased by saying that the vector field on $TM$ associated with the geodesic flow is semicomplete
on $TU \subseteq TM$.

\smallskip

Note that for real ordinary differential equations the definition of semicompleteness is moot and the notion of maximal domain (or interval) of definition is well understood. This is an important difference between real and complex ODEs, since for the generic complex ODE the standard gluing procedure of local domains might define a multivalued solution (for more details, the reader is referred to the survey \cite{RR}). However, in the case where the geodesic equation is semicomplete on $M$, owing to condition (ii) of definition \ref{def: geod-semicomplete}, we will say that  $V_p$ is the \emph{maximal domain} of definition of the geodesic $\gamma$,  and our geodesics are well-defined in the sense of being univalued.

\medskip

They following lemma provides and interesting relation between Killing fields and geodesics and will also be useful in  Section \ref{sectionEA}.

\begin{lemma}\label{lem: Killing-geodesic-constant}
Let $(M,g)$ be a holomorphic Riemannian manifold. If $X$ is a Killing field and $\gamma$ is a geodesic then $g(X_{\gamma(t)},\dot{\gamma}(t))$ is constant.
\end{lemma}

\begin{proof}

By definition of connection along a curve we have that
$\D{t}g(X,\dot{\gamma})=g(\nabla_{\dot{\gamma}}X,\dot{\gamma})+g(X,\nabla_{\dot{\gamma}}\dot{\gamma}).$
Since $\gamma$ is a geodesic the above expression simplifies to
$\D{t}g(X,\dot{\gamma})=g(\nabla_{\dot{\gamma}}X,\dot{\gamma}).$
Since $X$ is a Killing field then
$\D{t}g(X,\dot{\gamma})=-g(X,\nabla_{\dot{\gamma}}\dot{\gamma})$
Thus, using again the fact that $\gamma$ is a geodesic,  $\D{t}g(X,\dot{\gamma}) = 0$ and $g(X, \dot{\gamma})$ is constant.

\end{proof}


\section{Holomorphic-Riemannian Lie groups}\label{Sec: holomorphic-Lie}

We will now focus on the case where our manifold is a Lie group. A very detailed exposition of this topic, in the (perhaps less familiar) complex setting  can be found in the textbook by Lee, \cite{Lee}. Here we will only recall the basic material needed for our purposes.
As in the previous section, we will make our exposition in the complex case while bearing in mind the analogous definitions and results for the real case and pointing out the differences where needed.

\smallskip

The definition of a complex Lie group is very similar to that of a real Lie group, but instead of considering smooth group operations, we must considere holomorphic ones. More precisely, a connected complex manifold $G$ which is also equipped with a group structure is said to be a complex Lie group if the multiplication and the inversion operations are holomorphic maps.

\smallskip

Let $G$ be a complex Lie group with identity $e\in G$. As usual, we will denote by $L_g$
(resp. $R_g$, $C_g$) the left translation (resp. right translation, conjugation) map by $g \in G$ on $G$.

Take $X$ to be a vector field on the Lie group $G$.
Recall that $X$ is said to be left-invariant (resp. right-invariant) if it
coincides with its pullback by left
translations (resp. right translations).

Naturally, given a vector $x \in T_e G$, the vector $x$ can be extended to a left-invariant vector field $X$ by left-translating the tangent vector $x$ to other points of $G$. This allows us then to identify
the tangent space of $G$ at the identity, $\fg = T_e G$, with the space of left-invariant vector fields on $G$. Also, it is simple to see that such an identification is a $\CC$-linear map, thus endowing $\fg$ with the structure of a complex Lie algebra.

\smallskip

Let us recall the adjoint
representation and the infinitesimal adjoint representation of a Lie group $G$ which we will denote, as usual, by
$\mathrm{Ad}$ and $\mathrm{ad}$, respectively.

The adjoint representation of $G$ is the map
$\Ad:  G  \longrightarrow  \mathrm{GL}(\fg)$
where $\mathrm{Ad}_g$ is the linear map satisfying
$\Ad_g (x) = (D_e C_g) (x).$
Thus, the adjoint representation of $G$ is the derivative at the identity of the
conjugation map $C_g$ (viewed as an element of ${\rm Aut} (G)$, the automorphism group of $G$).
Note that, with the standard complex structure on $\mathrm{GL}(\fg) \simeq \mathrm{GL}(n, \mathbb{C})$, where $n=\mathrm{dim}\, G$, it is possible to prove that $\mathrm{Ad}$ is a holomorphic map.

\smallskip


The infinitesimal adjoint representation of $G$ (or, equivalently, the adjoint representation of the
Lie algebra $\fg$ associated to $G$) is the map
$\mathrm{ad}:  \fg  \longrightarrow  \mathrm{End}(\fg)$
where $\ad_x$ is defined as
$\ad_x y = (D_e \Ad(x))(y).$
The infinitesimal adjoint representation of $G$ is then the derivative at the identity of the adjoint representation. Moreover, we have the remarkable fact that
$$
\ad_x y = [X,Y]_e
$$
where $x,y$ are vectors in $T_e G$ and $X,Y$ are the left-invariant vector fields associated to $x$ and $y$ in the identification of $\fg = T_e G$ with the Lie algebra of left-invariant vector fields. From henceforth, we will simply write, where convenient and as is standard, $\ad_x y = [x,y]$.


\begin{obs}
If $G$ is a matrix group, that is, a Lie subgroup of the general linear group $\mathrm{GL}(m,\CC)$, for some $m\in \N$, then the adjoint representation is simply given by
$
\Ad_g(x) = gxg^{-1} ,
$
for every $g\in G$ and $x \in \fg$; also the infinitesimal adjoint representation is simply the usual commutation of matrices, i.e.,  for every  $x,y \in \fg$
$
\ad_x(y) = xy-yx .
$
\end{obs}

\smallskip

Consider now a holomorphic Riemannian metric $q$ on the complex Lie group $G$. The metric $q$ is said to be left-invariant
if all left translations are isometries, that is, if $(L_g)^\ast q = q$, for all
$g \in G$ or, more concretely, if
\begin{equation*}
q_{h}(x,y) = q_{gh}\lb (D_h L_g) x, (D_h L_g)y \rb ,
\end{equation*}
for all points $g,h$ of $G$ and all vectors $x,y$ of $T_hG$. Analogously, a metric $q$ is right-invariant if the right translations are isometries.

It is easy to see that there is a one-to-one correspondence between left-invariant holomorphic Riemannian metrics on the group $G$
and non-degenerate complex bilinear forms on the corresponding Lie algebra $\fg$.

Let us also recall the special case of bi-invariant metrics. A holomorphic Riemannian metric is said to be bi-invariant if it is both left-invariant and right-invariant. Notice that,
for a connected group $G$, this is equivalent to having $\ad_x$ be skew-symmetric with respect to $q$, that is,
\begin{equation}
q(\ad_x y,z)+q(y,\ad_x z) = 0 \label{eq: ad-invariant}
\end{equation}
for all $x,y,z \in \fg$.

The bilinear form $\kappa$ defined by
$\kappa(x,y) = \Tr(\ad_x \circ \ad_y)$
for every $x,y\in \fg$, where $\Tr$ stands for the trace of an endomorphism $\fg \longrightarrow \fg$, is called the Killing form of the Lie algebra $\fg$.

The Killing form always satisfies the ad-invariance condition of Equation (\ref{eq: ad-invariant}).  Also, a famous result of Cartan tells us that $G$ is a semisimple Lie group if and only if its Killing form is non-degenerate. Therefore, for complex semisimple Lie groups we have a preferred bi-invariant holomorphic Riemannian metric.


\smallskip

Consider now any Lie group (real or complex, not necessarily semisimple) with Lie algebra $\fg$ that can be equipped with a bi-invariant metric $B$. In the real setting, such groups are known as orthogonal, quadratic or quasi-classical Lie groups, \cite{MR}.  We will now show that the set of left-invariant metrics $q$  on $\fg$ is in one-to-one correspondence with $B$-self-adjoint isomorphisms $\Phi: \fg \longrightarrow \fg$ via the equality
\begin{equation*}\label{eq: Phi}
q(x,y) = B(\Phi x,y),
\end{equation*}
for every $x,y \in \fg$. Given a $B$-self-adjoint isomorphism $\Phi $, it is clear that the equation above defines a non-degenerate symmetric bilinear form (possibly of different signature than that of $B$, in the real setting) on $\fg$. Conversely, given a metric $q$ we can define the isomorphism $\Phi$ by the following commutative diagram

\[\begin{tikzcd}
\mathfrak{g} \arrow[r, "A_{q}"] \arrow[swap, d,"\Phi"] &
\mathfrak{g}^{\ast} \arrow[d, "\mathrm{id}"] \\
\mathfrak{g} \arrow[r,"A_{B}"] &
\mathfrak{g}^{\ast}  \end{tikzcd}
\]
where $A_q$ (resp. $A_B$) is the standard isomorphism between $\fg$ and $\fg^\ast$ given by $q$ (resp. $B$).



\smallskip

For a Lie group $G$ equipped with a left-invariant metric, the following lemma is an immediate consequence of the definitions but yet a noteworthy fact.

\begin{lemma}\label{lem: righ-invariant-Killing}
 If $X$ is a right-invariant vector field then $X$ is a Killing field.
\end{lemma}

\begin{proof}
Let $X$ be a right-invariant vector field. We start by observing that if $\varphi(t)$ is the one-parameter subgroup of $X_e \in \fg$ then the flow ${\Psi^t_X}$ of $X$ is given by ${\Psi^t_X}(g) = \varphi(t)g = L_{\varphi(t)}g$.  By definition of Lie derivative,
\begin{equation*}
\mathcal{L}_X q = \lim_{t\rightarrow 0} \frac{1}{t}\lb {(\Psi^t_X)^*}q - q \rb
\end{equation*}
But ${(\Psi^t_X)^*}q = L^\ast_{\varphi(t)} q$ from the above and  $L^\ast_{\varphi(t)} q = q$ since the metric is left-invariant. The claim then follows.
\end{proof}

\smallskip


\section{The Euler-Arnold formalism}\label{sectionEA}

In his celebrated article of 1966, \cite{Arnold-paper}, Arnold showed that the motions of a rigid body with fixed point can be seen as geodesics of a Lie group equipped with a left-invariant metric. This was a generalization of the result obtained by Euler for the particular case of rigid motions on $\mathbb{R}^3$. For this reason, the result of Theorem \ref{Thm: EA-equation} (Equation (\ref{eq: EA-equation}) below) is known as the Euler-Arnold equation for geodesics of Lie groups. For the proofs, we will follow the approach of N. Tholozan in this PhD thesis \cite{Tholozan}, which is very suitable to our complex setting.

\smallskip

Let $A$ be an open domain in $\mathbb{C}$ and $\gamma: A \longrightarrow G$ be an immersed holomorphic curve in $G$. Using left translations, we can define the associated curve $x: A \longrightarrow \fg$ in the Lie algebra $\fg$ of $G$, for every $t\in A$, as follows
\begin{equation*}
x(t) = D_{\gamma(t)}L_{\gamma^{-1}(t)}\dot{\gamma}(t).
\end{equation*}

\begin{teo}[Arnold, \cite{Arnold-paper, Arnold-book}]\label{Thm: EA-equation} Let $(G, q)$ be a holomorphic Riemannian Lie group. The curve $\gamma: A \longrightarrow G$ is a geodesic if and only if the associated curve $x: A \longrightarrow \fg$ satisfies, for every $t\in A$, the equation
\begin{equation}
\dot{x}(t) =  \mathrm{ad}_{x(t)}^\dagger x(t),\label{eq: EA-equation}
\end{equation}
where $\mathrm{ad}_{x(t)}^\dagger$denotes the formal adjoint of $\mathrm{ad}_{x(t)}$ with respect to $q$.
\end{teo}

\begin{proof}
For simplicity, we will present the proof for matrix groups only. We recall that, by Ado's theorem, every Lie algebra is isomorphic to the Lie algebra of a matrix group with the usual commutator, so we trust that the reader will not find our proof too restrictive.

If $G$ is a matrix group and $\gamma: A \longrightarrow G$ is a curve in $G$ then its associated curve $x:A \longrightarrow \fg$ is simply given by the expression $x(t) = \gamma^{-1}(t) \dot{\gamma}(t).$

Suppose that $\gamma: A \longrightarrow G$ is a geodesic. Let $z$ be any element in the Lie algebra $\fg$. By Lemma \ref{lem: righ-invariant-Killing}, the right-invariant vector field $Z_{\gamma(t)} := z \gamma(t)$ is a Killing field along the curve $\gamma$. Also, according to Lemma \ref{lem: Killing-geodesic-constant}, $q(Z_{\gamma(t)}, \dot{\gamma}(t))$ is constant. Since the metric $q$ is left-invariant, then
\begin{equation*}
c = q\lb Z_{\gamma(t)}, \dot{\gamma}(t) \rb
= q \lb \gamma^{-1}(t) Z_{\gamma(t)}, \gamma^{-1}(t) \dot{\gamma}(t)\rb
= q\lb \Ad_{\gamma^{-1}(t)}z ,  x(t)\rb
\end{equation*}
for some $c\in \CC$ and for all $t\in A$. Taking the derivative with respect to $t$, we have that
\begin{equation*}
0 = \frac{d}{dt}q\lb  \Ad_{\gamma^{-1}(t)}z ,  x(t)\rb
= q\lb \dot{x}(t), \Ad_{\gamma^{-1}(t)}z\rb
+ q\lb x(t), \frac{d}{dt} \Ad_{\gamma^{-1}(t)}z\rb .
\end{equation*}
From the equation above, by taking derivatives and noting that
$
\frac{d}{dt}\gamma^{-1}(t) = - \gamma^{-1}(t) \dot{\gamma}(t) \gamma^{-1}(t)$ and thus
$
\frac{d}{dt}\Ad_{\gamma^{-1}(t)}z = \gamma^{-1}(t)z\dot{\gamma}(t) - \gamma^{-1}(t)\dot{\gamma}(t)\gamma^{-1}(t)z \gamma(t) ,
$
we get
\begin{equation*}
q\lb \dot{x}(t), \Ad_{\gamma^{-1}(t)} (z)\rb  =  - q\lb x(t) , \gamma^{-1}(t)z\dot{\gamma}(t) - \gamma^{-1}(t)\dot{\gamma}(t)\gamma^{-1}(t)z \gamma(t)\rb.
\end{equation*}
Now using the definitions of $\Ad$ and of $x(t)$ we can rearrange our equation to
\begin{equation*}
q\lb\dot{x}(t), \Ad_{\gamma^{-1}(t)} z\rb
= - q\lb x(t), \lb \Ad_{\gamma^{-1}(t)}z\rb x(t) - x(t)\lb \Ad_{\gamma^{-1}(t)}z\rb \rb
\end{equation*}
Recalling that for matrix groups the adjoint map $\ad$ is given by the commutator, we obtain
\begin{equation*}
q\lb \dot{x}(t), \Ad_{\gamma^{-1}(t)} z\rb
 =  q\lb x(t), \ad_{x(t)}{\Ad_{\gamma^{-1}(t)} z} \rb
\end{equation*}
Finally by taking the formal adjoint, we can write
\begin{equation*}
q\lb \dot{x}(t), \Ad_{\gamma^{-1}(t)} z\rb
= q\lb \ad^\dagger_{x(t)}x(t), \Ad_{\gamma^{-1}(t)}z \rb .
\end{equation*}

Now using the fact that $z$ is a generic element in $\fg$ and $\Ad_g$ is surjective for all $g\in G$ and also that $q$ is non-degenerate, we can conclude that $\dot{x}(t) = \ad^\dagger_{x(t)} x(t)$.

Conversely, suppose that $x(t)$ is an integral curve of  the Euler-Arnold equation and take $v=x(0)$. There exists a unique geodesic $\gamma$ in $G$ such that $\gamma(0)=e$ and $\dot{\gamma}(0) =v$. Then, by the statement proved above, $y(t) = \gamma^{-1}(t)\gamma(t)$ is also an integral curve of the Euler-Arnold equation. Since $y(0) = \gamma^{-1}(0)\dot{\gamma}(0) = v$, by uniqueness, we conclude that $x(t)=y(t)$, in an open neighborhood of $0\in\CC$. Thus $x$ is an associated curve in $\fg$ of a geodesic of $G$.
\end{proof}

We remark that the theorem above, which we will call, as usual in the literature, the Euler-Arnold theorem, is valid for any Lie group equipped with a left-invariant metric. For Lie groups which can be equipped with a bi-invariant metric,  we have the following simplification of the Euler-Arnold equation.

\begin{proposition}
Let $G$ be an orthogonal Lie group equipped with a bi-invariant metric $B$. Let $q$ be a left-invariant metric on $G$ and $\Phi$ be the B-self-adjoint isomorphism defined by
$q(y,z) = B(\Phi y, z),$
for all $y,z\in \fg$. A curve $\gamma: A \longrightarrow G$ is a geodesic in $G$ if and only if its associated curve $x: A \longrightarrow \fg $ in $\fg$ is an integral curve of
$\Phi\dot{x}(t)= [\Phi x(t), x(t)] .$
\end{proposition}
\begin{proof}
 Let $z$ be any element in $\fg$. Since $x(t)$ satisfies the Euler-Arnold equation, we have that $q(\dot{x}(t), z) = q(\ad^\dagger_{x(t)} x(t),z)$. It thus follows that $B(\Phi \dot{x}(t),z) = B(\Phi x(t), \ad_{x(t)} z).$ Since $B$ is bi-invariant we can write that
$B(\Phi \dot{x}(t), z) = B([\Phi x(t),x(t)],z)$ and hence the claim follows.
\end{proof}

Recall that a first integral of a complex differential equation is a holomorphic function which is constant
along its solutions. The Euler-Arnold equation comes with the following first integrals.

\begin{proposition}\label{Prop: FI-EA}
The functions $I(x) =B(\Phi x,x)$ and $J(x) = B(\Phi x, \Phi x)$ are first integrals of the Euler-Arnold equation $\Phi\dot{x}(t) = [\Phi x(t), x(t)]$.
\end{proposition}
\begin{proof}
It suffices to prove that $\frac{d}{dt} I (x(t)) = 0$ and $\frac{d}{dt} J (x(t)) = 0$. We have that
\begin{equation*}
\frac{d}{dt} I(x(t)) = \frac{d}{dt} B(\Phi x(t),x(t)) = B(\Phi\dot{x}(t), x(t))+ B(\Phi x(t), \dot{x}(t)).
\end{equation*}
Using the fact that $\Phi$ is $B$-self-adjoint and also the Euler-Arnold equation, we get that
\begin{equation*}
\frac{d}{dt} I(x(t)) =  2 B(\Phi\dot{x}(t), x(t)) = 2 B([\Phi x(t), x(t)], x(t))
\end{equation*}
and since $B$ is bi-invariant then $\frac{d}{dt}I(x(t))=0$. The result is analogous for $J(x(t))$.
\end{proof}

The following result is known as the Lax-pair formulation of the Euler-Arnold equations.

\begin{coro}\label{cor: lax-pair}
Let $G$ be an orthogonal Lie group as formulated above. The geodesics of $G$ are in one-to-one correspondence with
curves in $\fg$ which satisfy the equation
\begin{equation*}
\dot{z}(t) = [z(t), \Phi^{-1} z(t)] \, .
\end{equation*}
Furthermore, the functions $F (z) = B(z, z)$ and $G (z) = B(z,\Phi^{-1} z)$ are first integrals of the
Lax-pair equation $\dot{z}(t) = [z(t), \Phi^{-1} z(t)]$.
\end{coro}

\begin{proof}
This is immediate from the two previous propositions by taking $z=\Phi x$.
\end{proof}

\smallskip

As can be seen from the discussion above, for orthogonal Lie groups with a pseudo or holomorphic Riemannian metric, the geodesic equation, which is for general manifolds an ODE of order 2, becomes an ODE of order 1 (or, equivalently, a vector field) in Euclidean space which comes with the added feature of having two first integrals.  This allows us to prove following result in the real setting.

\begin{teo}\label{teo:complete-PR-metrics}
Let $G$ be a Lie group equipped with a bi-invariant pseudo-Riemannian metric $B$.  Consider the set of left-invariant metrics for which the isomorphism $\Phi$ is diagonalizable. Then, for every possible signature, there is an open set of eigenvalues of $\Phi$ which corresponds to a set of complete left-invariant pseudo-Riemannian metrics.
\end{teo}

\begin{proof}
Suppose that $B$ has signature $(p,q)$ with $p+q=n$, where $n=\mathrm{dim}\, G$.  Fix, for the remainder of the proof, an orthogonal basis $(e_1, \cdots, e_p, f_1,\cdots,f_q)$ such that $B(e_i,e_i) =1=-B(f_j,f_j)$. Consider a pseudo-Riemannian metric $q$ such that $q(x,y) = B(\Phi x, y)$ where $\Phi$ is a linear isomorphism which is represented by a diagonal matrix with respect to our fixed basis.  Supposing that $\Phi^{-1}=\mathrm{diag}(\nu_1,\cdots,\nu_p, \mu_1, \cdots, \mu_q)$ then $q$ is represented by $\mathrm{diag}(1/\nu_1,\cdots,1/\nu_p, -1/\mu_1, \cdots, -1/\mu_q)$.   Notice that the signature of $q$ is determined by the signs of $\nu_i,\mu_j$.

If $(z_1,\cdots,z_p, w_1,\cdots, w_q)$ stands for the chosen coordinates of $z\in\frak g$ then the first integrals of Corollary \ref{cor: lax-pair} are given by
\begin{align*}
I_1(z) = & \, z_1^2+\cdots +z_p^2 - w_1^2-\cdots -w_q^2 \\
I_2(z) = & \, \nu_1 z_1^2+\cdots+ \nu_p z_p^2 -\mu_1 w_1^2-\cdots\ -\mu_q w_q^2  \, ,
\end{align*}
and they are linearly independent for a generic metric.

Suppose that $q$ has signature $(p+r, q-r)$ where $0\leq r \leq q$.  Such a signature can be obtained by taking, for instance,
$$\nu_1,\cdots,\nu_p >1, \quad \mu_1,\cdots, \mu_r <0 \quad \mbox{and} \quad 0< \mu_{r+1},\cdots,\mu_{q}<1.$$
Consider then the new first integral for the Lax-pair system given by $J=I_2-I_1$. More precisely,
$$J(z) = (\nu_1-1) z_1^2 + \cdots + (\nu_p-1)z_p^2+(1-\mu_1) w_1^2 + \cdots +(1-\mu_p)w_p^2$$
which, with our choices, is clearly a quadratic positive definite first integral. This implies that the integral curves are all contained in a compact part of $\mathbb{R}^n$ and the associate geodesics are, therefore, complete curves.

For metrics with signature $(p-r,q+r)$ where $0\leq r \leq p$ the proof follows by an analogous argument or, simply, by symmetry by replacing the metric $q$ with $-q$.
\end{proof}

Notice that, in particular, this proves Theorem \ref{Teo: 1.1} presented in the Introduction section.

\section{Lax-pair equations on $\mathrm{SL}(2,\mathbb{R})$}\label{sec: LP_equations}

Let us consider the semisimple (in fact, simple) Lie group $G=\SLR$ and its corresponding Lie algebra $\slR$ equipped with a left-invariant metric $q$. Let $\Phi: \slR \longrightarrow\slR$ be the (unique) $B$-self-adjoint isomorphism satisfying $q(x,y) = B(\Phi x,y)$, for every $x,y \in
\slR$, where $B$ stands for the Killing form of $\slR$.

Let us mention that, in this section,  a particular normalization for the Killing form will be taken. More precisely, we let
$B(x,y)= \frac{1}{2} \Tr(\ad_x\circ \ad_y)$.  Since,  up to a constant, there can only be one bi-invariant bilinear form on a simple Lie algebra, we can conclude that $B(x,y) = 2\Tr(xy)$. Also, it can easily be checked that $B$ has signature $(2,1)$.

\begin{definition}\label{def: B-ON_B-PON}
Let $v=(v_k)$ be a basis of $\slR$. We say that $v=(v_k)$ is $B$-orthonormal if, up to reordering its elements, we have
\begin{equation}\label{def: B-ON}
B(v_1,v_1) = B(v_2,v_2) = -B(v_3,v_3)=1
\;\;\;\; \text{ and } \;\;\;\;
B(v_k,v_l)=0
\;\; , \;\;
\forall  k\neq l .
\end{equation}
In turn, we say that $v=(v_k)$ is $B$-pseudo-orthonormal if, up to reordering its elements, we have
\begin{equation}\label{def: B-PON}
B(v_1,v_1) = B(v_2,v_3) = 1
\;\;\;\; \text{ and } \;\;\;\;
B(v_1,v_k) = B(v_k,v_k)=0
\;\;  \text{ for }  \;\;
k = 2,3 \, .
\end{equation}
\end{definition}

The next result, relating $B$-(pseudo-)orthonormality and bracket relations,  is a well-known but remarkable property of $\slR$, which is essentially due to the fact that it is a 3-dimensional Lie algebra. More precisely, we have the following.

\begin{lema}\label{lm:A}
Let $v=(v_k)$ be a basis of $\slR$. We have that $v=(v_k)$ is a $B$-orthonormal basis satisfying Condition~(\ref{def: B-ON})
if and only if
\begin{equation}\label{eq:Lrel1sl2R}
[v_1,v_2] = \delta v_3
\;\;\;\; , \;\;\;\;
[v_1,v_3] = \delta v_2
\;\;\;\; \text{ and } \;\;\;\;
[v_2,v_3] = -\delta v_1 \, ,
\end{equation}
for some $\delta \in \{-1,1\}$. Analogously, $v=(v_k)$ is a $B$-pseudo-orthonormal basis satisfying condition~(\ref{def: B-PON})
if and only if
\begin{equation}\label{eq:Lrel2sl2R}
[v_1,v_2] = \delta v_2
\;\;\;\; , \;\;\;\;
[v_1,v_3] = -\delta v_3
\;\;\;\; \text{ and } \;\;\;\;
[v_2,v_3] = \delta v_1 \, ,
\end{equation}
again for some $\delta \in \{-1,1\}$.
\end{lema}

\begin{proof} Let $v=(v_k)$ be a basis of $\slR$ and assume that their elements satisfy the bracket relations (\ref{eq:Lrel1sl2R}) for some $\delta\in\{-1,1\}$.  The
 linear maps $\ad_{v_k}$  can then be readily calculated and
since $B(x,y) = \frac{1}{2}\Tr(\ad_x\circ \ad_y)$, a simple computation allows us to check that $v=(v_k)$ is, indeed, a $B$-orthonormal basis satisfying condition~(\ref{def: B-ON}).

We now check that the converse holds.  To begin with, let us consider the particular basis $(e_k)$ of $\slR$, where
\begin{equation}\label{Eq: sl2-basis}
e_1=\lb \begin{matrix}  1/2 & 0 \cr 0 & -1/2 \end{matrix}\rb
\;\;\;\; , \;\;\;\;
e_2=\lb\begin{matrix}  0 & 1/2 \cr 1/2 & 0 \end{matrix}\rb
\;\;\;\; , \;\;\;\;
e_3=\lb\begin{matrix}  0 & 1/2 \cr -1/2 & 0 \end{matrix}\rb .
\end{equation}
The basis in question is clearly $B$-orthogonal and satisfies the bracket relations in (\ref{eq:Lrel1sl2R}) with $\delta =1$.
Furthermore, any other $B$-orthogonal basis $v=(v_k)$ is obtained from the present basis $(e_k)$ through the action of the orthogonal group of $B$
which is isomorphic to the classical group $\OO(2,1)$. Consider also the automorphism group of the Lie algebra $\slR$, the group $\Aut(\slR)$.

\bigbreak

\noindent {\it Claim:} $\Aut(\slR)$ is isomorphic to $\mathrm{SO}(2,1)$.

\begin{proof}[Proof of the Claim]
Take any basis $v=(v_k)$ satisfying Condition (\ref{eq:Lrel1sl2R}) and let $M$ be the matrix of $\varphi\in\mathrm{Aut}(\mathfrak{sl}(2,\mathbb{R}))$ with respect to this basis. The equation $[\varphi x,\varphi y]=\varphi[x,y]$ is easily seen, by direct computation, to be equivalent to
\begin{equation*}
c(M) = I_{2,1} M I_{2,1}
\end{equation*}
where $c(M)$ is the cofactor matrix of $M$ and $I_{2,1} = \Diag(1,1,-1)$. Then by taking transposes and multiplying
by $\det(M)$, we get that
\begin{equation*}
\det(M)M^{-1} = I_{2,1} M^t I_{2,1} \, .
\end{equation*}
This implies that $\det(M) = 1$ and the above equation is equivalent to having $M^t I_{2,1} M = I_{2,1}$. Thus $M\in\mathrm{SO}(2,1)$ and the claim is proved.
\end{proof}

Returning now to our main argument, we have that the bracket relations of a $B$-orthonormal basis of $\slR$ are parametrized
by the quotient space $\OO_B/\Aut(\slR)\simeq \OO(2,1)/\mathrm{SO}(2,1)$. This space is discrete and has two elements. Our first equivalence has thus been proved.

\bigbreak

In order to finish the proof of our lemma,  it remains to check that conditions (\ref{def: B-PON}) and (\ref{eq:Lrel2sl2R}) are equivalent.  The proof that (\ref{eq:Lrel2sl2R}) implies (\ref{def: B-PON}) is, again, a direct computation using the definition of $B$.  Conversely, to prove that (\ref{def: B-PON}) implies (\ref{eq:Lrel2sl2R}), take the change of basis
\begin{equation*}
u_1 = v_1
\;\;\;\; , \;\;\;\;
u_2 = \frac{v_2+v_3}{\sqrt{2}}
\;\;\;\; , \;\;\;\;
u_3 = \frac{v_2-v_3}{\sqrt{2}} \ .
\end{equation*}
and apply the result previously proved in this lemma.

\end{proof}

Recall the $B$-self-adjoint isomorphism $\Phi: \slR \longrightarrow \slR$ which defines the left-invariant metric $q$ via the identity $q(x,y)=B(\Phi x,y)$, for every $x, y \in \slR$.  Let $m_a(\lambda)$ and $m_g(\lambda)$ stand, respectively, for the algebraic and the
geometric multiplicity of an eigenvalue $\lambda$ of $\Phi$. By means of Jordan's normal form theorem, there exists a basis $v=(v_k)$
of $\RR^3$ where $\Phi$ is represented by one of the following matrices:

\begin{multicols}{2}
\begin{itemize}

\item[] {\bf Case 1:}  diagonal
$$
\lb\begin{matrix} \lambda_1 & 0 & 0 \cr 0 & \lambda_2 & 0 \cr 0 & 0 & \lambda_3 \end{matrix}\rb \phantom{\;\;\;\; ,
\zeta \neq 0}
$$

\item[] {\bf Case 3:} a real eigenvalue $\lambda$ such that $m_a(\lambda)-m_g(\lambda)=1$
$$
\lb\begin{matrix} \mu & 0 & 0 \cr 0 & \lambda & \zeta \cr 0 & 0 & \lambda \end{matrix}\rb
\;\;\;\; , \;\;\;\;
\zeta \neq 0
$$

\item[] {\bf Case 2:}  two complex  eigenvalues
$$
\lb\begin{matrix} \mu & 0 & 0 \cr 0 & \alpha & \beta \cr 0 & -\beta & \alpha \end{matrix}\rb
\;\;\;\; , \;\;\;\;
\beta \neq 0
$$

\item[] {\bf Case 4:}  a real eigenvalue $\lambda$ such that $m_a(\lambda)-m_g(\lambda)=2$
$$
\lb\begin{matrix} \lambda & 0 & \zeta \cr \zeta & \lambda & 0 \cr 0 & 0 & \lambda \end{matrix}\rb
\;\;\;\; , \;\;\;\;
\zeta \neq 0
$$
\end{itemize}
\end{multicols}

The fact that $\Phi$ is $B$-self-adjoint allows us to characterize the bases $v=(v_k)$ for each one of the above normal forms. This is
the content of the following lemma.

\begin{lema}\label{lm:B}
The basis in which the normal form of $\Phi$ is represented can be made:
\begin{itemize}
\item[(a)] $B$-orthonormal in cases 1 and 2;
\item[(b)] $B$-pseudo-orthonormal in case 3 and normalized so that $\zeta = q(v_3, v_3)$;
\item[(c)] $B$-pseudo-orthonormal in case 4 and normalized so that $\zeta>0$.
\end{itemize}
\end{lema}

\begin{proof}
We are going to consider separately the four different cases above.

\smallskip

\noindent {\it Case 1:} We start by assuming that the three eigenvalues of $\Phi$ are real and pairwise distinct. In this case, the $B$-self-adjointness of $\Phi$ immediately implies that $
B(v_k,v_l)=0$ whenever $k \ne l$.  It now
suffices to normalize the vectors $v_k$ so that we obtain a $B$-orthonormal basis.

Suppose now that there exists an eigenvalue with (algebraic) multiplicity $2$. Concretely, suppose that $\lambda_1 = \lambda_2$ and $\lambda_3 \neq \lambda_1$. A similar argument to the one applied above shows that $v_3$ is
orthogonal to $\Span(v_1,v_2)$, the space spanned by $v_1$ and $v_2$. Then $\Span(v_1,v_2)$ is either an Euclidean plane or a Lorentzian plane and we can use a Gram-Schmidt process to get a $B$-orthonormal basis of eigenvectors.
Finally, if there is an eigenvalue of (algebraic) multiplicity $3$, then $\Phi$ is a multiple of the identity and the result is immediate.

\smallskip

\noindent {\it Case 2:} Let $\mu$ be the real eigenvalue and $\Lambda,\overline{\Lambda}$ be the two complex conjugate eigenvalues
of $\Phi$ with $\Lambda=\alpha+i\beta$, $\beta \neq 0$. Consider $\slR+i\slR$  the complexification of $\slR$, and $H$ and $\Omega$ the complex
extensions of $B$ and $\Phi$, respectively. Note that $H$ is non-degenerate and $\Omega$ is $H$-self-adjoint.

Let $v_1, \, V, \, \overline{V}$ be (non-zero) eigenvectors associated to the eigenvalues $\mu,\Lambda,\overline{\Lambda}$ of $\Omega$,
respectively, where $V = v_2 + i v_3$ for some real vectors $v_2, v_3 \in \slR$.  The fact that $\Omega$ is $H$-self-adjoint implies
that the basis $(v_1,V, \overline{V})$ is $H$-orthogonal. This, in turn, implies that $B(v_1,v_2) = B(v_1,v_3) = 0$ and also that
$B(v_2,v_2) + B(v_3,v_3) = 0$. Furthermore, the non-degeneracy of $B$ implies that $B(v_1,v_1) \neq 0$ and, therefore, we must have
$B(v_1,v_1) > 0$, otherwise $\Span(v_2,v_3)$ would be an Euclidean plane meaning that $\Phi$ would be diagonalizable.
Finally, we use the non-degeneracy of $H$ to conclude that $H(V,V) \neq 0$ and we normalize $V$ so that $H(V,V) = 2$. This implies
that $B(v_2,v_2) - B(v_3,v_3) = 2$ and $B(v_2,v_3) =0$. Summing all up, we get that $v=(v_k)$ is a $B$-orthonormal basis.

\smallskip

\noindent{\it Case 3:} Assume that $\Phi$ has a (real) eigenvalue $\lambda$ such that $m_a(\lambda)-m_g(\lambda)=1$. Let then
$v=(v_k)$ be a basis for which $\Phi$ is represented by the corresponding normal form with $\zeta =1$.


From the fact that $\Phi$ is $B$-self-adjoint, we can conclude that $B(v_1,v_2)=0 = B(v_2,v_2)$. Since $B$ is non-degenerate we get that $B(v_2,v_3)\neq 0$ and also that $B(v_1,v_1)\neq 0$, otherwise the restriction of $B$ to $\mathrm{Span}(v_1,v_2)$ would vanish identically. If $M$ is the matrix associated to $B$, the fact that $\mathrm{det}(M) < 0$ implies that $B(v_1,v_1)>0$. Taking
\begin{equation*}
u_1  = \frac{v_1}{\sqrt{B(v_1,v_1)}}-\frac{B(v_1,v_3)}{B(v_2,v_3)\sqrt{B(v_1,v_1)}}v_2 \,\,\,\,\, ,\,\,\,\,\,
u_2 = \frac{v_2}{B(v_2,v_3)} \,\,\,\,\, , \,\,\,\,\,
u_3  = v_3 - \frac{B(v_3,v_3)}{2B(v_2,v_3)} \, v_2 \, .
\end{equation*}

It can be easily checked that $u=(u_k)$ is a $B$-pseudo-orthornormal basis.  If the eigenvalues $\lambda$ and $\mu$ are such that $\lambda \neq \mu$, it can be additionaly proved that $B(v_1,v_3)=0$. In either case, we can check that $\Phi$ maintains its normal form but now with
$\zeta  = B(v_{2},v_{3}) \lb =B(\Phi u_3, u_3) = q(u_3, u_3)\rb$.

\bigbreak

\noindent{\it Case 4:} Assume that $\Phi$ has a (real) eigenvalue $\lambda$ such that $m_a(\lambda) - m_g(\lambda) = 2$ and
let $v=(v_k)$ be a basis for which $\Phi$ is represented by the corresponding normal form with $\zeta = 1$.

Since $\Phi$ is $B$-self-adjoint, there follows that $B(v_1,v_2) = B(v_2,v_2) = 0$ and $B(v_1,v_1) = B(v_2,v_3)$. In turn, the
non-degeneracy of $B$ implies $B(v_1,v_1) = B(v_2,v_3) \neq 0$. Consider then the basis given by the vectors
\begin{equation*}
u_1  = v_1 + K v_2 \, ,  \quad
u_2 = v_2 \, ,  \quad
u_3 = v_3 + K v_1 + L v_2
\end{equation*}
where
\[
K = - \frac{B(v_1,v_3)}{2 B(v_1,v_1)} \qquad \mbox{ and } \qquad L= - \frac{B(v_3,v_3)+2 K B(v_1,v_3) + K^2B(v_1,v_1)}{2B(v_2,v_3)} \, .
\]
By what precedes, it can be checked that $B(u_1,u_2) = B(u_1, u_3) = B(u_2, u_2) = B(u_3,u_3) =0$. We can deduce, in addition, that
$B(u_1,u_1) = B(u_2, u_3) \neq 0$. Furthermore, we must have $B(u_1,u_1)>0$, otherwise $\Span(u_2,u_3)$ would be an Euclidean
plane and therefore could not contain null vectors. Now, by letting
\begin{equation*}
w_1 = \frac{u_1}{\sqrt{B(u_1,u_1)}} \, , \quad
w_2 = \frac{u_2}{B(u_2,u_3)} \, , \quad
w_3 = u_3.
\end{equation*}
we obtain a $B$-pseudo-orthornormal basis,$(w_k)$, with respect to which the isomorphism $\Phi$ maintains its normal form
but now with $\zeta = \sqrt{B(u_1,u_1)} = \sqrt{B(v_1,v_1)}$.
\end{proof}

Equipped with these hand-on results about the structure of $\slR$, we will reobtain,  in Sections \ref{Sec:Case1},\ref{Sec:Case3}, \ref{sec:Case2and4} and by different methods, the well-known result of Bromberg and Medina.

\begin{teo}[\cite{B-M}]\label{Thm:B-M}
The completeness of a left-invariant pseudo-Riemannian metric on $\SLR$ is characterized in terms of its associated isomorphism $\Phi$ as follows.
\begin{itemize}
\item[-] In case 1, the metric $q$ is complete if and only if $\left(\frac{1}{\lambda_2}-\frac{1}{\lambda_3}\right)\left(\frac{1}{\lambda_3}-\frac{1}{\lambda_1}\right)\leq 0$.
\item[-] In case 2, the metric $q$ is incomplete.
\item[-] In case 3, the metric $q$ is complete if and only if $\left(\frac{1}{\mu}-\frac{1}{\lambda}\right)\zeta \leq 0$.
\item[-] In case 4, the metric $q$ is incomplete.
\end{itemize}

\end{teo}

Furthermore, we will present a detailed study of completeness or incompleteness of every single geodesic of the metrics in question while, in addition and in the case of an incomplete geodesic, we estimate the size of its maximal domain of definition.

\smallskip

From the representation theory of $\SLR$, it can be seen that every non-compact semisimple Lie group contains a Lie subgroup that is isomorphic to $\SLR$.  This well-known fact and the previous theorem can be combined to prove the following.

\begin{coro}\label{cor:incomplete-pr-metrics}
Let $G$ be a non-compact semisimple Lie group. Then $G$ can be endowed with incomplete left-invariant pseudo-Riemannian metrics.
\end{coro}

\begin{proof}
Let $\mathfrak{g}$ be the Lie algebra of $G$. Consider $\mathfrak{h}$ a copy of $\mathfrak{sl}(2,\mathbb{R})$ inside $\mathfrak{g}$ and write
$\mathfrak{g} = \mathfrak{h} \oplus V$.
Take a left-invariant pseudo-Riemannian metric $q$ with associated isomorphism $\Phi$ such that $\Phi(\mathfrak{h}) \subseteq \mathfrak{h}$ and the restriction of  $q$ to $\mathfrak{h}$ is one of the incomplete metrics of Theorem \ref{Thm:B-M}. By construction, the geodesic flow leaves the Lie subalgebra $\mathfrak{h}$ invariant and we have incomplete integral curves of the Euler-Arnold vector field contained in $\mathfrak{h}$ and therefore in $\mathfrak{g}$.
\end{proof}

\medskip

\section{Characterization of geodesics in case 1}\label{Sec:Case1}

Let us start by assuming that the isomorphism $\Phi$ is diagonalizable. It was proved in Lemma~\ref{lm:B} that there exists
a $B$-orthonormal basis $v = (v_k)$ with respect to which $\Phi$ is a diagonal matrix $\Phi = {\rm diag} \, (\lambda_1, \lambda_2,
\lambda_3)$. In this case $\Phi^{-1}$ is a diagonal matrix as well. More precisely, $\Phi^{-1} = {\rm diag} \, (\nu_1, \nu_2,
\nu_3)$, where $\nu_i = 1/\lambda_i$. Fix an element $z \in \slR$ and let $(z_1,z_2,z_3)$ stand for its coordinates with respect
to the basis $v = (v_k)$. Recall that the geodesics associated to $q$ on ${\rm SL} \, (2,\RR)$ are in one-to-one correspondence
with the integral curves on $\slR$ of the differential system $\dot{z} = [z,\Phi^{-1}z]$, $z \in \slR$. Being $v = (v_k)$ a
$B$-orthonormal basis, up to reordering its elements we can assume that $B(v_1,v_1) = B(v_2,v_2) = -B(v_3,v_3) = 1$ and $B(v_i,v_j)
= 0$ whenever $i \ne j$. From Lemma~\ref{lm:A} there follows that the differential system above can be written, in the affine
coordinates $(z_1, z_2, z_3) \in \RR^3$, as
\begin{align}\label{sist_EA}
\left(\begin{matrix}\dot{z}_{1} \cr \dot{z}_{2} \cr \dot{z}_{3}\end{matrix}\right)=
\left(\begin{matrix}
\delta (\nu_{2} - \nu_{3}) z_{2} z_{3} \cr
\delta (\nu_{3} - \nu_{1}) z_{1} z_{3} \cr
\delta (\nu_{2} - \nu_{1}) z_{1} z_{2}
\end{matrix}\right) \, ,
\end{align}
for $\delta \in \{-1,1\}$. Naturally, it suffices to consider the case where $\delta = 1$, since if $\varphi = \varphi(t)$ is a
solution of the system for $\delta = 1$, then $\psi = \psi(t) = \varphi(-t)$ is a solution for $\delta = -1$. In particular, one
system is complete if and only if so is the other. So, let us assume in what follows that $\delta = 1$. By letting $a = \nu_{2} -
\nu_{3}$, $b = \nu_{3} - \nu_{1}$ and $c = \nu_{2} - \nu_{1} = a + b$, the system of differential equations~(\ref{sist_EA}) is
in natural correspondence with the vector field
\begin{equation}\label{eq:EAVF1}
E = a z_{2}z_{3} \dd{z_{1}} + b z_{1}z_{3} \dd{z_{2}} + c z_{1}z_{2} \dd{z_{3}} \, ,
\end{equation}
which will be called Lax vector field or, by abuse of language, Euler-Arnold vector field.


\subsection{Characterization of geodesically complete metrics}

The characterization of the geodesically complete left-invariant pseudo-Riemannian metrics $q$ in this case can be easily described
in terms of the eigenvalues of $\Phi$ or, equivalently, $\Phi^{-1}$. More precisely, we have the following.

\begin{teo}\label{thm:slR(S1)}
The left-invariant pseudo-Riemannian metric $q$ is geodesically complete if and only if $ab \leq 0$.
\end{teo}

To prove Theorem~\ref{thm:slR(S1)} we may consider separately the cases where $ab > 0$ and $ab \leq 0$. With respect to the case
where $ab > 0$ we can prove the existence of the so-called idempotents (cf. \cite{B-M}), which correspond to straight lines through
the origin that are invariant by the foliation induced by $E$ and such that the restriction of $E$ to them does not vanish identically.
The restriction of $E$ to an idempotent is a quadratic (homogeneous) vector field and, consequently, the solution of the associated
differential equation is not complete. In the case where $ab \leq 0$ we will prove that the integral curves of $E$ are contained in
a compact part of $\R^3$. The solutions of the differential system associated to $E$ must then be complete. Let us begin with the
case $ab > 0$.

\begin{proposition}\label{prp:N1}
If $ab > 0$ then there exists at least one geodesic that is incomplete.
\end{proposition}

\begin{proof}
Recall that the geodesics are the integral curves of the system of differential equations associated with the Euler-Arnold vector
field $E$. Since $E$ is a polynomial vector field on $\RR^{3}$, it then follows that $E$ admits a rational extension to $\RR\PP(3)$,
the compactification of $\RR^{3}$ by adjuntion of the plane at infinity $\Delta_{\infty} = \RR\PP(3) \setminus \RR^3$. The singular
foliation $\fol$ induced by this extension on $\RR\PP(3)$ is however analytic since, locally, the foliation is still represented by a
polynomial vector field. Furthermore, $\Delta_{\infty}$ is invariant for the foliation in question since $E$ is a homogeneous vector
field distinct from a multiple of the radial vector field.

Let us consider affine coordinates $(x_1,x_2,x_3)$ nearby $\Delta_{\infty}$. More precisely, consider the affine coordinates
$(x_1,x_2,x_3)$ related to $(z_1,z_2,z_3)$ through the map $\Psi$ on $\RR^3 \setminus \{x_3 = 0\}$ defined by
\begin{equation}\label{eq:x-chart}
\Psi(x_{1},x_{2},x_{3}) = \left( \frac{x_{1}}{x_{3}}, \frac{x_{2}}{x_{3}}, \frac{1}{x_{3}} \right) = (z_{1},z_{2},z_{3})
\end{equation}
Since the vector field $E$ will be represented in several different coordinates, it is convenient to denote some of these representations
by a different letter so as to avoid misunderstanding. In particular, let $X$ denote the vector field $E$ in the coordinates $(x_{1},x_{2},
x_{3})$ so that
\begin{equation}\label{eq:X_VF}
X = \frac{1}{x_{3}} \left[ x_{2}(a-cx_{1}^{2}) \dd{x_{1}} + x_{1}(b-cx_{2}^{2}) \dd{x_{2}} - cx_{1}x_{2}x_{3} \dd{x_{3}} \right] \, .
\end{equation}
Recalling that $c = a+b$,
the fact that $ab > 0$ ensures that all of $a$, $b$ and $c$ have the same sign. Therefore, it can easily be checked that the intersection
of the singular set of $\fol$ with the plane at infinity, in the considered affine chart, is constituted by $5$ singular points, namely
\[
(0,0,0), \, \, \, \left(\sqrt{a/c},\sqrt{b/c},0\right), \, \, \, \left(\sqrt{a/c},-\sqrt{b/c},0\right), \, \, \,
\left(\sqrt{a/c},-\sqrt{b/c},0\right) \, \, \, \left(-\sqrt{a/c},-\sqrt{b/c},0\right) \, .
\]
The straight line ``above'' the singular point $(0,0,0)$, i.e. the line given by $\{x_1 = x_2 = 0\}$, is contained in the singular set of
$X$. The same, however, does not occur with respect to the lines above the remaining singular points. The restriction of $X$ to the straight
line ``above'' $(\sqrt{a/c}, \sqrt{b/c}, 0)$, i.e. to the line $L = \{x_1 = \sqrt{a/c}, \, x_2 = \sqrt{b/c} \}$, is the constant vector field
$X_{L} = -c \sqrt{a/c} \sqrt{b/c} \, \partial /\partial{x_3}$. In particular, $X_L$ is regular at the origin, which means that $L$ ``crosses''
the plane at infinity, given by $\{x_3 = 0\}$, in finite
time. The corresponding geodesic is therefore incomplete and the result follows.
\end{proof}

Consider now the case $ab \leq 0$ where the following can be proved.

\begin{proposition}\label{prp:S1}
If $ab \leq 0$ then all geodesics are complete.
\end{proposition}

\begin{proof}
Consider first the case where $ab = 0$. It can then be assumed without loss of generality that $a = 0$, which is equivalent to saying
that the eigenvalues $\nu_2$ and $\nu_3$ coincide. The first equation of the differential system~(\ref{sist_EA}) reduces then to
$\dot{z}_1 = 0$, which means that $z_1(t) = k$, with $k \in \RR$, for all $t \in \RR$. Thus, the Euler-Arnold differential system
reduces to a linear system in the variables $z_2, \, z_3$, namely
\begin{align*}
\left(\begin{matrix}\dot{z}_{2} \cr \dot{z}_{3}\end{matrix}\right)=
\left(\begin{matrix}
bk z_{3}\cr
ck z_{2}
\end{matrix}\right) \, ,
\end{align*}
which is clearly complete.

Assume now that $ab < 0$. This is equivalent to having $\nu_3 < {\rm min} \{\nu_1, \nu_2\}$ or $\nu_3 > {\rm max} \{\nu_1,\nu_2\}$.
Next, recall that $I_1(z) = B(z,z)$ and $I_2(z) = B(z, \Phi^{-1}z)$ are first integrals for the Lax-pair equations. In particular, so
is $\lambda I_1 + \beta I_2$ for every $\lambda, \, \beta \in \RR$. By using the fact that $v = (v_k)$ is a $B$-orthonormal basis, we
get that
\[
I_1(z_1,z_2,z_3) = z_1^2 + z_2^2 - z_3^2 \qquad \text{and} \qquad I_2(z_1,z_2,z_3) = \nu_1 z_1^2 + \nu_2 z_2^2 - \nu_3 z_3^2 \, .
\]
Assume that $\nu_3 < {\rm min} \{\nu_1, \nu_2\}$ (resp. $\nu_3 > {\rm max} \{\nu_1,\nu_2\}$). Then, we have that, for every $\beta > 0$
(resp. $\beta < 0$), $\lambda + \beta \nu_3$ is smaller than both $\lambda + \beta \nu_1$ and $\lambda + \beta \nu_2$. In particular,
$\lambda$ and $\beta$ can be chosen so that $\lambda + \beta \nu_3$ is negative, while $\lambda + \beta \nu_1$ and $\lambda + \beta
\nu_2$ are both positive. For such parameters $\lambda$ and $\beta$, we have that $\lambda I_1 + \beta I_2$ is a quadratic positive
definite first integral, which implies that all geodesics are contained in a compact part of $\RR^3$. The geodesics are therefore
complete and the result follows in that case.
\end{proof}

Note that the existence of a quadratic positive definite first integral does not prevent the existence of invariant straight lines
through the origin for $\fol$. Nonetheless, if such invariant lines exist, they must be contained in the singular set of $X$ so that
every single point of the straight line in question is a geodesic. A left-invariant pseudo-Rimannian metric as above is then complete
if and only if the Euler-Arnold system does not admit idempotents (cf. \cite{B-M}).

It is known since \cite{B-M} that the metrics associated to diagonalizable isomorphisms are geodesically incomplete in the case where $ab
> 0$ and the proof goes as above, i.e. they proved that a particular geodesic is incomplete. The completeness of the remaining leaves has
not been discussed and its study is the content of the rest of this section. More precisely, in what follows we will determine the maximal
domain of definition of every single geodesic for these incomplete metrics. The approach that we will follow is based on a method introduced
in~\cite{RR_Applications} to investigate the maximal domain of definition in $\C$ of the solutions of a complex differential equation. The
method goes as follows.

To begin with, recall that the geodesics are in correspondence with the leaves of the foliation $\fol$ associated with the Euler-Arnold
vector field $E$. Since $E$ is polynomial, it admits a rational extension to $\RR\PP(3)$. Furthermore, being  $E$ a homogeneous vector
field distinct from a multiple of the radial vector field, the plane at infinity is invariant by $\fol$. The homogeneity of $E$ ensures
the existence of invariant cones for $\fol$ through the origin, and each one of these cones cuts $\Delta_{\infty}$ along a curve that
is a leaf for the foliation induced by the extension of $E$ to $\Delta_{\infty}$. We start by describing the foliation induced by this
extension to $\Delta_{\infty}$. Next, we study the foliation over the invariant ``cones'' above every single leaf of the foliation on
$\Delta_{\infty}$. More precisely, for every single leaf on a fixed cone, we will study the ``height'' function of the leaf with respect
to the plane at infinity. If the ``height'' function is bounded from below by a strictly positive constant, then the leaf remains in a
compact part of $\RR^3$ and, then, the corresponding geodesic is naturally complete. If the ``height'' function can be made arbitrarily
close to zero (i.e. the associated geodesic escapes from every compact subset of $\RR^3$) we will decide on the completeness or incompleteness
of the geodesic by taking the integral of the associated time-form along the leaf. The time-form induced by $E$ on a leaf $L$ is the $1$-form
$dT$ satisfying $dT.X|_L \equiv 1$. Furthermore, if $p, q$ are two points on a same leaf and $c$ is the path joining these two points, then
$\int_c dT$ measures the time to go from $p$ to $q$. We will then investigate the convergence of the integral of the time-form over the entire
leaf to decide on its completeness (cf. \cite{RR_Applications} for more details).


\subsection{Study of geodesics for $ab > 0$}

In this subsection we will determine the maximal domain of definition of every single geodesic, in the case $ab > 0$, by following the approach
previously described. As previously mentioned, we may start by describing the foliation induced by $E$ on $\Delta_{\infty}$. Let us then
fix one of the charts nearby $\Delta_{\infty}$, namely let us consider affine coordinates $(x_1,x_2,x_3)$ related to $(z_1,z_2,z_3)$ through the
map $\Psi$ on $\RR^3 \setminus \{x_3 = 0\}$ defined by~(\ref{eq:x-chart}). It has already been mentioned that the vector field $E$ in the coordinates
$(x_{1},x_{2},x_{3})$ is given by~(\ref{eq:X_VF})
\[
X = \frac{1}{x_{3}} \left[ x_{2}(a-cx_{1}^{2}) \dd{x_{1}} + x_{1}(b-cx_{2}^{2}) \dd{x_{2}} - cx_{1}x_{2}x_{3} \dd{x_{3}} \right] \, .
\]
A representative vector field for $\fol_{\infty}$, the foliation induced by $\fol$ on $\Delta_{\infty}$, is naturally given by
\begin{align}\label{eq:X_infty}
X_{\infty} = x_{2}(a-cx_{1}^{2}) \dd{x_{1}} + x_{1}(b-cx_{2}^{2}) \dd{x_{2}} \, .
\end{align}
We claim that $X_{\infty}$ admits a non-constant first integral. In fact, although the first integrals $I_1$ and $I_2$ for $\fol$, given in
the affine coordinates $(x_1,x_2,x_3)$ of $\RR^3$ by $I_{1}(x_1,x_2,x_3) = (x_{1}^{2}+x_{2}^{2}-1)/x_{3}^{2}$ and $I_{2}(x_1,x_2,x_3) = (\nu_{1}x_{1}^{2}+\nu_{2}x_{2}^{2}-\nu_{3})/x_{3}^{2}$, are not defined over $\Delta_{\infty}$, their quotient is and it naturally corresponds
to a first integral for $\fol_{\infty}$
\begin{equation}\label{eq:first_integral_I_1}
I(x_1,x_2,x_3) = \frac{x_{1}^{2}+x_{2}^{2}-1}{\nu_{1}x_{1}^{2}+\nu_{2}x_{2}^{2}-\nu_{3}} \, .
\end{equation}
The leaves of $\fol_{\infty}$ are then contained in the level sets of $I$, which naturally satisfy
\begin{equation}\label{eq:Q1Int1}
(1 - K\nu_{1}) x_{1}^{2} + (1 - K\nu_{2}) x_{2}^{2} = 1 - K \nu_{3} \, ,
\end{equation}
for $K \in \R$. Each leaf is then contained in a certain conic and the type of the conic in the present coordinates depend on the
value of $K$. To completely describe the leaves of $\fol_{\infty}$ on $\Delta_{\infty} \simeq \RR\PP(2)$, two more charts should
be considered. So, let $(y_1, y_2)$ and $(u_1,u_2)$ be affine coordinates related with $(x_1,x_2)$ as follows
\begin{align*}
\psi (y_1,y_2)& = \left( \frac{y_1}{y_2}, \frac{1}{y_2} \right) = (x_1, x_2)\\
\phi (u_1,u_2)& = \left( \frac{1}{u_1}, \frac{u_2}{u_1} \right) = (x_1, x_2)
\end{align*}
A vector field representing $\fol_{\infty}$ in the affine coordinates $(y_1,y_2)$ is then given by
\begin{align*}
Y_{\infty} = y_{2}(a-by_{1}^{2}) \dd{y_{1}} + y_{1}(c-by_{2}^{2}) \dd{y_{2}}
\end{align*}
while, in the affine coordinates $(u_1,u_2)$ is given by
\begin{align*}
U_{\infty} = u_{2}(c-au_{1}^{2}) \dd{u_{1}} + u_{1}(b-au_{2}^{2}) \dd{u_{2}} \, .
\end{align*}
The foliation $\fol_{\infty}$ induced by the Euler-Arnold vector field on $\Delta_{\infty}$ has then a total of 7 singular points,
namely $p_1, \, p_2, \, p_3$ and $p_4$ given, respectively, in the affine coordinates $(x_1,x_2,x_3)$ by
\[
\left(\sqrt{a/c},\sqrt{b/c},0\right), \, \, \, \left(-\sqrt{a/c},\sqrt{b/c},0\right), \, \, \,
\left(-\sqrt{a/c},-\sqrt{b/c},0\right) \, \, \, \left(\sqrt{a/c},-\sqrt{b/c},0\right) \, .
\]
and $q_x, \, q_y$ and $q_u$ standing, respectively, for the origin of the affine coordinates $(x_1,x_2)$, $(y_1,y_2)$ and $(u_1,u_2)$.
The dynamics of $\fol_{\infty}$ on each one of the affine coordinates is summarized in the picture below.

$$\begin{array}{ccccc}
\includegraphics[width=0.3\textwidth]{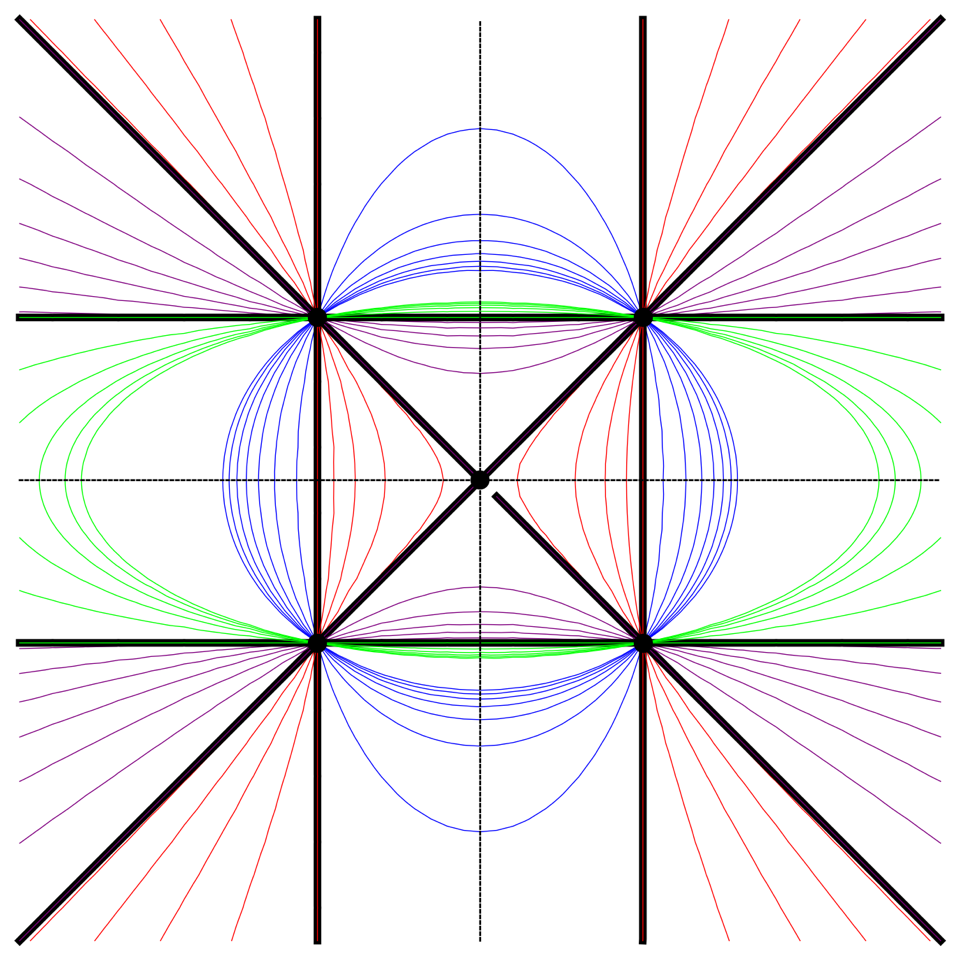} & \; \; &
\includegraphics[width=0.3\textwidth]{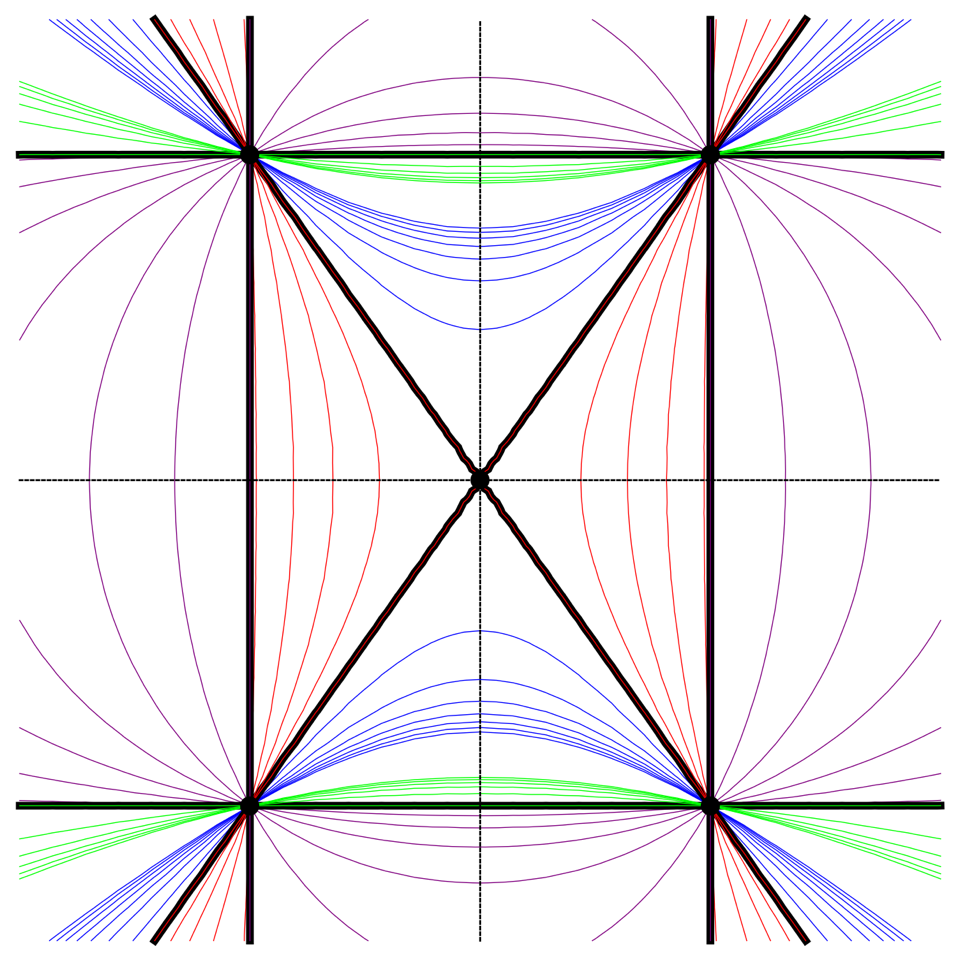} & \; \; &
\includegraphics[width=0.3\textwidth]{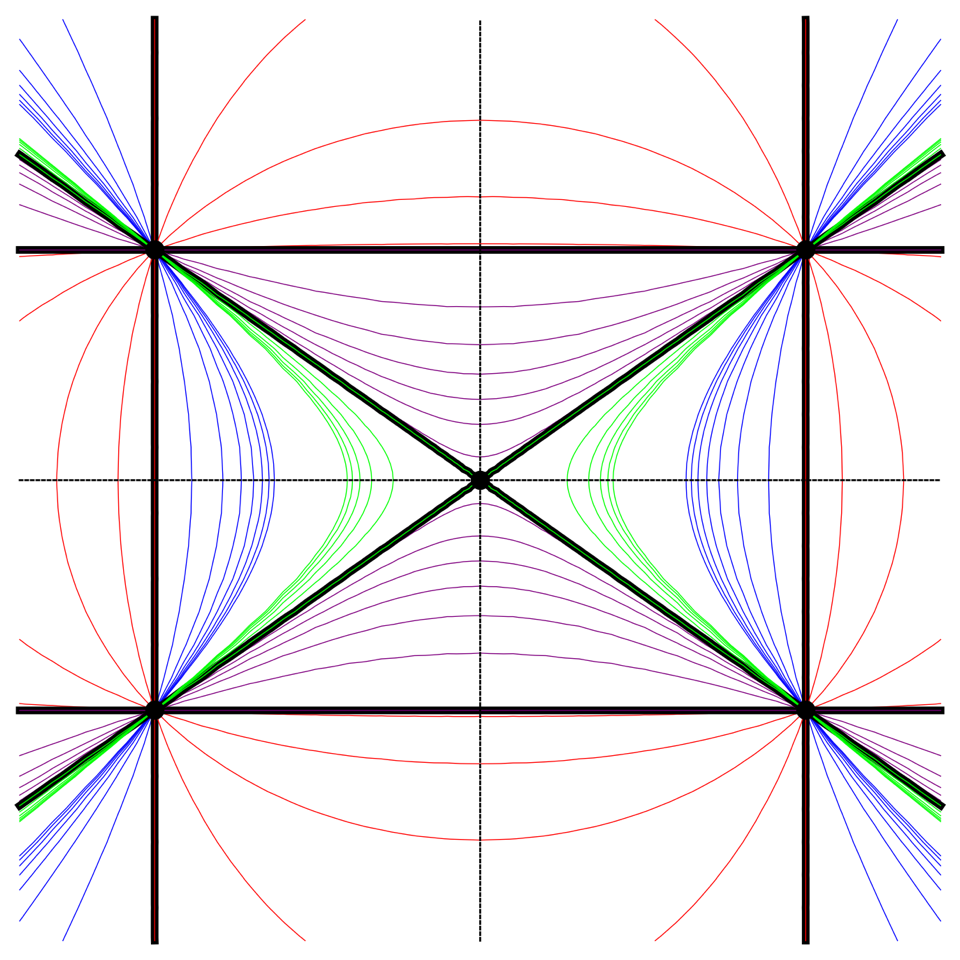}  \cr
 \text{Dynamics on $(x_1,x_2)$} & \; \; &
 \text{Dynamics on $(y_1,y_2)$} & \; \; &
 \text{Dynamics on $(u_1,u_2)$}
\end{array}$$

Recall just that the ``cone above'' every leaf of $\fol_{\infty}$ is invariant by $\fol$. The dynamics of the foliation over
these invariant cones play a role in the proof of the proposition below, where the maximal domain of definition of every single
geodesic is given.

\begin{proposition}\label{prop:description_complete_geodesics_case1}
Let $L$ be a leaf of $\fol$ and $L_{\infty}$ its projection on $\Delta_{\infty}$. The geodesic associated with $L$ is complete
if its projection coincides with one of the singular points $q_x, \, q_y$ or $q_u$ and is incomplete in all the other cases. An
incomplete geodesic is, in fact, $\R^+$ or $\R^-$ complete if and only if $L_{\infty}$ joins a singular point $p_i$ to a singular
point $q_j$ or if it reduces to one of the points $p_i$.
\end{proposition}

\begin{proof}
Recall that the leaves of $\fol_{\infty}$ are contained in the level sets of the function $I$ on~(\ref{eq:first_integral_I_1})
and, consequently, their points satisfy Equation~(\ref{eq:Q1Int1}) for some $K$, which means that they are contained in a certain
conic in the affine coordinates $(x_1,x_2,x_3)$.

Let us first discuss the case where the projection of a leaf $L$ of $\fol$ coincides with one of the singular points of $\fol_{\infty}$.
Note that the argument to prove that the straight line ``above'' $p_1$ is an incomplete geodesic can also be applied to the straight lines
``above'' the singular points $p_2, \, p_3$ and $p_4$. The restriction of the Euler-Arnold vector field to the mentioned leaves is a (non-zero)
constant multiple of the vector field $x_3^2 \, \partial /\partial x_3$ so that the geodesics are, in fact, $\R^+$ or $\R^-$ complete. In turn,
the straight line ``above'' the singular point $q_x$, given by $l \simeq \{x_1 = x_2 = 0\}$, is contained in the singular set of $\fol$. The
restriction of $X$ to $l$ vanishes identically and, therefore, the geodesic passing through every single point of $l$ reduces to the point
itself and hence it is complete. The same holds for the straight line ``above'' the singular points $q_y$ and $q_u$.

Next, note that every single leaf of $\fol_{\infty}$ goes from one singular point of $\fol_{\infty}$ to another singular point of
$\fol_{\infty}$. Furthermore, since the vector field is homogeneous, $\fol$ has no singular points away from the above mentioned
straight lines. Thus, the only chance for a leaf $L$ of $\fol$ on $\RR^3$ (i.e. a geodesic) to approach the infinity, is when its
projection on $\Delta_{\infty}$ approaches one of the $7$ singular points. Therefore, in order to study the completeness of the
remaining geodesics, it is sufficient to study the domain of definition of the solutions nearby each one of the $7$ singular points.
In fact, it will be sufficient to consider the behavior of the solutions nearby the singular points $p_1$ and $q_x$ since the pull-back
of $X$ by rotation of angle $\pi$ on the variables $(x_1,x_2)$ or by reflection on the coordinate axes of affine coordinates $(x_1,x_2)$
coincides with $X$ or with $-X$, respectively. Furthermore, the expressions of $X_{\infty}, \, Y_{\infty}$ and $U_{\infty}$ nearby the
origin ensures as well that the behavior nearby $q_x$ is similar to the behavior nearby $q_y$ and $q_u$. So, let us fix once and for all
the affine coordinates $(x_1,x_2,x_3)$ nearby $\Delta_{\infty}$ and let us study the behavior and the domain of definition of the geodesics
accumulating at $p_1$ and/or $q_x$.

With respect to $p_x$, there are just four leaves of $\fol_{\infty}$ accumulating at the point in question, namely the straight
lines joining $q_x$ to each one of the singular points $p_1, \, p_2, \, p_3$ and $p_4$. In turn, from the symmetries previously
described, it suffices to consider the leaf $L_{\infty}$ joining $q_x$ to $p_1$. This leaf is parameterized by $H_0(x_1) = (x_1,
\sqrt{b/a} \, x_1, 0)$, $x_1 \in ]0, \sqrt{a/c}[$, so that a leaf $L$ on the cone above $L_{\infty}$ is parameterized by $H(x_1)
= (x_1, \sqrt{b/a} \, x_1, x_3(x_1))$ where $x_3$ satisfies
\[
\frac{dx_{3}}{dx_1} = \frac{-c x_1}{a - c x_1^2} \, x_{3} \qquad \qquad \text{so that} \qquad \qquad
x_{3}(x_1) = x_{3}^{0} \sqrt{\frac{a - c x_1^2}{a - c (x_1^0)^2} }
\]
Thus, for $L$ distinct from $L_{\infty}$, the absolute value of $x_3 = x_3(x_1)$ is bounded from below by a strictly positive
constant nearby $x_1 = 0$. Thus, the associated geodesic remains away from $\Delta_{\infty}$ (it goes to a singular point of
$\fol$ on the straight line ``above'' $q_x$) so that it is at least $\R^+$ or $\R^-$ complete, depending on weather the geodesic
is approaching or moving away from the singular point as $t$ increases. To prove that the geodesic is, in fact, incomplete we go
as follows. First, note that $x_3 = x_3(x_1)$ goes to zero when $x_1$ goes to the singular point $\sqrt{a}{c}$, which means that
$L$ approaches $\Delta_{\infty}$. Next, consider the time-form associated with $X$ along $L$, which is given by
\[
dT = \frac{x_3(x_1)}{x_2(x_1) (a-cx_1^2)} \, dx_1 = \frac{C}{x_1 \sqrt{|a-cx_1^2|}} \, dx_1 \, ,
\]
for some constant $C \in \R^{\ast}$. The time we take to go from a point $(x_1^0, x_2^0, x_3^0)$ to $p_1$ along $L$ is given by
$\int_{x_1^0}^{\sqrt{a/c}} \, dT$ and so, to decide wether the geodesic reaches $\Delta_{\infty}$ in finite time or not, we just
need to decide on the convergence or divergence of the mentioned integral. Since $a-cx_1^2$ may be written as $a(1-\sqrt{c/a} \,
x_1) (1+ \sqrt{c/a} \, x_1)$ and since $\int_{x_1^0}^{\sqrt{a/c}} \, \frac{1}{\sqrt{1-\sqrt{c/a} \, x_1}} \, dx_1$ converges,
there follows that $\int_{x_1^0}^{\sqrt{a/c}} \, dT$ converges as well and the corresponding geodesic is incomplete.

\smallbreak

From now on let us focus on the foliation nearby $p_1$. We claim first that the geodesics associated with the leaves in the invariant
plane $\{x_2 = \sqrt{b/a} \, x_1\}$ are $\R^+$ or $\R^-$ complete. In fact, this is the case for those where $x_1 \in ]0,\sqrt{a/c}[$.
As for the remaining leaves, the claim follows from similar calculations along with the previously mentioned symmetry arguments. Consider
then the invariant plane $\{x_2 = \sqrt{b/c}\}$ and assume first that $L_{\infty}$ is such that $-\sqrt{a/c} < x_1 < \sqrt{a/c}$. If $L$
is a leaf of $\fol$ away from $\Delta_{\infty}$ and projecting on $L_{\infty}$, then the variable $x_3$ satisfies
\[
\frac{dx_3}{dx_1} = - \frac{cx_1}{a-cx_1^2} \, x_{3} \qquad \qquad \text{so that} \qquad \qquad
x_3(x_1) = x_3^0 \sqrt{\frac{a - cx_1^2}{a - c(x_1^0)^2}} \ .\
\]
We have that $x_3$ goes to zero as $x_1$ goes to $\pm \sqrt{a/c}$, which means that $L$ approaches $\Delta_{\infty}$. Let us then check
that the associated geodesic reaches the infinity in finite time, so that it is incomplete. The time-form associated with $X$ along $L$
is given by
\[
dT = \frac{x_3(x_1)}{x_2(x_1)(a-cx_1^2)} \, dx_1 = \frac{C}{\sqrt{|a-cx_1^2|}} \, dx_1 \, ,
\]
for some constant $C \in \R^{\ast}$. Again, the time that we take to go from the point $(x_1^0, x_2^0, x_3^0)$ to $p_1$ along $L$ is
given by $\int_{x_1^0}^{\sqrt{a/c}} \, dT$. Again, this integral converges if and only if so does $\int_{x_1^0}^{\sqrt{a/c}} \,
\frac{1}{\sqrt{1-\sqrt{c/a} \, x_1}} \, dx_1$ and, since the later is clearly convergent, the associated geodesic is incomplete.
Taking into account the previously described symmetries, the geodesic is neither $\R^+$ nor $\R^-$ complete. In other words, the
maximal interval of definition is a bounded open interval.

The study of the remaining leaves on the invariant plane $\{x_2 = \sqrt{b/c}\}$ is similar, with the exception of the fact that their
projection on $\Delta_{\infty}$ escapes from every compact subset of the affine chart $(x_1,x_2)$. If $L$ is a leaf on this plane with
$x_1 > \sqrt{a/c}$, then the argument above allows us to say that $L$ reaches $\Delta_{\infty}$ in finite time when its projection
approaches the singular point $p_1$. In turn, to discuss $\R^+$ or $\R^-$ completeness of the leaf, either we consider the affine
charts $(w_1,w_2,w_3)$ related with $(z_1,z_2,z_3)$ through the map
\begin{equation}\label{eqn:chart-w}
\Lambda(w_1,w_2,w_3) = \left( \frac{1}{w_1}, \frac{w_2}{w_1}, \frac{w_3}{w_1} \right) = (z_1,z_2,z_3) \, ,
\end{equation}
and where the divisor at infinity is represented by $\{w_1 = 0\}$, or we study directly the time-form in the present coordinates.
Following the second approach, we may note that
\[
\int_{x_1^0}^{+\infty} \, dT = \int_{x_1^0}^{+\infty} \frac{C}{\sqrt{|a - cx_1^2|}} \, dx_1 = \int_0^{1/x_1^0} \frac{C}{s\sqrt{|as^2 - c|}} \, ds
\]
clearly diverges since $\sqrt{|as^2 - c|}$ is bounded in the considered interval. The leaf $L$ is then $\R^+$ or $\R^-$ complete.
The study of the leaves on the invariant plane $\{x_1 = \sqrt{a/c}\}$ is analogous.

Assume now that $L$ is a leaf whose projection $L_{\infty}$ is contained in an ellipse and accumulates at $p_1$. The leaf $L_{\infty}$
can thus be parameterized by $H(\theta) = (\alpha_1 \cos \theta, \alpha_2 \sin \theta)$, where $\alpha_i = \sqrt{\frac{1-K\nu_3}{1-K\nu_i}}$,
i=1,2, for some $K$ such that all $1-K\nu_i$, $i=1,2,3$, have the same sign.
Furthermore, we can assume that $\theta$ belongs to $]\theta_1,\theta_2[$, where $\theta_i$ is such that $H(\theta_i) = p_i$. The
cone $S$ above $L_{\infty}$ then admits a natural parametrization by $(\theta, x_3)$ and the pull-back of $X$ through the mentioned
parametrization is given by
\[
\frac{1}{x_3} \left[ \left( -\frac{a\alpha_2}{\alpha_1} + c \alpha_1\alpha_2 \cos^2 \theta \right) \dd{\theta} -
x_3 c \alpha_1 \alpha_2  \cos \theta \sin \theta \dd{x_3} \right] \ .
\]
Thus, it can easily be checked that
\[
x_{3}(\theta) = x_{3}^{0} \sqrt{\frac{-a\alpha_2 + c \alpha_1^2 \alpha_2 \cos^2 \theta}{-a\alpha_2 + c \alpha_1^2 \alpha_2 \cos^2 \theta_0}} \ .
\]
The height function $|x_3| = |x_3(\theta)|$ goes to zero as $\theta$ goes to $\theta_1$ (and to $\theta_2$ as well), which means that $L$
approaches $\Delta_{\infty}$. The time-form over $L$ is given by
\[
dT = \frac{C}{\sqrt{|-a \alpha_2 + c \alpha_1^2 \alpha_2 \cos^2\theta|}} \, d\theta \, ,
\]
for some $C \in \R^{\ast}$ (naturally assuming $L$ distinct from $L_{\infty}$). Now, we have the following.

\smallbreak

\noindent {\it Claim}: There exists $\varepsilon$, with $0 < \varepsilon < \theta_2 - \theta_1$, and $M > 0$ such that
\[
0 \leq M \left| \theta - \theta_1 \right| \leq  \left| -a \alpha_2 + c \alpha_1^2 \alpha_2 \cos^2\theta \right|
\]
for every $\theta \in ]\theta_1, \theta_1 + \varepsilon[$.

\begin{proof}[Proof of the Claim]
Consider the function $u(\theta) = -a \alpha_2 + c \alpha_1^2 \alpha_2 \cos^2\theta$. Since $u$ vanishes at $\theta = \theta_1$,
the claim immediately follows if $u'(\theta_1) \ne 0$. In fact, if this is the case, it suffices to take $M = |f'(\theta_1)|/2$).
Since $u'(\theta) = -2c \alpha_1^2 \alpha_2 \sin \theta \, \cos \theta$ and recalling that $\alpha_1 \cos \theta_1 = \sqrt{a/c}$ and
$\alpha_2 \sin \theta_1 = \sqrt{b/c}$, we get that $u'(\theta_1) = 2 \alpha_1 \sqrt{ab}$. The value of $u'(\theta_1)$ is clearly
non-zero and the result follows.
\end{proof}

As an immediate consequence of the previous claim, we have
\[
\int_{\theta_i}^{\theta_i + \varepsilon} \left| \frac{C}{\sqrt{|-a\alpha_2 + c \alpha_1^2 \alpha_2 \cos^2 s |}} \right| \, ds \, \leq \,
\int_{\theta_i}^{\theta_i+\varepsilon} \left| \frac{C}{\sqrt{ M |s - \theta_1|}} \right| ds \,  < \, \infty \ .
\]
In other words, the corresponding geodesic reaches infinity in finite time and hence the geodesic is incomplete. In fact, there follows
from the previously mentioned symmetries that the maximal domain of definition is an open interval.

The case where $L$ is a leaf whose projection $L_{\infty}$ is part of a hyperbola accumulating at $p_1$ is left to the reader - the
calculations for the case of an ellipse applies equally well to this case.
\end{proof}

Let us finish this section by making some comments on the dynamics of $\fol$ when $ab < 0$. Recall that, when $ab < 0$, the leaves
of $\fol$ are contained in a compact part of $\RR^3$. This is in contrast with the case where $ab > 0$, were geodesics escaping from
every compact subset of $\RR^3$ naturally appears associated with the so-called idempotents. In turn, an idempotent induces a singular
point for the foliation induced in $\Delta_{\infty}$. In the case where $ab < 0$ we have no idempotents. In fact, if $c \ne 0$, there
are only 3 singular points for $\fol_{\infty}$, namely the origin of each one of the three charts considered above. The straight
line ``above'' them, although invariant for $\fol$, is constituted by singular points of $\fol$. If $c = 0$, the origin of the affine
coordinates is the only singular point for $\fol$ and the straight line ``above'' it is also constituted by singular points.

Note that in the case where $c=0$, the leaves of $\fol_{\infty}$ in the affine coordinates $(x_1,x_2)$ are
nothing but circles centered at the origin and the line at infinity of the mentioned chart is a leaf as well. In particular, the
foliation has no singular points on the other charts. The leaves on the invariant cones are all closed from from Proposition~\ref{prp:S1}.
As for the case where $c \ne 0$
\begin{itemize}
\item[1.] the foliation $\fol_{\infty}$ has a center at the origin of the affine coordinates $(x_1, x_2)$;
\item[2.] the foliation $\fol_{\infty}$ has a center at the origin of the affine coordinates $(y_1, y_2)$ and a saddle at the origin of the
affine coordinates $(u_1, u_2)$ in the case where $ac < 0$;
\item[3.] the foliation has $\fol_{\infty}$ a saddle at the origin of the affine coordinates $(y_1, y_2)$ and a center at the origin of the
affine coordinates $(u_1, u_2)$ in the case where $ac > 0$;
\end{itemize}
The figure below exhibits the leaves of $\fol_{\infty}$ in the case where $0 < \nu_1 < \nu_2 < \nu_3$ where, in particular, we
have $ac < 0$.
\[
\begin{array}{ccccc}
\includegraphics[width=0.3\textwidth]{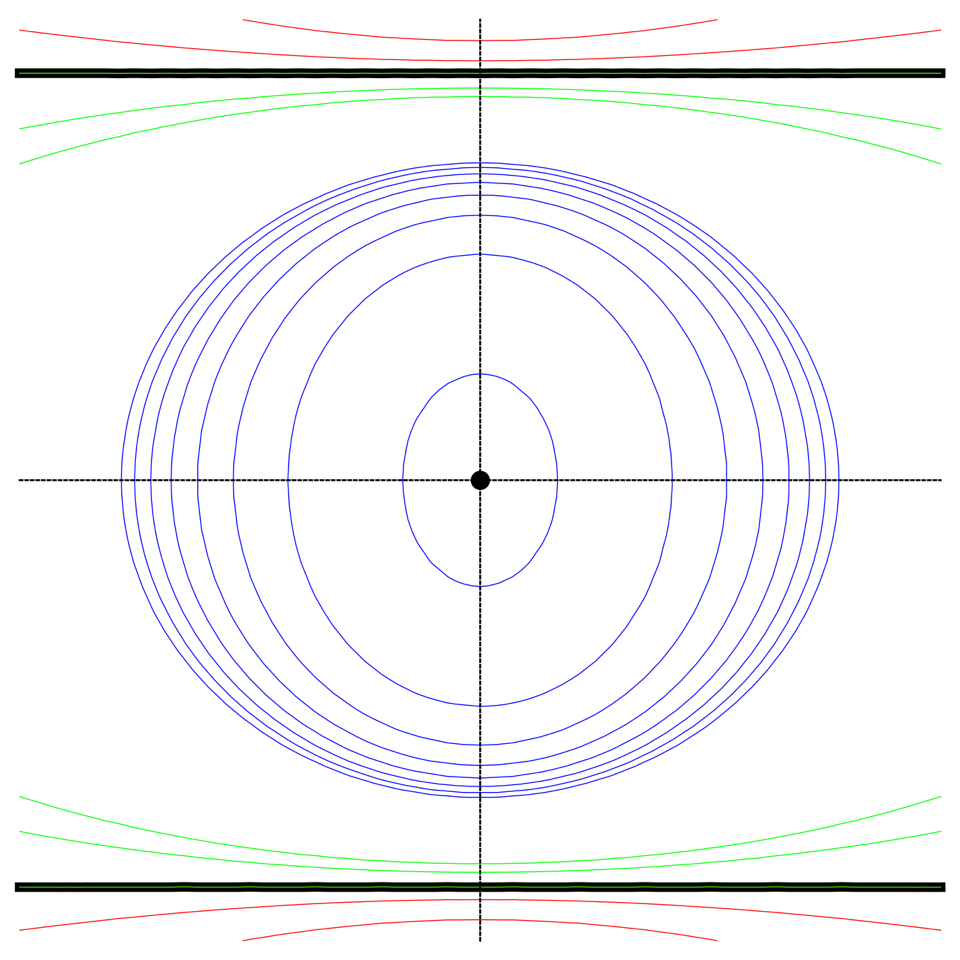} & \; \; &
\includegraphics[width=0.3\textwidth]{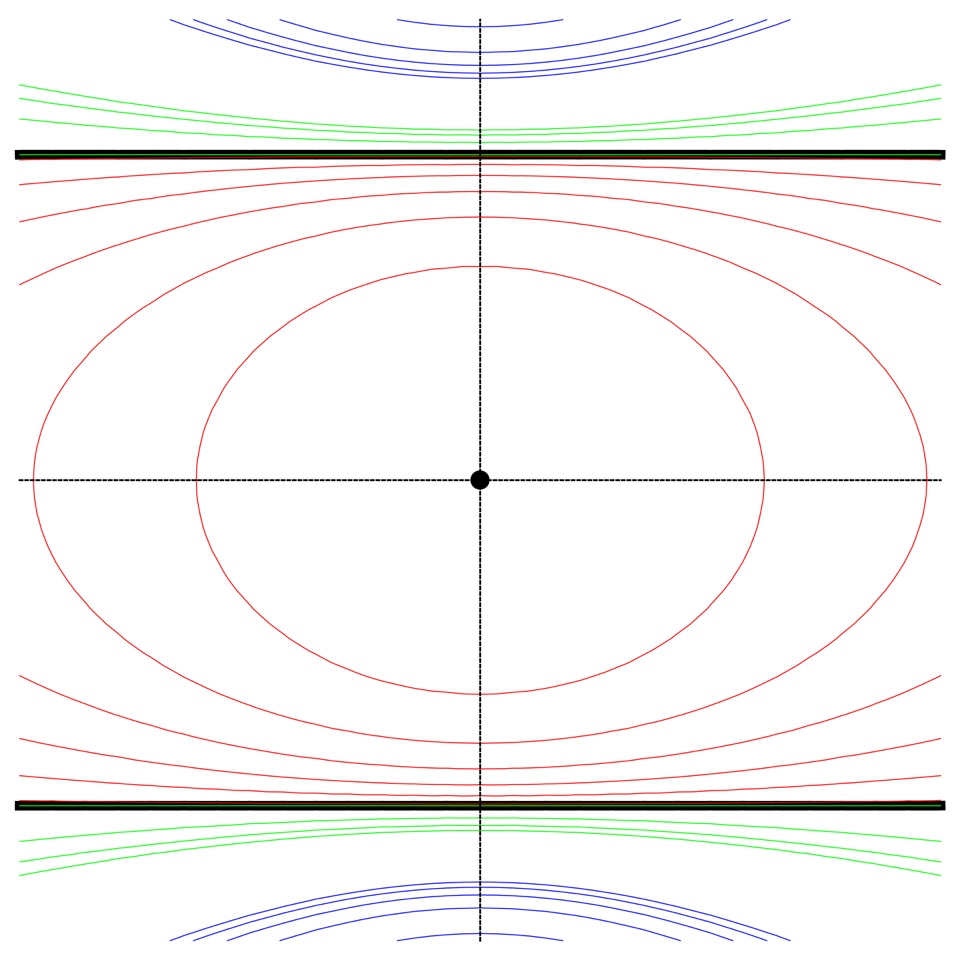} & \; \; &
\includegraphics[width=0.3\textwidth]{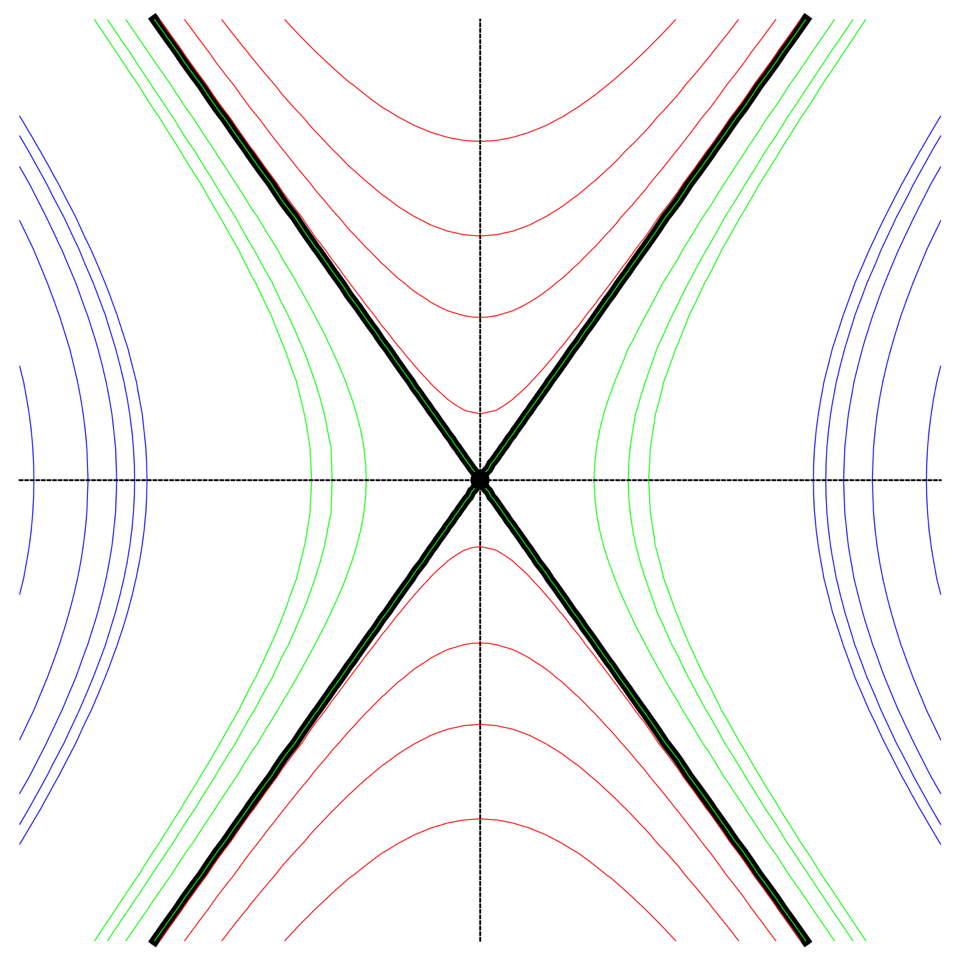}  \cr
 \text{Dynamics on $(x_1,x_2)$} & \; \; &
 \text{Dynamics on $(y_1,y_2)$} & \; \; &
 \text{Dynamics on $(u_1,u_2)$}
\end{array}
\]
Again, the leaves over each invariant cone are closed.


\section{Characterization of geodesics in case 3}\label{Sec:Case3}

Let us now consider the case where the isomorphism $\Phi$ admits a (real) eigenvalue $\lambda$ such that $m_a(\lambda) - m_g(\lambda) = 1$.
It has been proved in Lemma~\ref{lm:B} that there exists a $B$-pseudo-orthonormal basis $v = (v_k)$, satisfying $B(v_1,v_1) = B(v_2,v_3) = 1$
and $B(v_1,v_k) = B(v_k,v_k) = 0$ for $k=2,3$, with respect to which the isomorphism $\Phi$ can be written as in Section~\ref{sec: LP_equations},
where $\zeta \ne 0$ coincides with $q(v_3,v_3)$. In this case $\Phi^{-1}$ is given by
\begin{align}\label{eq:Phi-1-3}
\Phi^{-1} = \left(\begin{matrix} \eta & 0 & 0 \cr 0 & \nu & -\zeta\nu^{2} \cr 0 & 0 & \nu \end{matrix}\right)
\end{align}
where $\eta=1/\mu$ and  $\nu=1/\lambda$. As previously done, fix an element $z \in \slR$ and let $(z_1,z_2,z_3)$ stands for its coordinates
with respect to the basis $v = (v_k)$. From Lemma~\ref{lm:A} there follows that the Euler-Arnold vector field is written, in the affine
coordinates $(z_1, z_2, z_3) \in \RR^3$, as
\begin{align}\label{sist_EA3}
E = \zeta\nu^{2} z_3^2 \dd{z_1} + z_1 \left((\nu - \eta) z_2 - \zeta\nu^{2} z_3\right) \dd{z_2} + (\eta - \nu) z_1 z_3 \dd{z_3} \, .
\end{align}


\subsection{Characterization of geodesically complete metrics}

The characterization of the geodesically complete left-invariant pseudo-Riemannian metrics is given, in this case, not only in terms
of the eigenvalues of $\Phi$ (or, equivalently, $\Phi^{-1}$) but also in terms of $\zeta = q(v_3,v_3)$. To be more precise, the following
can be said.

\begin{teo}\label{thm:slR(S3)}
Then the left-invariant pseudo-Riemannian metric $q$ is geodesically complete if and only if $(\eta - \nu) \zeta \leq 0$.
\end{teo}

To prove Theorem~\ref{thm:slR(S3)} we may consider separately the cases where $(\eta - \nu) \zeta > 0$ and $(\eta - \nu) \zeta \leq 0$.
As in the previous case, we will prove the existence of the so-called idempotents in the case where $(\eta - \nu) \zeta > 0$. As to the
case where $(\eta - \nu) \zeta \leq 0$ the proof has to be different from the corresponding one in case where $\Phi$ is diagonalizable.
In fact, we will see that in this case the leaves although complete are not contained in a compact subset of $\R^3$. The proof of the
theorem is then divided into the two following propositions.

\begin{proposition}\label{prp:N3}
If $(\eta - \nu) \zeta > 0$ then there exists at least one geodesic that is incomplete.
\end{proposition}

\begin{proof}
Consider the rational extension of the Euler-Arnold vector field to $\RR\PP(3)$. The plane at infinity $\Delta_{\infty}$ is
invariant by the foliation $\fol$ induced by this extension since the Euler-Arnold vector field is a homogeneous polynomial
vector field distinct from a multiple of the radial vector field. Next, consider the affine coordinates $(x_1,x_2,x_3)$ nearby
$\Delta_{\infty}$ related to $(z_1,z_2,z_3)$ through the map $\Psi$ defined by~(\ref{eq:x-chart}). The Euler-Arnold vector field
is given, in the coordinates $(x_{1},x_{2},x_{3})$, by
\begin{equation}\label{eq_X_case3}
X = \frac{1}{x_3} \left[ (b - ax_1^2) \dd{x_{1}} - x_{1} (b + 2ax_2) \dd{x_{2}} - ax_1x_3 \dd{x_3} \right]  \, ,
\end{equation}
where $a = \eta - \nu$ and $b = \zeta \nu^2$. The assumption on the statement of the proposition translates, in these new parameters,
by $ab > 0$ and it can easily be checked that the intersection of the singular set of $\fol$ with $\Delta_{\infty}$, in the considered
affine chart, is constituted by $2$ points, namely
\[
p_1 = \left( \sqrt{\frac{b}{a}}, -\frac{b}{2a}, 0 \right) \qquad \text{and} \qquad p_2 = \left( -\sqrt{\frac{b}{a}},-\frac{b}{2a},0 \right) \, .
\]
The straight lines ``above'' each one of these singular points are regular leaves for $\fol$ which means that the restriction
of $X$ to them is a constant vector field. For example, the restriction of $X$ to the straight line or, more precisely, leaf
$L$ above $p_1$ is given by $X_{L} = -a \sqrt{b/a} \, \dd{x_3}$. In particular, $X_L$ is regular at the origin, which implies
that $L$ reaches the plane at infinity in finite time. The corresponding geodesic is therefore incomplete and the result follows.
\end{proof}

Let us now consider the case where $ab \leq 0$. In this case the following can be proved.

\begin{proposition}\label{prp:S3}
If $ab \leq 0$ then all geodesics are complete.
\end{proposition}

\begin{proof}
The proposition can be easily proved in the case where $ab = 0$. In fact, assuming $ab = 0$ (or, equivalently, $a=0$ since $b \ne 0$
by assumption), $\Phi$ has a unique eigenvalue, whose difference between its algebraic and geometric multiplicity is equal to~$1$.
The third component of $X$ in~(\ref{sist_EA3}) vanishes identically and, hence, the Euler-Arnold vector field reduces to
$bk^2 \dd{z_{1}} - bk z_1 \dd{z_{2}}$, which is clearly complete since the solution of the associated differential system is
polynomial with respect to $t$.

Assume now that $ab < 0$. The first integrals $I_1(z) = B(z,z)$ and $I_2(z) = B(z, \Phi^{-1}z)$ for $\fol$ are given in coordinates
$(z_1,z_2,z_3)$ by $I_1(z_1,z_2,z_3) = z_1^2 + 2 z_2z_3$ and $I_2(z_1,z_2,z_3) = \eta z_1^2 - \zeta \nu^2 z_3^2 + 2\nu z_2 z_3$,
which are clearly linearly independent since $\zeta \ne 0$. Consider then the linear combination of the present first integrals given
by $I = \nu I_1 - I_2$
\[
I(z_1,z_2,z_3) = (\nu - \eta) z_1^2 + \zeta \nu^2 z_3^2 = -a z_1^2 + b z_3^2 \, ,
\]
and which is also a first integral for $\fol$. Note that $I$ is a quadratic positive definite first integral for the system
in the variables $z_1$ and $z_3$ formed by the first and third equations associate with the Euler-Arnold system. This means that
the projection of every leaf of $\fol$ through the map $\pi(z_1,z_2,z_3) = (z_1,z_3)$ is contained in a compact part of $\R^2$
(although the leaf itself may escape to infinity). This means that every solution of the Euler-Arnold system is such that the
expression of $(z_1,z_3) = (z_1(t), z_3(t))$ is defined for every $t$. The completeness of the solution $(z_1,z_2,z_3) = (z_1(t),
z_2(t),z_3(t))$ of the Euler-Arnold system thus depends on the second equation. The second equation, in turn, can be seen as a
first order non-homogeneous linear differential equation in the variable $z_2$. In fact, it corresponds to the differential equation
$\dot{z}_2 = f(t) z_2 + g(t)$
where $f$ and $g$ are the differential functions on $\R$ given, respectively, by $f(t) = (\nu - \eta) z_1(t)$ and $g(t) = - \zeta\nu^{2} z_1(t)
z_3(t)$. The general solution of such a differential equation takes on the form
\[
z_2(t) = \left( G(t) + C\right)e^{F(t)}
\]
where $F$ is a primitive function of $f$, $G$ is a primitive function of $g(t) e^{-F(t)}$ and $C$ a real constant. Since $f$ and
$g$ are differentiable functions defined on $\R$ so are their primitive functions and, consequently, so is $z_2 = z_2(t)$. We thus
conclude that $(z_1,z_2,z_3) = (z_1(t), z_2(t),z_3(t))$ is complete.
\end{proof}

In what follows, we will use the method above to
give an alternative proof for the characterization of the geodesically complete metrics. Furthermore, in the case of incomplete
metrics, we will present the maximal domain of definition of every single geodesic.


\subsection{Geodesics for $ab < 0$}\label{subsec:Case3abnegative}

Recall that $\fol$ is given in the affine coordinates $(x_1,x_2,x_3)$ by the vector field $X$ in Equation~(\ref{eq_X_case3}). The
foliation $\fol_{\infty}$ induced by $\fol$ in $\Delta_{\infty}$ is, in particular, represented by the vector field
\begin{equation}\label{eq:Xinf_case_3}
X_{\infty} = (b - ax_1^2) \dd{x_{1}} - x_{1} (b + 2ax_2) \dd{x_{2}}
\end{equation}
and it becomes clear that $\fol_{\infty}$ does not admit singular points in the present coordinates. However, $\fol_{\infty}$ has
two singular points namely, the origin of the affine coordinates $(y_1,y_2)$ and the origin of the affine coordinates $(u_1,u_2)$
(with $(y_1,y_2)$ and $(u_1,u_2)$ as in Section~\ref{Sec:Case1}). By calculating the vector fields $Y_{\infty}$ and $U_{\infty}$,
representatives of $\fol_{\infty}$ in the affine coordinates $(y_1,y_2)$ and $(u_1,u_2)$, respectively, it becomes clear that the
origin is a degenerate singular point for $Y_{\infty}$ while the origin is a saddle for $U_{\infty}$. Furthermore, by taking the
quotient of the two independent first integrals $I_1, \, I_2$ on the affine coordinates $(x_1,x_2,x_3)$, it is clear that the leaves
of $\fol_{\infty}$ are all contained in conics. The figures below represent $\fol_{\infty}$ in the different affine coordinates

$$
\begin{array}{ccccc}
\includegraphics[width=0.3\textwidth]{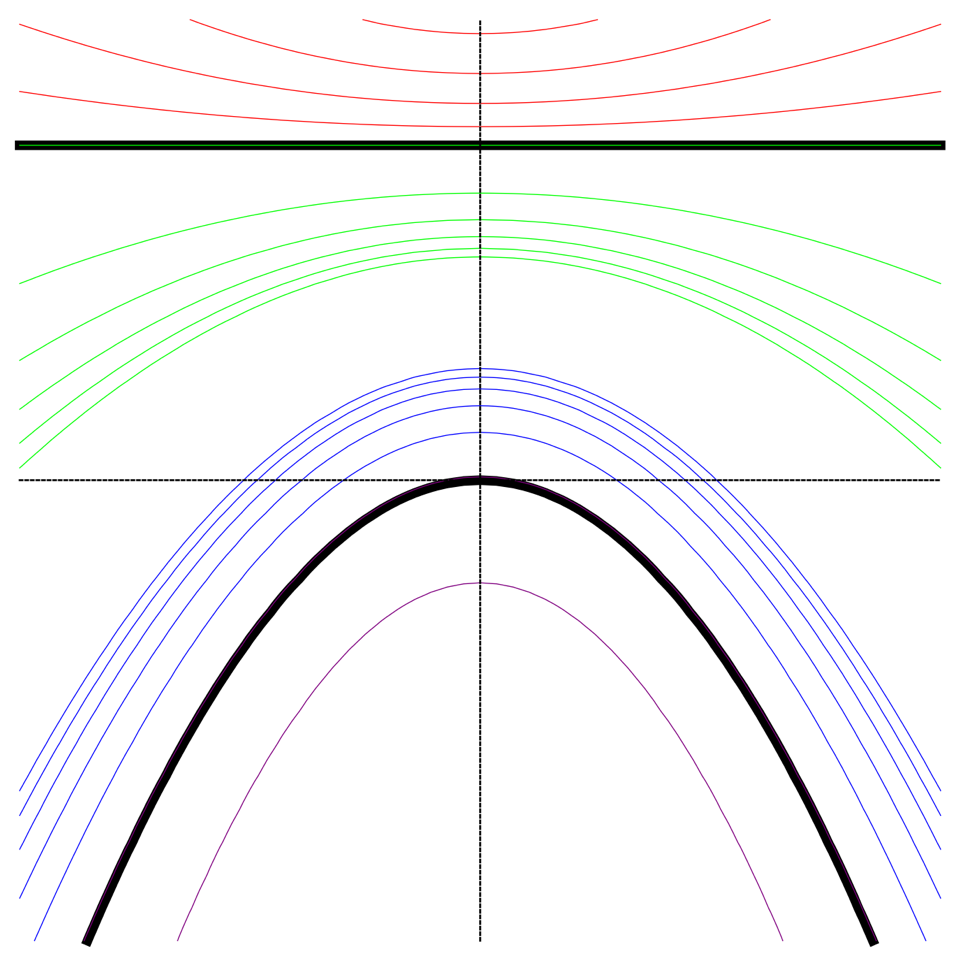} & \; \; &
\includegraphics[width=0.3\textwidth]{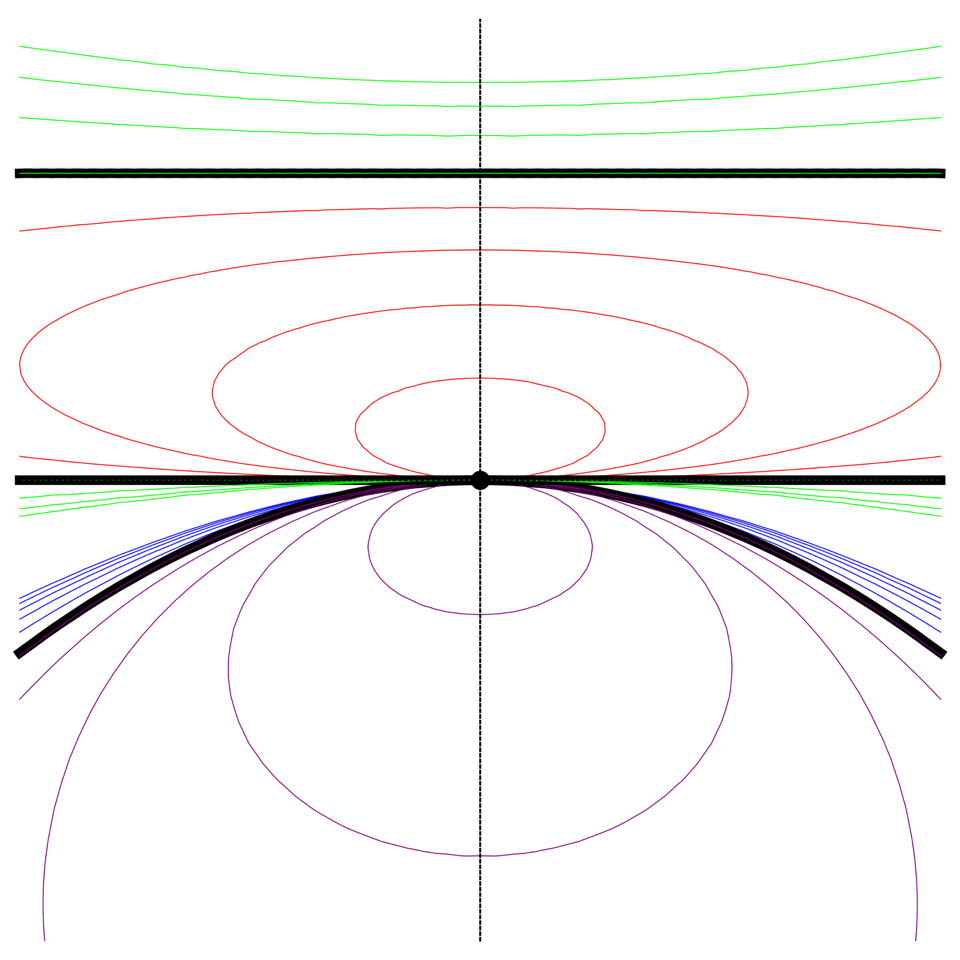} & \; \; &
\includegraphics[width=0.3\textwidth]{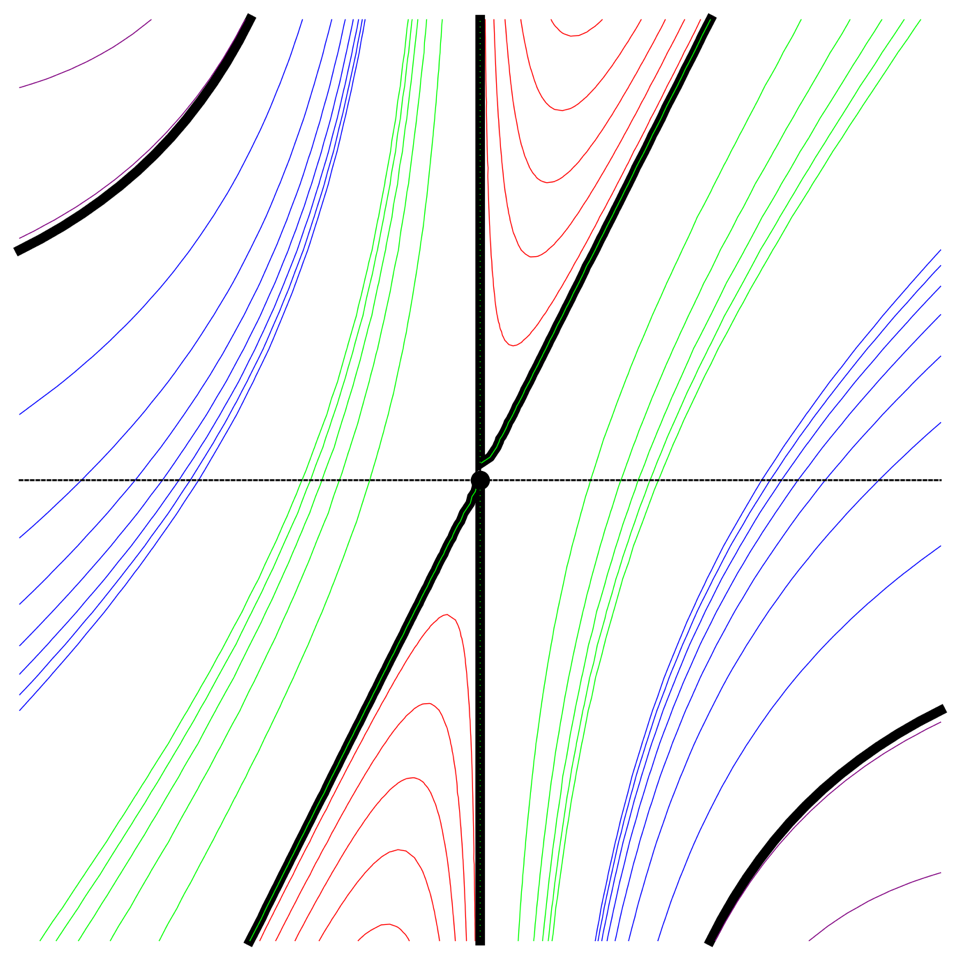}  \cr
 \text{Dynamics on $(x_{1},x_{2})$} & \; \; &
 \text{Dynamics on $(y_{1},y_{2})$} & \; \; &
 \text{Dynamics on $(u_{1},u_{2})$}
\end{array}
$$

By using the quotient of the two first integrals $I_1$ and $I_2$, it can easily be checked that the leaves of $\fol$ over $\Delta_{\infty}$
are given in coordinates $(x_1,x_2)$ by $x_2 = \alpha x_1^2 + \beta$, with $\alpha = \frac{K\eta - 1}{2(1-K\nu)}$ and $\beta = -
\frac{K\zeta \nu^2}{2(1-K\nu)}$,
which means that the leaves can be globally parameterized through the variable $x_1$. Fix then a leaf $L_{\infty}$ of $\fol$ over $\Delta_{\infty}$
and note that the cone above $L_{\infty}$ can naturally be parameterized through $(x_1,x_3)$. The pull-back of the Euler-Arnold vector field to the
mentioned cone if given by
\[
\frac{1}{x_1} \left[ (b-ax_a^2) \frac{\partial}{\partial x_1} - ax_1 x_3 \frac{\partial}{\partial x_3} \right] \, .
\]
Fixed a leaf $L$ in the mentioned cone (and naturally away from $\Delta_{\infty}$) we have that along $L$
\begin{equation}\label{eq:diff_x3_case3}
\frac{dx_3}{dx_1} = -\frac{x_1}{b-ax_1^2} \, x_3 \qquad \qquad \text{so that} \qquad \qquad
x_3 = x_3^0 \sqrt{\frac{b-ax_1^2}{b-a\left(x_1^0\right)^2}} \, .
\end{equation}
Now, since the time-form of the vector field along $L$ is given by
\[
dT = \frac{x_3(x_1)}{b-ax_1^2} \, dx_1 = \frac{C}{\sqrt{|b-ax_1^2|}} \, dx_1
\]
for some non-zero real constant $C$, we can easily conclude that the integral of the time-form along the entire leaf diverges. In fact,
since
\[
\int_1^{+\infty} \frac{C}{\sqrt{|b-ax_1^2|}} \, dx_1 = \int_1^{+\infty} \frac{C}{x_1 \,  \sqrt{\left| b/x_1^2 - a \right|}} \, dx_1
\]
and $\sqrt{\left| b/x_1^2 - a \right|}$ is bounded in the interval of integration, we get that the latter integral diverges.
Furthermore, since the integrand is an even function, the integral of the time-form along the entire leaf diverges
as well. We thus conclude that the geodesic associated with $L$ is complete.

Let us finally consider the leaves on the cone above the line at infinity of the affine coordinates $(x_1,x_2)$, which is invariant by
$\fol_{\infty}$. To study the mentioned leaves, we should then consider then the affine coordinates $(w_1, w_2, w_3)$ related with
$(z_1,z_2,z_3)$ through the map $\Lambda$ on $\RR^3 \setminus \{w_1 = 0\}$ defined by~(\ref{eqn:chart-w}), where divisor at infinity
is given by $\{w_1 = 0\}$ and the invariant plane by $\{w_3 = 0\}$. The Euler-Arnold vector field given in these coordinates by
\[
W = \frac{1}{w_{1}} \left[ -bw_{1}w_{3}^{2} \dd{w_{1}} - (aw_{2} + bw_{3} + bw_{2}w_{3}^{2}) \dd{w_{2}} + w_3(a - bw_{3}^2) \dd{w_{3}} \right]
\]
so that its restriction to the invariant plane $\{w_3 = 0\}$ is simply given by $-aw_2/w_1 \, \partial /\partial w_2$. The variable
$w_1$ along every single leaf $L$ in the mentioned plane is such that $\frac{dw_1}{dw_2} \equiv 0$ and, consequently, $w_1$ is
constant for all $t$. This implies that $L$ remains way from $\Delta_{\infty}$ nearby the origin of the present coordinates. The
leaf is, therefore, at least, $\RR^+$ or $\RR^-$ complete. To decide on the completeness of the leaf we should take the integral
of the associated time-form along the part of the leaf escaping from the domain of these charts and decide on its completeness.

Recall that $L$ is parameterized by $H(w_2) = (c,w_2,0)$ for some (non-zero) real constant $c$ and $w_2 > 0$ or $w_2 < 0$. Assume for
simplicity that $w_2 > 0$ along $L$. Since $dw_2/dT = -aw_2/w_1$, the time-form takes on the form
\[
dT = -\frac{c}{aw_2} \, dw_2
\]
It becomes clear that $\int_1^{+\infty} dT$ diverges and, consequently, the geodesic associated with $L$ is complete.


\subsection{Geodesics for $ab > 0$}

Proposition~\ref{prp:N3} ensures the existence of incomplete geodesics in the case where $ab > 0$. To be more precise, it was proved
that the straight line transverse to $\Delta_{\infty}$ and passing through $p_1$ is associated with incomplete geodesics. Let us now
investigate the domain of definition of the remaining geodesics.

Consider the restriction of $\fol$ to the divisor at infinity $\Delta_{\infty}$ along with the affine coordinates $(x_1,x_2)$,
$(y_1,y_2)$ and $(u_1,u_2)$ previously introduced. By considering the representative $X_{\infty}$ of the mentioned foliation (cf.
Equation~(\ref{eq:Xinf_case_3})) in the affine coordinates $(x_1, x_2)$, it can easily be checked that $\fol_{\infty}$ possesses
four singular points: two singular points in the affine coordinates $(x_1,x_2)$ (that, since distinct from the origin of these coordinates,
they are also presented in the other two affine coordinates) along with the origin of $(y_1,y_2)$ and of $(u_1,u_2)$. As previously mentioned,
the quotient of the first integrals $I_1$ and $I_2$ defines a first integral for $\fol_{\infty}$ and, since $I_1, \, I_2$ are homogeneous
polynomials of degree~$2$, the leaves of $\fol_{\infty}$ are conics on the different affine coordinates for $\Delta_{\infty}$. The figure
below represents $\fol_{\infty}$ in the different affine coordinates.
$$
\begin{array}{ccccc}
\includegraphics[width=0.3\textwidth]{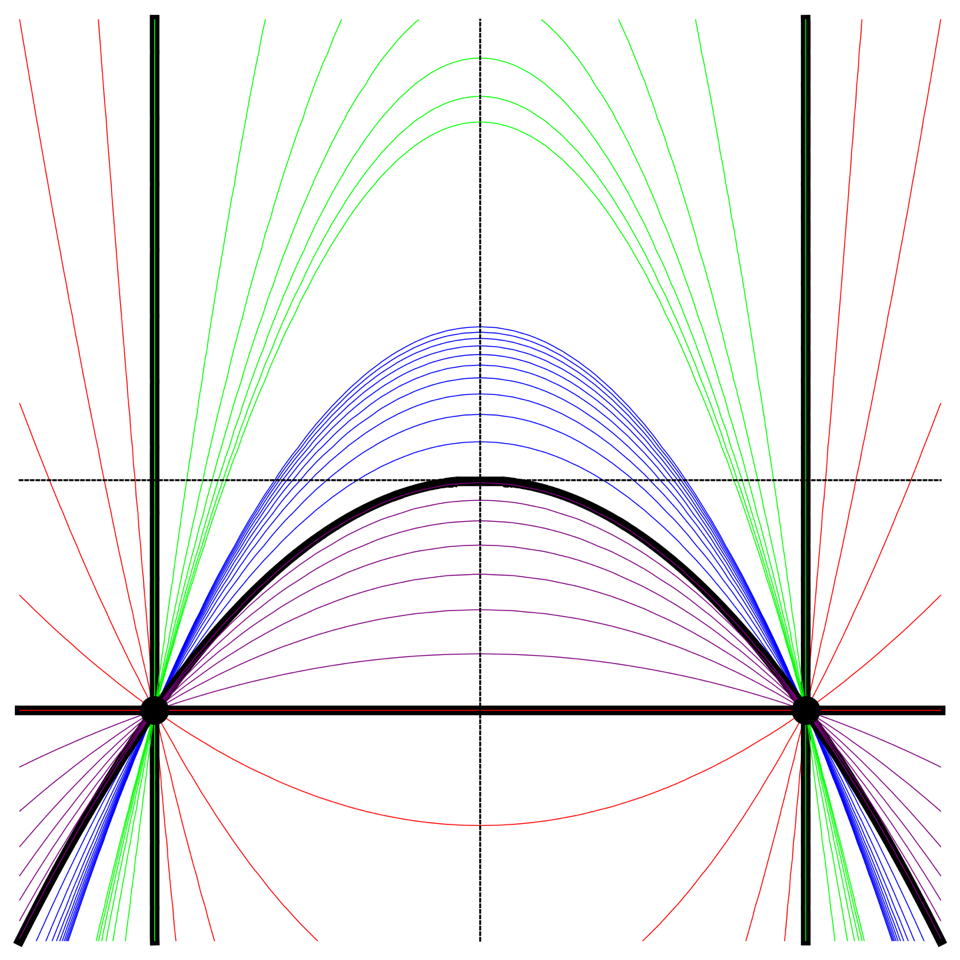} & \; \; &
\includegraphics[width=0.3\textwidth]{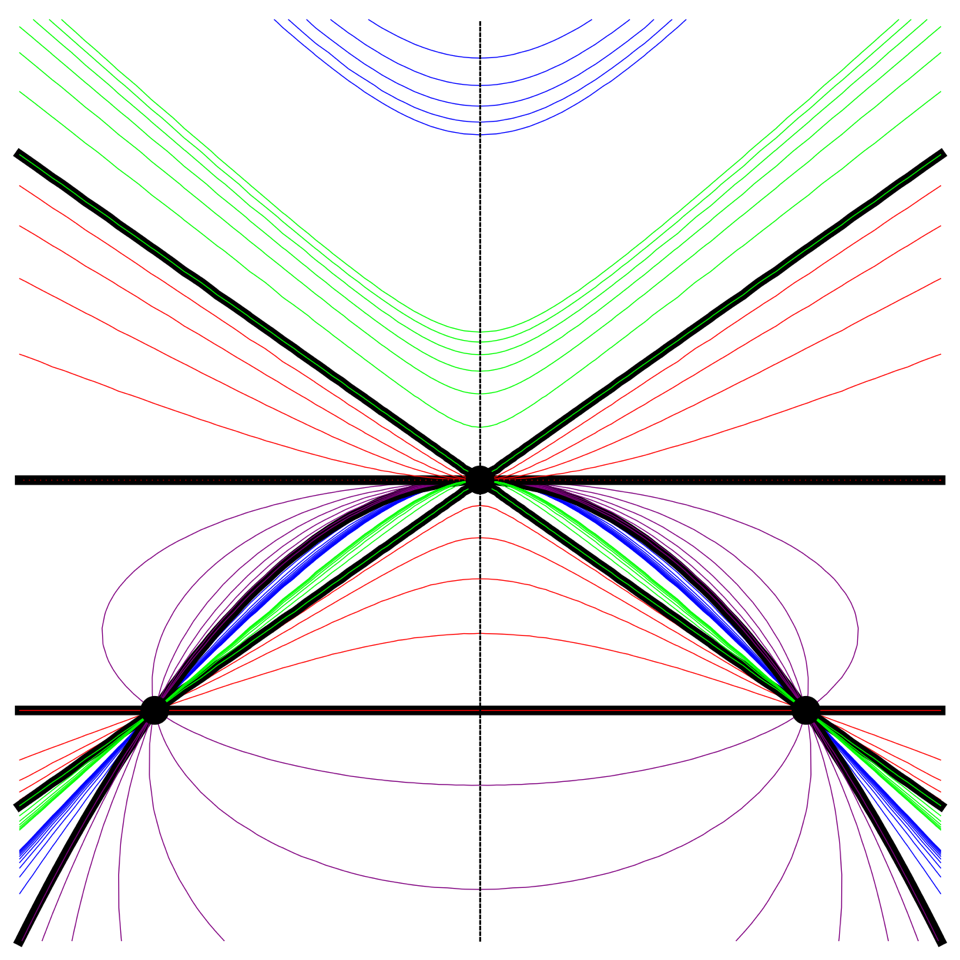} & \; \; &
\includegraphics[width=0.3\textwidth]{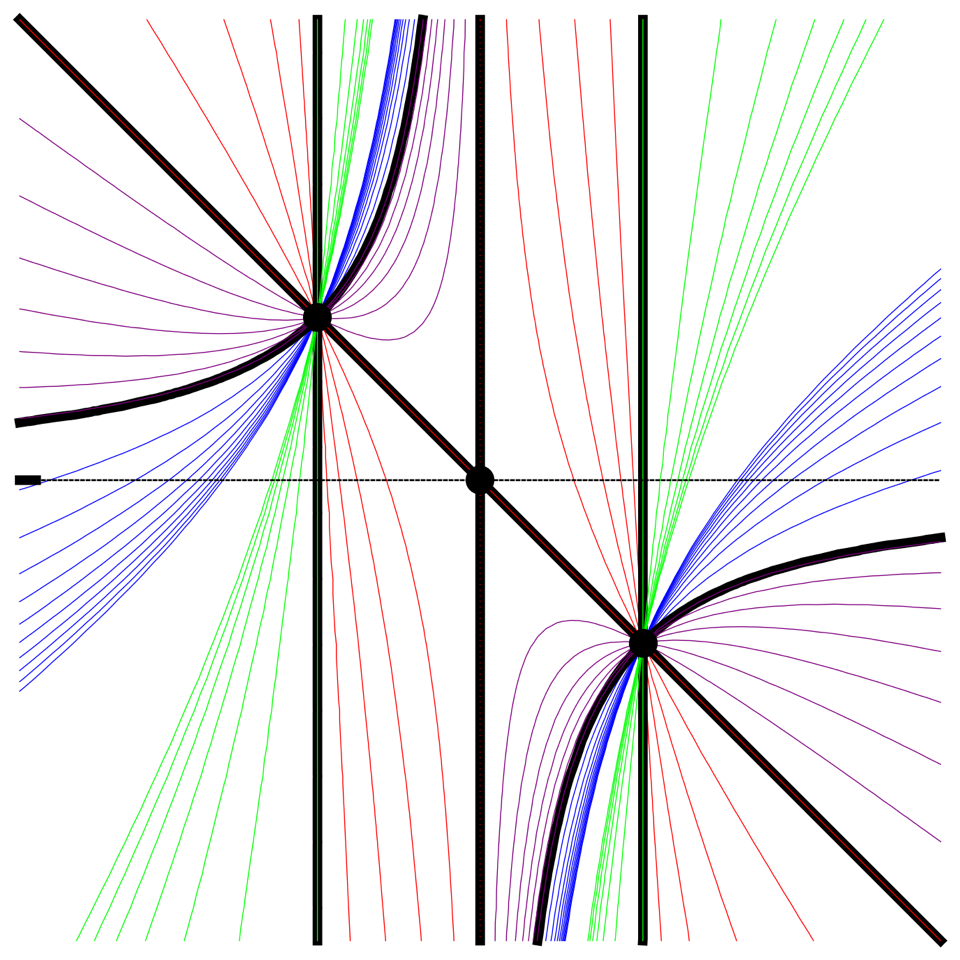}  \cr
 \text{Dynamics on $(x_{1},x_{2})$} & \; \; &
 \text{Dynamics on $(y_{1},y_{2})$} & \; \; &
 \text{Dynamics on $(u_{1},u_{2})$}
\end{array}
$$

The characterization of the domain of definition of the geodesics for $ab>0$ are described in the proposition below

\begin{proposition}\label{prop:incomplete_case_3}
Consider a leaf $L$ of $\fol$ and let $L_{\infty}$ stand for its projection on $\Delta_{\infty}$. The geodesic associated with
$L$ is complete if $L_{\infty}$ is contained in the line at infinity of the affine coordinates $(x_1,x_2)$ and incomplete in the
other cases. An incomplete geodesic is, in fact, $\R^+$ or $\R^-$ complete if $L_{\infty}$ is not contained in the bundle defined
by $-\sqrt{b/a} < x_1 < \sqrt{b/a}$. In other words, if $L_{\infty}$ is not a leaf joining the singular points $p_1$ and $p_2$.
\end{proposition}

\begin{proof}
Recall that if $L_{\infty}$ is a leaf of $\fol_{\infty}$ contained in the chart associated with $(x_1,x_2)$, then its points
satisfy the equation $(1 - K\eta) x_1^2 + 2(1 - K\nu) x_2 = -K \zeta \nu^2$
for some $K \in \R$. In particular, the leaves can be globally parameterized by the variable $x_1$ with exception of two of
them that corresponds to the vertical straight lines given by $x_1 = \pm \sqrt{b/a}$. Assume first that $L_{\infty}$ is one
of the parabolas described above. We go as in the case $ab<0$. More precisely, we start by noticing that a leaf $L$ in the
cone above $L_{\infty}$ can be parameterized as $H(x_1) = (x_1,x_2(x_1),x_3(x_1))$, where
\[
x_2(x_1) = x_2 = \frac{K\eta - 1}{2(1-K\nu)} \, x_1^2 - \frac{K\zeta \nu^2}{2(1-K\nu)} \qquad \text{and} \qquad
x_3(x_1) = x_3 = x_3^0 \sqrt{\frac{b-ax_1^2}{b-a\left(x_1^0\right)^2}} \, ,
\]
since $x_3$ satisfies the differential equation~(\ref{eq:diff_x3_case3}). Suppose that $L$ is such that the domain of definition
above the parametrization is the bounded interval $]-\sqrt{b/a},\sqrt{b/a}[$. In this case we have that $L$ approaches the plane
at infinity either when $x_1$ approaches $-\sqrt{b/a}$ or $\sqrt{b/a}$. The time we need to go through $L$ is given by
\[
\int_{-\sqrt{b/a}}^{\sqrt{b/a}} \, dT = \int_{-\sqrt{b/a}}^{\sqrt{b/a}} \, \frac{C}{\sqrt{|b - ax_1^2|}} \, dx_1 \, .
\]
The denominator of the integrand is given by $\sqrt{\sqrt{\frac{b}{a}} - x_1} \, \, \sqrt{\sqrt{\frac{b}{a}} + x_1}$, up to a
multiplicative constant, so that we conclude that the integral in question converges. The associated geodesic is then incomplete,
being defined in a bounded interval.

Assume now that the domain of the parametrization for $L$ is $]\sqrt{b/a}, +\infty[$. It is easy to check that the geodesic
associated with $L$ is incomplete. In fact, from the expression of $x_3 = x_3(x_1)$ we know that $L$ approaches $\Delta_{\infty}$.
Indeed, as previously seen, $x_3$ goes to zero as $x_1$ goes to $\sqrt{b/a}$. Furthermore, the argument presented above allows
us to claim that $\int_{\sqrt{b/a}}^{\sqrt{b/a} + \varepsilon} \, dT$ converges for every $\varepsilon > 0$. We then reach
$\Delta_{\infty}$ in finite time. The geodesic is, however, $\R^+$ or $\R^-$ complete. In fact, it can easily be checked that
$\int_{\sqrt{b/a} + \varepsilon}^{+\infty} \, dT$ diverges. Note that the argument presented in the case $ab<0$ applies equally
well in this case and the other case can be treated similarly.

Consider now the restriction of $X$ to the invariant plane $\{x_1 = \sqrt{b/a}\}$ (the restriction to $\{x_1 = -\sqrt{b/a}\}$
is analogous), which is given in the present coordinates by
\[
\frac{1}{x_3} \left[ -\sqrt{\frac{b}{a}} (b+2ax_2) \frac{\partial}{\partial x_2} - a\sqrt{\frac{b}{a}} x_3 \right] \, .
\]
It can easily be checked that $x_3$ is given by
\[
x_3(x_2) = x_3^0 \, \sqrt{\frac{b+2ax_2}{b+2ax_2^0}} \, .
\]
The expression of $x_3$ allows us to say that every leaf $L$ on the mentioned plane approaches $\Delta_{\infty}$ when its
projection $L_{\infty}$ approaches the singular point $p_1$. The leaf $L_{\infty}$ leaves then the domain of the present
chart and accumulate at the origin of the affine coordinates $(y_1,y_2)$. Since the affine coordinates $(x_1,x_2,x_3)$
parameterize the entire leaf $L$, we will use them to estimate the integral if the time-form along the leaf. To begin with,
note that the time-form is given by
\[
dT = -\frac{x_3(x_2)}{\sqrt{\frac{b}{a}} \, (b + 2ax_2)} \, dx_2 = \frac{C}{\sqrt{|b + 2ax_2|}} \, dx_2
\]
for some $C \in \R^{\ast}$. It is clear that the integrals $\int_{-\frac{b}{2a}}^{-\frac{b}{2a} + \varepsilon} \, dT$ and
$\int_{-\frac{b}{2a} + \varepsilon}^{+\infty} \, dT$ have different nature, were $\varepsilon > 0$. In fact, the first one
converges while the second one diverges so that associated geodesics are $\R^+$ or $\R^-$ complete.

It remains to consider the leaves contained in the invariant cone above the line at infinity of the affine coordinates $(x_1,x_2)$.
The line in question is given in coordinates $(u_1,u_2)$ by $u_1 = 0$. Consider then the affine coordinates $(w_1,w_2,w_3)$ related
with $(z_1,z_2,z_3)$ through the map $\Lambda$ defined by (\ref{eqn:chart-w}). In the present coordinates the divisor at infinity is given by
$\{w_1 = 0\}$ and $(w_3,w_2)$ coincides with $(u_1,u_2)$. The proof that the leaves in the invariant plane $\{v_3 = 0\}$ are complete
goes word by word as in the case $ab < 0$ in Subsection~\ref{subsec:Case3abnegative}. In fact, the calculations involved do not
depend on the parameters $a$ and $b$.
\end{proof}


\section{Characterization of geodesics in cases 2 and 4}\label{sec:Case2and4}

We will prove that every left-invariant pseudo-Riemannian metric in cases 2 and 4 is geodesically incomplete by
exhibiting an incomplete geodesic for each one of these metrics. As in the previous cases, we will also characterize
the maximal domain of definition of every single geodesic.


\subsection{Case 2}

In the case where the isomorphism $\Phi$ has two non-real eigenvalues, there exists a $B$-orthonormal basis $v = (v_k)$ with respect to which
the isomorphism $\Phi$ takes on the form presented in Section~\ref{sec: LP_equations}, where $\beta \ne 0$. The isomorphism $\Phi^{-1}$ is
then given by
\[
\Phi^{-1} = \left(\begin{matrix} \eta & 0 & 0 \cr 0 & \gamma & \zeta \cr 0 & -\zeta & \gamma \end{matrix}\right)
\]
where $\eta=1/\mu$, $\gamma=\alpha/(\alpha^{2}+\beta^{2})$ and $\zeta= -\beta/(\alpha^{2}+\beta^{2})$. Since we are assuming $\beta\neq 0$,
we have $\zeta\neq0$ as well and, in this particular case, the following can be said.

\begin{teo}\label{thm:slR(S2)}
There exists at least one geodesic that is incomplete.
\end{teo}

\begin{proof}
Let $(z_1,z_2,z_3) \in \RR^3$ stand for the coordinates of $z \in \slR$ in the above mentioned $B$-orthonormal basis $v$. In the
present coordinates, the Euler-Arnold vector field is given by
\begin{equation}\label{eq:EAVF2}
E = b(z_{2}^{2}+z_{3}^{2}) \dd{z_{1}} + z_1 (az_{3} - bz_{2}) \dd{z_{2}} + z_1 (az_{2} + bz_{3}) \dd{z_{3}} \, ,
\end{equation}
where $a = \gamma - \eta$ and $b = \zeta$ (and, in particular, $b \ne 0$). Consider now the usual affine coordinates $(x_1,x_2,x_3)$
where $\Delta_{\infty} \simeq \{x_3 = 0\}$. The vector field $E$ is written in these coordinates as
\[
X = \frac{1}{x_{3}} \left[ \left( b(x_2^2 - x_{1}^{2} + 1) - ax_{1}^{2}x_{2} \right) \dd{x_{1}} + x_1 \left( a(1 - x_{2}^{2}) - 2bx_{2} \right)
\dd{x_{2}} - x_{1} x_{3} (ax_{2} + b) \dd{x_{3}} \right]
\]
and the induced foliation on this chart has two singular points $p_1, \, p_2$ over $\Delta_{\infty}$. The position of the singularities
depend, however, on the parameters $a$ and $b$. More precisely, in the case where $a=0$, the two singular points on $\Delta_{\infty}$ are
$p_1 = (1,0,0)$ and $p_2 = (-1,0,0)$. In turn, in the case where $a \ne 0$, we have that $p_1 = (\rho_1,\rho_2,0)$ and $p_2 = (-\rho_1,\rho_2,0)$,
where
\[
\rho_1 = \sqrt{2b(-b+\sqrt{b^2+a^2})}/|a| \qquad \text{and} \qquad \rho_2 = (-b+\sqrt{b^2+a^2})/a \, ,
\]
if $b > 0$, and
\[
\rho_1 = \sqrt{2b(-b-\sqrt{b^2+a^2})}/|a| \qquad \text{and} \qquad \rho_2 = (-b-\sqrt{b^2+a^2})/a
\]
if $b < 0$. It can easily be checked that the restriction of $X$ to the line above each one of the above mentioned singular points
is a (non-zero) constant vector field meaning that the associated geodesic reaches $\Delta_{\infty}$ in finite time. These geodesics
are then incomplete.
\end{proof}

The study of this case is arduous since, as it becomes clear from the proof of Theorem~\ref{thm:slR(S2)}, not only the cases
$a=0$ and $a \ne 0$ should be considered separately, as in the case where $a \ne 0$, we should also look separately at the cases
$b > 0$ and $b < 0$. We will describe the foliation for the generic case where $a \ne 0$ and $b>0$ and the figure below exhibits
the leaves of $\fol_{\infty}$ in the case where $0 < \eta < \gamma < \zeta$.
$$
\begin{array}{ccccc}
\includegraphics[width=0.3\textwidth]{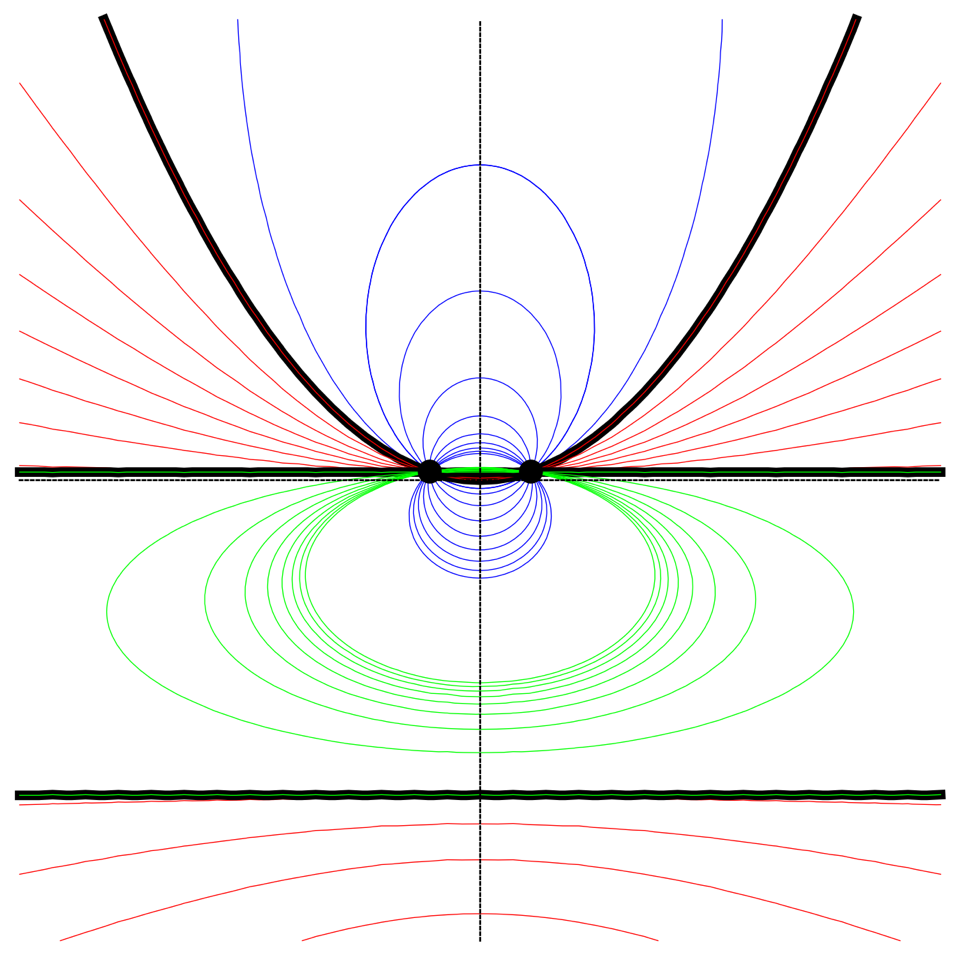} & \; \; &
\includegraphics[width=0.3\textwidth]{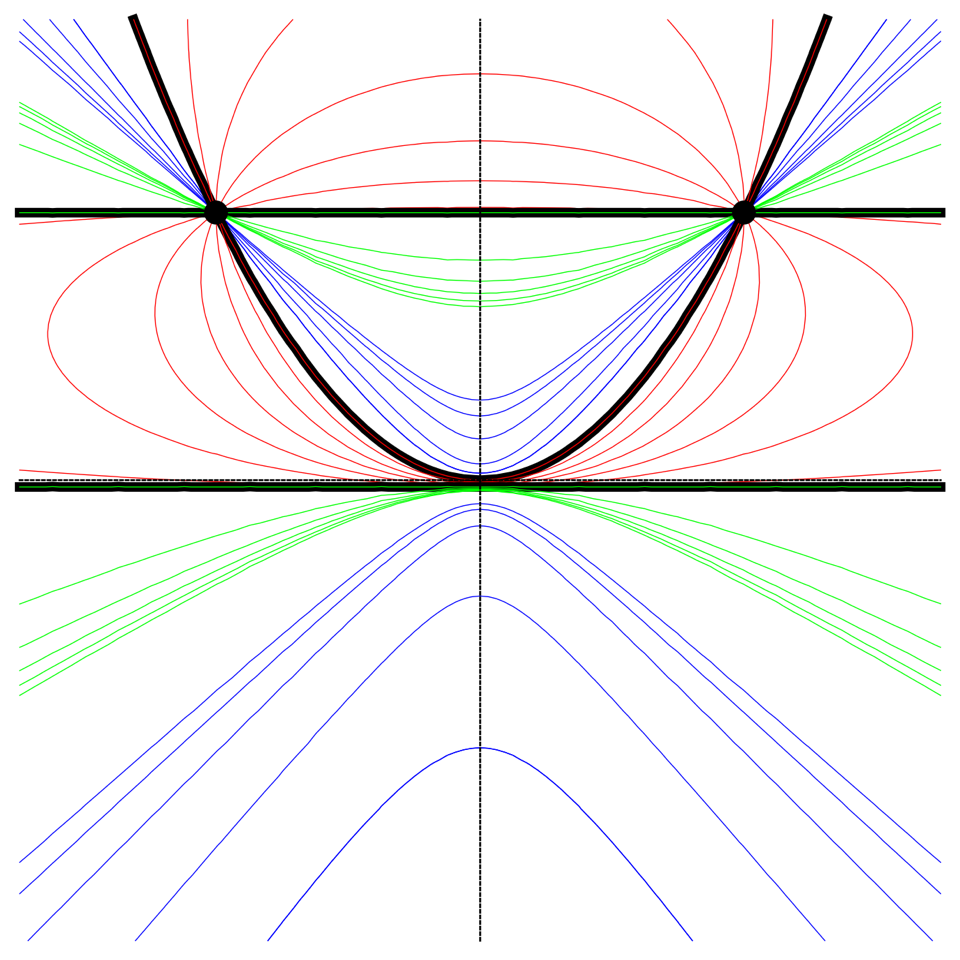} & \; \; &
\includegraphics[width=0.3\textwidth]{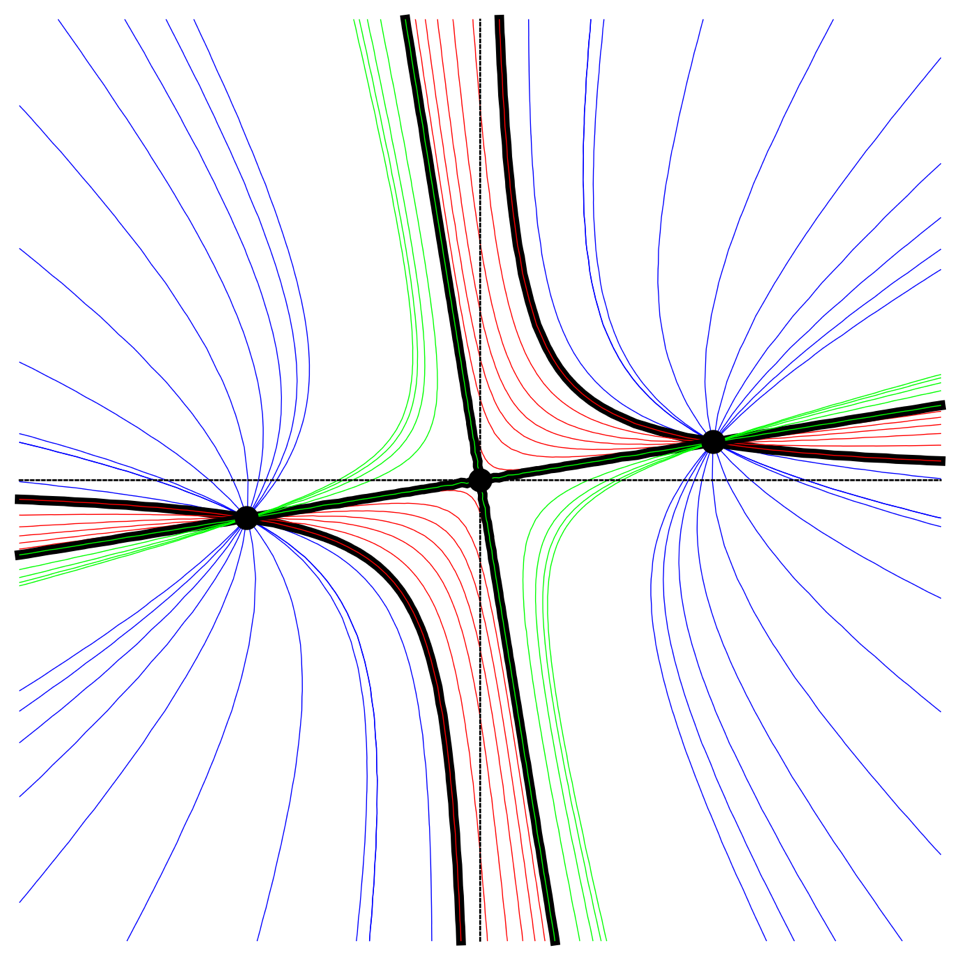}  \cr
 \text{Dynamics on $(x_{1},x_{2})$} & \; \; &
 \text{Dynamics on $(y_{1},y_{2})$} & \; \; &
 \text{Dynamics on $(u_{1},u_{2})$}
\end{array}
$$

The foliation $\fol_{\infty}$ has a total of three singular points, namely the two singular points $p_1, \, p_2$ mentioned above along with
the origin of the affine coordinates $(u_1,u_2)$ being, the later, a saddle for the mentioned foliation. It is clear from the picture that
every leaf of $\fol_{\infty}$ accumulates at both singular points $q_1$ and $q_2$, with exception of three of them, namely the straight
line given, in the affine coordinates $(x_1,x_2)$, by $x_2 = (-b-\sqrt{a^2+b^2})/a$ and the two unbounded leaves contained in the straight
line $x_2 = (-b+\sqrt{a^2+b^2})/a$. The three leaves in question are the unique leaves of $\fol_{\infty}$ accumulating at the origin of the
coordinates $(u_1,u_2)$. An analogous study to the one presented in the previous sections allows us to state the following.

\begin{proposition}
The geodesic associated with a leaf $L$ is complete if $L$ is contained in the invariant plane $\{x_2 = (-b-\sqrt{a^2+b^2})/a\}$ and incomplete
in all the other cases. An incomplete geodesic is, in fact, $\R^+$ or $\R^-$ complete if and only if $L$ is contained in the invariant plane
$\{x_2 = (-b+\sqrt{a^2+b^2})/a\}$ and its projection does not coincide with the line segment joining $p_1$ and $p_2$.
\end{proposition}

\begin{proof}[Idea of the proof]
Fix a leaf $L_{\infty}$ of $\fol_{\infty}$ accumulating at both $p_1$ and $p_2$ and let $L$ be a leaf of $\fol$ projecting on $L_{\infty}$
and distinct from $L_{\infty}$. It can be checked that if $H(\theta,x_3) = (h_1(\theta), h_2(\theta), x_3)$ stands for a parametrization of
the the cone above $L_{\infty}$, the pull-back of the Euler-Arnold vector field always takes on the form
\[
\frac{u(\theta)}{x_3} \dd{\theta} + \frac{u'(\theta)}{2} \dd{x_3}
\]
for some function $u = u(\theta)$ vanishing at $\theta_1$ and $\theta_2$, where $\theta_1$ and $\theta_2$ are such that $H(\theta_1,0) = p_1$
and $H(\theta_2,0) = p_2$. In particular, the function $x_3 = x_3(\theta)$ is such that
\[
\frac{dx_3}{d\theta} = \frac{u'(\theta)}{2u(\theta)} \, x_3
\]
so that $x_3(\theta) = x_3^0 \sqrt{u(\theta)/u(\theta_0)}$. Since $u = u(\theta)$ vanishes at $p_1$ and $p_2$, so does $x_3 = x_3(\theta)$.
In particular, the leaves $L$ accumulate at both $p_1$ and $p_2$. In order to discuss the completeness of the leaves, the time-form has to be
considered. In fact, recall that the geodesic associated to the leaf $L$ is complete if and only if both integrals $\int_{\theta_1}^{\theta}
\, dT$ and $\int_{\theta_2}^{\theta}$ diverge for $\theta$ arbitrarily close to $\theta_1$ and $\theta_2$, respectively.
Up to a multiplicative constant, the time-form of $X$ along $L$ takes on the form $dT = 1/\sqrt{|u(\theta)|} \, d\theta$. The expression of
$u = u(\theta)$ depends on the type of the leaf $L_{\infty}$ but they are similar to the expressions obtained in the previous two sections.
In particular, we are able to prove that both of the integrals above converge. The geodesic associated with $L$ is then incomplete (being
neither $\R^+$ nor $\R^-$ complete) and the result follows in this generic case.

Assume now that $L$ is contained in the invariant plane $\{x_2 = (-b-\sqrt{a^2+b^2})/a\}$ and distinct from $L_{\infty}$. We claim that
$L$ is contained in a compact subset of $\R^3$. In fact, since $L$ contains no singular points in the chart $(x_1,x_2,x_3)$, the ``height''
function is bounded from below by a strictly positive constant on any compact subset of it. The projection of the complement of a compact
subset of $L$ is necessarily contained in a compact subset of the domain of definition of the affine coordinates $(u_1,u_2)$ for $\Delta_{\infty}$.
With respect to the coordinates $(w_1,w_2,w_3)$ (where $(u_1,u_2)$ is identified with $(w_3,w_2)$ and $\Delta_{\infty} \simeq \{w_1 = 0\}$),
the invariant plane becomes $\{w_2 = w_3(-b-\sqrt{a^2+b^2})/a\}$ and it can easily be checked that the ``height'' function over this complement
of $L$ is also bounded from below by a strictly positive constant. Indeed, the leaf converges to a point on the line above the origin, distinct
from the origin itself. Since $L$ is contained in a compact subset of $\R^3$, its associated geodesic is complete.

Finally, it remains to consider the leaves whose projection $L_{\infty}$ on $\Delta_{\infty}$ belong to an unbounded part of the straight line
through $p_1$ and $p_2$. Note that $L_{\infty}$ becomes a line segment joining $p_1$ or $p_2$ to the origin in the affine coordinates $(w_2,w_3)$.
By re-writing the Euler-Arnold vector field in the coordinates $(w_1,w_2,w_3)$ and restricting it to the invariant plane above $L_{\infty}$, we
can easily check that the ``height'' function $|w_1|$ is bounded from below by a positive constant. To be more precise, $L$ converges to a point
in the singular set of $\fol$ given by $\{w_2 = w_3 = 0\}$, with $w_1 \ne 0$. Furthermore, the geodesics associated with each point in the singular
set in question are then naturally complete.
\end{proof}


\subsection{Case 4}

Let us finally consider case 4, i.e. the case where the isomorphism $\Phi$ possesses a (real) eigenvalue $\lambda$ such that $m_a(\lambda)
-m_g(\lambda)=2$. In this case there exists a $B$-pseudo-orthonormal basis $v = (v_k)$ with respect to which the isomorphism $\Phi$ takes
on the form presented in Section~\ref{sec: LP_equations}, where $\zeta$ can be assumed strictly positive. The isomorphism $\Phi^{-1}$ takes
then on the form
\[
\Phi^{-1} = \left(\begin{matrix} \nu & 0 & -\zeta\nu^{2} \cr -\zeta\nu^{2} & \nu & \zeta^{2}\nu^{3}\cr 0 & 0 & \nu \end{matrix}\right)
\]
where $\nu=1/\lambda$.

\begin{teo}\label{thm:slR(S4)}
There exists at least one geodesic that is incomplete.
\end{teo}

\begin{proof}
If $(z_1,z_2,z_3) \in \RR^3$ stands for the coordinates of $z \in \slR$ in the above mentioned $B$-pseudo-orthonormal basis,
the Euler-Arnold vector field is given in these coordinates by
\begin{equation}\label{eq:EAVF4}
E = \zeta \nu^2 \left[ z_3(-\zeta \nu z_3 + z_1) \dd{z_{1}} + (z_{2}z_{3} - z_{1}^{2} + \zeta \nu z_{1}z_{3}) \dd{z_{2}}
- z_{3}^{2} \dd{z_{3}} \right] \, .
\end{equation}
In turn, on the affine coordinates $(x_1,x_2,x_3)$, the Euler-Arnold vector field becomes
\[
X = \frac{\zeta \nu^{2}}{x_3} \left[ (2x_1 - \zeta \nu) \dd{x_1} + \left( \zeta \nu x_1 - x_1^2 + 2x_2 \right) \dd{x_2}
+ x_3 \dd{x_3} \right] \ .
\]
and the induced foliation on this chart has a unique singular points over $\Delta_{\infty}$, namely $p = \left( \frac{\zeta \nu}{2},
-\frac{\zeta^2 \nu^2}{8} ,0 \right)$. The restriction of $X$ to the invariant straight line ``above'' $p$ is the constant vector
field $\zeta \nu^{2} \dd{x_{3}}$ and, since it is non-zero, the result follows.
\end{proof}

The foliation induced by $\fol$ on $\Delta_{\infty}$ possesses two singular points: the point $p$ in the previous lemma and $q$, the
origin of the affine coordinates $(y_1,y_2)$. Furthermore, it can be checked that every single leaf of $\Delta_{\infty}$ accumulates
at both $p$ and $q$, with exception of the leaf that in coordinates $(y_1,y_2)$ is given by $y_2 = 0$. The latter accumulates uniquely
at $q$. The figure below exhibits the leaves of $\fol_{\infty}$ in the case where $0 < \nu < \zeta$.
$$
\begin{array}{ccccc}
\includegraphics[width=0.3\textwidth]{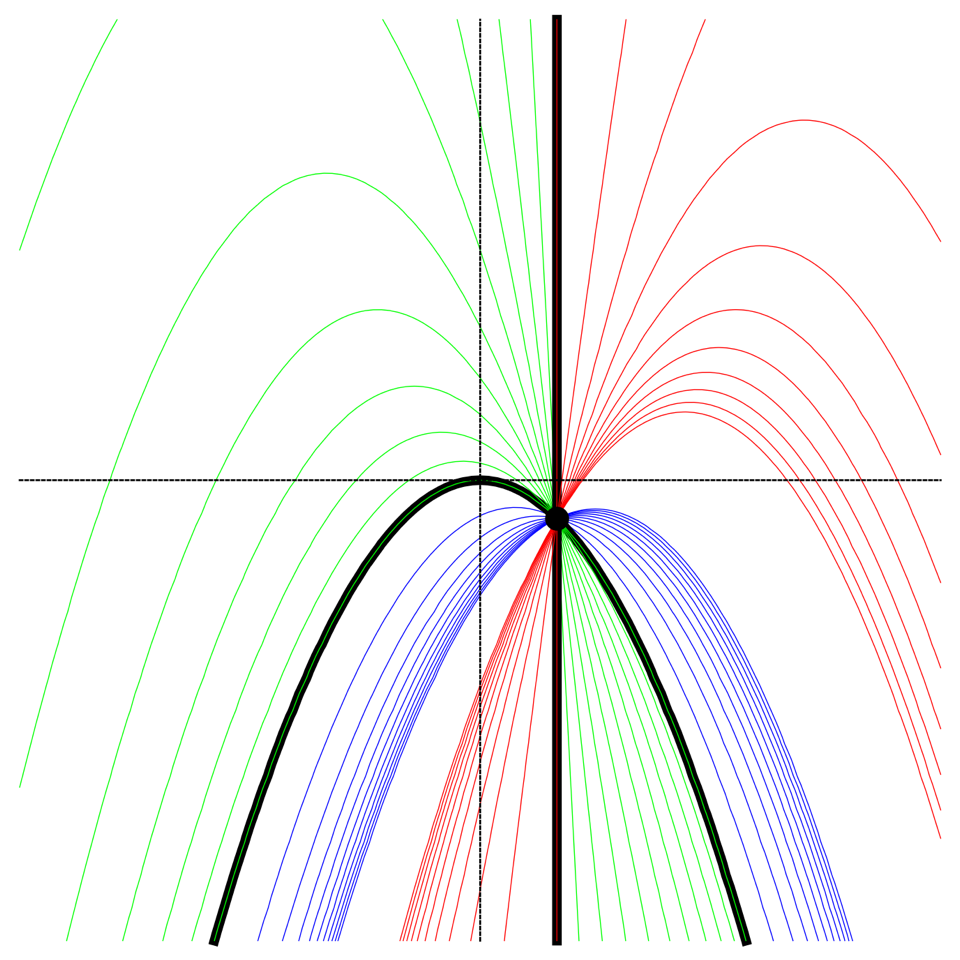} & \; \; &
\includegraphics[width=0.3\textwidth]{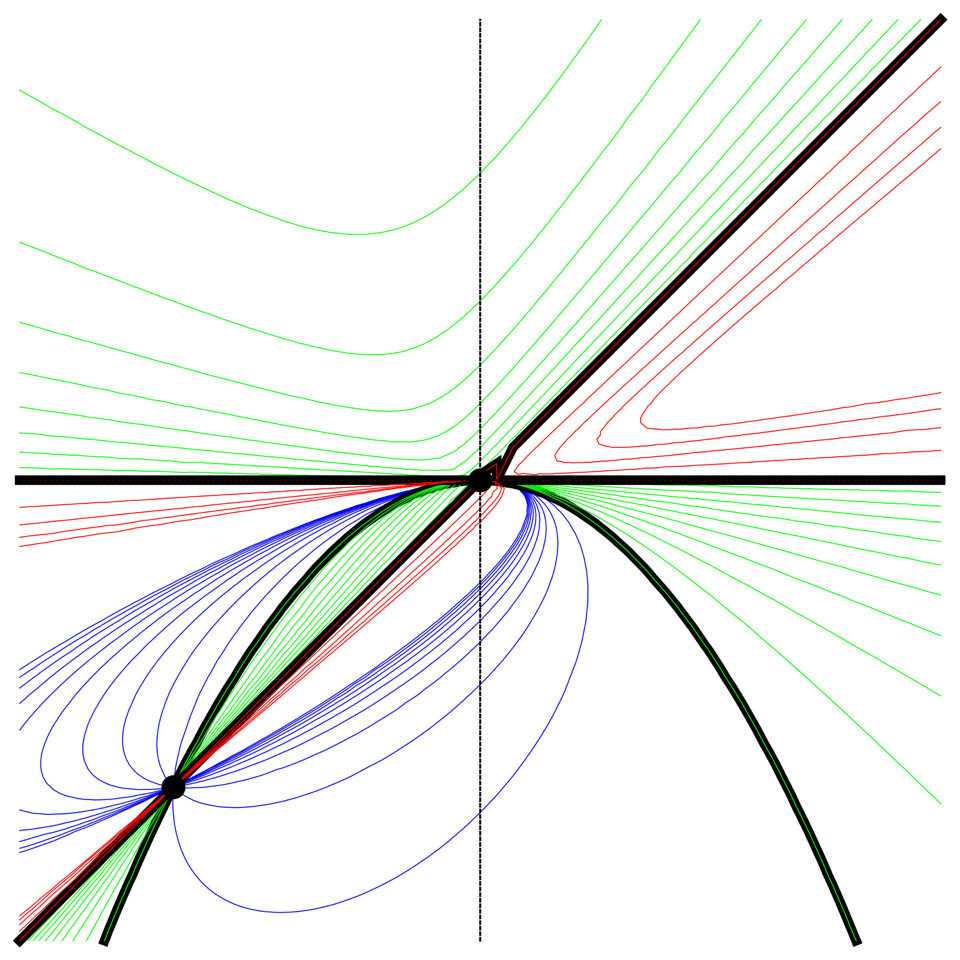} & \; \; &
\includegraphics[width=0.3\textwidth]{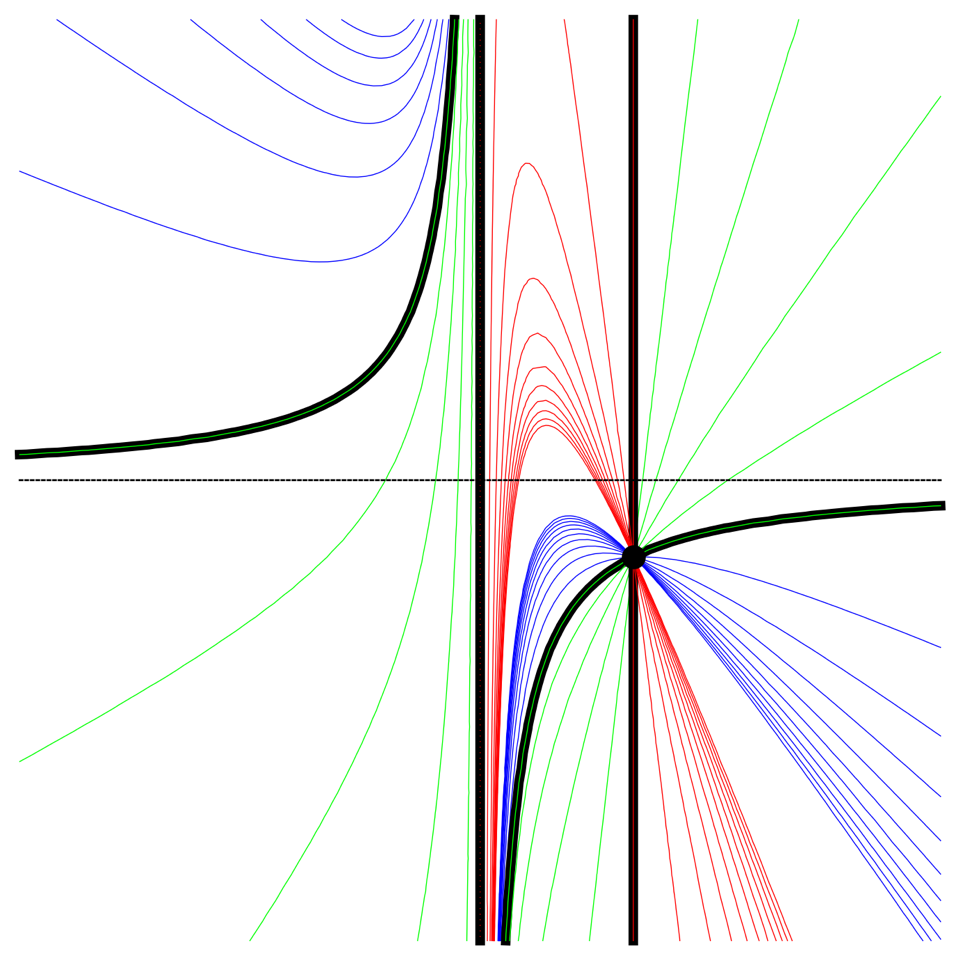}  \cr
 \text{Dynamics on $(x_{1},x_{2})$} & \; \; &
 \text{Dynamics on $(y_{1},y_{2})$} & \; \; &
 \text{Dynamics on $(u_{1},u_{2})$}
\end{array}
$$

\begin{proposition}
The geodesics associated with leaves in the invariant plane $\{y_2 = 0\}$ are all complete. The remaining ones are incomplete.
Moreover, they are $\R^+$ or $\R^-$ complete since, although all of them approach $\Delta_{\infty}$ as their projections goes
to $p$ and to $q$, they reach $\Delta_{\infty}$ in finite time as their projections approach $p$ but not when they approach $q$.
\end{proposition}

\begin{proof}
The leaves of $\fol$ over $\Delta_{\infty}$ are given in the affine coordinates $(x_1,x_2)$ by
\[
x_2 = -\frac{1}{2}x_1^2 - \frac{K\zeta \nu^2}{1-K\nu} x_1 + \frac{K\zeta^2 \nu^3}{2(1-K\nu)} \, ,
\]
which means that in the present coordinates they are all parabolas, with exception of a unique leaf that is the straight line given
by $\{x_1 = \zeta \nu/2\}$. Fix a generic leaf for $\fol_{\infty}$, i.e. a parabola, and take the natural parametrization of the
cone above $L_{\infty}$ by $(x_1,x_3)$. The pull-back of the Euler-Arnold vector field to the mentioned cone is given by
\[
\frac{\zeta \nu^2}{x_3} \left[ (2x_1 - \zeta \nu) \frac{\partial}{\partial x_1} + x_3 \frac{\partial}{\partial x_3} \right] \, .
\]
Fixed a leaf $L$ in the mentioned cone (and naturally away from $\Delta_{\infty}$) we have that along $L$
\[
\frac{dx_3}{dx_1} = \frac{x_3}{2x_1 - \zeta \nu} \qquad \qquad \text{and, consequently,} \qquad \qquad
x_3 = x_3^0 \sqrt{\frac{2x_1 - \zeta \nu}{2x_1^0 - \zeta \nu}} \, .
\]
There immediately follows that the leaf $L$ in the fixed cone goes to $\Delta_{\infty}$ as its projection goes to $p$. Now,
since the time-form of the vector field along $L$ is given by
\[
dT = \frac{C}{\sqrt{|2x_1 - \zeta \nu|}} \, dx_1 \, ,
\]
for some non-zero real constant $C$, we can easily conclude that $\int_{x_1}^{\zeta \nu/2} dT$ converges, while
$\int_{x_1}^{\infty}$ diverges. The leaves are then $\R^+$ or $\R^-$ complete.

It remains to study the leaves in the invariant plane $\{x_1 = \zeta \nu/2\}$ and in the plane ``above'' the line at infinity of the
present chart. The study for the remaining leaves will be done in the affine coordinates $(v_1,v_2,v_3)$, where the mentioned cones
are given, respectively, by $S_1 = \{v_1 = (\zeta \nu/2) v_3\}$ and $S_2 = \{v_3 = 0\}$. By taking the natural parametrization of
$S_1$ through $(v_1, v_2)$, the pull-back of the Euler-Arnold vector field is given by
\[
\frac{v_1}{v_2} \left[ v_1( -4\nu - \zeta \nu^2 v_1) \dd{v_1} + v_2 (-2\nu - \zeta \nu^2 v_1) \dd{v_2} \right] \ .
\]
There immediately follows that
\[
v_2(v_1) = v_2^0 \sqrt{\frac{v_1( -4 - \zeta \nu v_1)}{v_1^0( -4 - \zeta \nu v_1^0)}} \, ,
\]
which means that every leaf in $S_1$ goes to $\Delta_{\infty}$ as its projection approaches $p$ and $q$. Note that the time-form along
a leaf of $\fol$ in $S_1$ (and not contained in $\Delta_{\infty}$) is given by
\[
dT = \frac{C}{|v_1|\sqrt{| -4 - \zeta \nu v_1|}} \, dv_1
\]
for some non-zero constant $C \in \R$. The integral $\int_0^{\varepsilon} dT$ naturally diverges for $\varepsilon$ arbitrarily close to
zero while $\int_{-4/(\zeta \nu)}^{-4/(\zeta \nu) + \varepsilon} dT$ converges for $\varepsilon$ arbitrarily close to zero (note that
the point $p$ has coordinates $(-4/(\zeta \nu), 0, -8/(\zeta \nu)^2)$ in the present chart). The result follows in this case.

Finally, let us consider the invariant plane $S_2$. The restriction of the Euler-Arnold vector field to $S_2$ is, up to a multiplicative
factor, a multiple of the radial vector field. In fact, it is given by
\[
\frac{\zeta \nu^2 v_1^2}{v_2} \left[ v_1 \dd{v_1} + v_2 \dd{v_2} \right] \ .
\]
and so, its associated leaves naturally satisfy $v_2(v_1) = k v_1$ for some constant $k \in \R$. The time-form along a leaf $L$ in $S_2$
(and not contained in $\Delta_{\infty}$) then takes on the form
\[
dT = \frac{C}{v_1^2} \, dv_1
\]
for some non-zero constant $C \in \R$ and the integral $\int_0^{\varepsilon} \, dT$ clearly diverges for $\varepsilon$ arbitrarily close
to zero. To claim that the associated geodesic is complete, it suffices to notice that a leaf in $S_2$ through a point $r = (r_1,r_2,0)$
with $r_1 > 0$, goes through the “infinity” of the present coordinates and ``returns'' to the present coordinates as a leaf in $S_2$ passing
through a point $s = (s_1,s_2,0)$ with $s_1 < 0$.
\end{proof}

\bigskip

\begin{obs}
A classical result by J. Lafuente, \cite{Lafuente}, states that for a locally symmetric Lorentzian manifold, the three types of causal incompleteness (timelike, lightlike and spacelike) are equivalent.  Our approach allowed us to check, case-by-case, that the same result holds for the Lie group $\mathrm{SL}(2,\mathbb{R})$. If follows from Lebnitz's rule that idempotents are always null geodesics. It is also interesting to observe that,  for this particular space, they are the only ones. 
\end{obs}

\bigskip


\section{The Lie group $\mathrm{SL}(2,\mathbb{C})$}\label{sec:SLC}




A full classification and characterization of pseudo-Riemannian metrics on $\mathrm{SL}(2,\mathbb{R})$, by explicitly writing the Euler-Arnold vector field, was attainable since there are only four normal forms for the matrix of $\Phi$ and certain orthogonality conditions can also be derived (see Lemma \ref{lm:A}). Moreover,    the possibilities for the basis in which the matrix of $\Phi$ is written are very controlled (as proved in Lemma \ref{lm:B}). This is essentially due to the fact that the orthogonal group of the Killing form modulo the automorphism group of $\mathfrak{sl}(2,\mathbb{R})$ is isomorphic to $\mathbb{Z}_2$, a finite group with only two elements.

Considering the space $F$ of all $B$-orthogonal basis of $\mathfrak{sl}(2,\mathbb{C})$, the orthogonal group of $B$, $O_B =O(3,3)$ acts transitively and freely on $F$ turning it into a principal homogeneous space.  Any Lie algebra automorphism is preserved by the Killing form $B$ and $\mathrm{Aut}(\mathfrak{sl}(2,\mathbb{C}))$ is a closed subgroup of $O_B$.  Moreover, two orthonormal basis will have the same bracket relations if and only if they are in the same orbit under the action of $\mathrm{Aut}(\mathfrak{sl}(2,\mathbb{C}))$ on $O(3,3)$.  From \cite[Chapter 7, Section 3]{Bourbaki},   $\mathfrak{sl}(2,\mathbb{C})$ is the image of $\mathrm{SL}(2,\mathbb{C})$ under the adjoint representation $\mathrm{Ad}: \mathrm{SL}(2,\mathbb{C}) \longrightarrow \mathrm{GL}(\mathfrak{sl}(2,\mathbb{C}))$. Thus the orbit space $O(3,3)/\mathrm{Aut}(\mathfrak{sl}(2,\mathbb{C})) \simeq O(3,3)/\mathrm{Ad}(\mathrm{SL}(2,\mathbb{C}))$ is a homogeneous space of dimension 15 - 6 = 9, which means that we have an infinity of ways of writing the Lax-pair equations.

These considerations mean, therefore, that a similar direct approach to the study of completeness of pseudo-Riemannian metric on $\mathrm{SL}(2,\mathbb{C})$ is perhaps an unfeasible project.

\bigskip



We now turn our attention to the case of holomorphic metrics on $\mathrm{SL}(2,\mathbb{C})$.  Considering the complex Lie algebra
$\mathfrak{sl}(2,\mathbb{C})$, we remark that the
Killing form $B$ is now a non-degenerate complex bilinear form.

Clearly the Killing form corresponds to a complete holomorphic metric, since the Lax-pair system is simply given by $\dot{z} = 0$. However,
complete metrics in the holomorphic case are ``very rare''. To precise the claim, recall that in the case of (real) orthogonal Lie groups we were able to find an open set of $\R^n$,  $n=\mathrm{dim}\, G$, with respect to which every left-invariant pseudo-Riemannian metric associated with an
isomorphism $\Phi$ having eigenvalues in this open set is complete, cf. Theorem \ref{teo:complete-PR-metrics}.  Nevertheless, the same cannot be said in the holomorphic case. In fact, we have a complete characterization of completeness holomorphic metrics on $\mathrm{SL}(2,\mathbb{C})$, as follows.

\begin{teo}\label{Thm:incomplete-holomorphic-metrics}
Let $q$ be the holomorphic metric on $\mathfrak{sl}(2,\mathbb{C})$ defined by $q(X,Y) = B(\Phi X, Y)$, where $B$ is the Killing form
and $\Phi$ a $B$-self-adjoint isomorphism. Then, $q$ is a complete holomorphic metric if and only if $\Phi$ has an eigenvalue whose eigenspace has dimension at least 2.
\end{teo}

\begin{proof}
The proof of this theorem is divided in $3$ cases, namely the cases where
\begin{itemize}
\item[(a)] $\Phi$ is diagonalizable;
\item[(b)] $\Phi$ has an eigenvalue $\lambda$ such that $m_a(\lambda) - m_g(\lambda) = 1$;
\item[(c)] $\Phi$ has an eigenvalue $\lambda$ such that $m_a(\lambda) - m_g(\lambda) = 2$.
\end{itemize}

As is well-known from Sylvester's rigidity theorem, there is only one non-degenerate complex bilinear form in each dimension, up to
isomorphism. However, since this will be convenient in what follows, we will retain the definitions of orthonormal and
pseudo-orthonormal basis from Definition \ref{def: B-ON_B-PON}.

The proof of Lemma~\ref{lm:B} applies equally well to the present case. In fact, where in the original argument we have used the positive
sign of $B(v_1,v_1)$ (and of similar expressions), when the ground field is $\C$ all that is required is to have these terms different from
zero. Close reading through the proof shows that the arguments used in Lemma~\ref{lm:B} still work and the proof becomes in fact simpler.

As for the proof of Lemma~\ref{lm:A}, the arguments adapt as well and a similar result holds but we need to be more precise. The already
mentioned Sylvester's theorem implies that the group of isometries of the Killing form corresponds to $\mathrm{O}(3,\mathbb{C})$. As in
the real case, any automorphism $\varphi$ of the Lie algebra preserves the Killing form and $\mathrm{Aut}(\mathfrak{sl}(2,\mathbb{C}))$
is a closed subgroup of $\mathrm{O}(3,\mathbb{C})$. Writing the equation $[\varphi x, \varphi y] = \varphi [x,y]$   and considering $M$
the matrix of $\varphi$ with respect to the standard basis $(e_1,e_2,e_3)$, we get, by direct computation that, $c(M) = M$, where $c(M)$
is the cofactor matrix of $M$. Similiar arguments to those of Lemma ~\ref{lm:A}, show that  $\mathrm{Aut}(\mathfrak{sl}(2,\mathbb{C}))$
is then $\mathrm{SO}(3,\mathbb{C})$. Therefore, the bracket relations of bases satisfying certain orthogonality conditions are therefore
parametrized by $\mathrm{O}(3,\mathbb{C})/\mathrm{SO}(3,\mathbb{C}) \simeq \mathbb{Z}_2$. From the discussion above, we can conclude that the Lax-pair differential system for a holomorphic metric on $\mathrm{sl}(2,\mathbb{C})$ is nothing other than the
complexification of the corresponding system in cases $(a), (b)$ and $(c)$.

Consider first the case where $\Phi$ is diagonalizable and assume that $\Phi^{-1} = {\rm diag} \, (\nu_1, \nu_2, \nu_3)$. As seen
in Section~\ref{Sec:Case1}, the Euler-Arnold vector field is given by~(\ref{eq:EAVF1})
\[
E = a z_{2}z_{3} \dd{z_{1}} + b z_{1}z_{3} \dd{z_{2}} + c z_{1}z_{2} \dd{z_{3}} \, ,
\]
where now $z_1, z_2, z_3 \in \C$. Furthermore, we still have $a = \nu_2 - \nu_3$, $b = \nu_3 - \nu_1$ and $c = \nu_2 - \nu_1 = a+b$.
Assume that $\Phi$ has an eigenvalue whose eigenspace has dimension at least 2, i.e. there exist $i \ne j$ such that $\nu_i = \nu_j$.
This is equivalent to saying that $abc = 0$. The proof goes as in the first part of Proposition~\ref{prp:S1}. To be more precise, assuming,
without loss of generality that $a=0$, the first equation of the differential system~(\ref{sist_EA}) reduces then to $\dot{z}_1 = 0$.
Thus $z_1(t) = k$, with $k \in \CC$, for all $t \in \CC$ and the Euler-Arnold differential system reduces to a linear system in the
variables $z_2, \, z_3$, which is clearly complete.

Assume now that $\Phi$ has no eigenvalues whose eigenspace has dimension at least 2, which is equivalent to saying that $abc \ne 0$. Recall
that, being $E$ a polynomial vector field on $\C^3$, it admits a meromorphic extension to $\C\p(3)$. Again, if we consider the affine
coordinates $(x_1,x_2,x_3)$ related to $(z_1,z_2,z_3)$ through the map $\Psi$ on $\C^3 \setminus \{x_3 = 0\}$ taking on the form~(\ref{eq:x-chart}),
the Euler-Arnold vector field in the new coordinates is given by~(\ref{eq:X_VF}). Since $abc \ne 0$, the intersection of the singular
set of $\fol$, the foliation induced by $X$, with the plane at infinity $\Delta_{\infty}$ has exactly $5$~points, namely
\[
(0,0,0), \, \, \, \left(\sqrt{a/c},\sqrt{b/c},0\right), \, \, \, \left(\sqrt{a/c},-\sqrt{b/c},0\right), \, \, \,
\left(\sqrt{a/c},-\sqrt{b/c},0\right) \, \, \, \left(-\sqrt{a/c},-\sqrt{b/c},0\right) \, .
\]
The straight line $\{x_1 = \sqrt{a/c}, \, x_2 = \sqrt{b/c}\}$ is invariant by $\fol$ and the restriction of $X$ to it is a (non-zero)
constant vector field. The associated geodesic is then incomplete.

Consider now the case where $\Phi$ has an eigenvalue $\lambda$ such that $m_a(\lambda) - m_g(\lambda) = 1$. Consider a $B$-pseudo-orthonormal
basis where $\Phi^{-1}$ takes on the form~(\ref{eq:Phi-1-3}) and the Euler-Arnold vector field is written as~(\ref{sist_EA3})
\[
E = b z_3^2 \dd{z_{1}} - z_1 \left(a z_2 + b z_3\right) \dd{z_{2}} + a z_1 z_3 \dd{z_{3}} \, ,
\]
where $a=\eta-\nu$ and $b= \zeta \nu^2$. If $\Phi$ has an eigenvalue whose eigenspace has dimension at least 2, then $a = 0$. The third equation of the differential
system~(\ref{sist_EA3}) reduces then to $\dot{z}_3 = 0$, which means that $z_3(t) = k$, for some $k \in \CC$ and for all $t \in \CC$.
Hence, the Euler-Arnold differential vector field reduces to $bk^2 \dd{z_{1}} - bk z_1 \dd{z_{2}}$, which is clearly complete since
the solution of the associated differential system is polynomial with respect to $t$. Assume then that $\Phi$ has no eigenvalue whose
eigenspace has dimension at least 2. By following the calculations in Proposition~\ref{prp:N3}, it becomes clear that the singular
set of the foliation induced by the Euler-Arnold vector field reduces to two points in the chart associated with the coordinates
$(x_1,x_2,x_3)$. The points in question are given by $p_1 = \left( \sqrt{b/a}, -b/(2a),0 \right)$ and $p_2 = \left( -\sqrt{b/a},
-b/(2a),0 \right)$. The straight lines ``above'' each one of these singular points are regular leaves for $\fol$ since the restriction
of $X$ to them is a (non-zero) constant vector field. The associated geodesics are then incomplete.

Finally, consider the case (c). Note that in the present case there is no eigenvalue whose eigenspace has dimension at least 2.
Theorem~\ref{thm:slR(S4)} applies equally well to the present case, thus the corresponding metrics are incomplete.
\end{proof}

\medskip

It is a well-known fact that any complex semisimple Lie algebra is built out of copies of $\mathfrak{sl}(2,\mathbb{C})$ in a certain way (by means of root systems). This allows us to prove the following corollary, whose proof is completely analogous to that of the real case (cf. Corollary \ref{cor:incomplete-pr-metrics}).

\begin{coro}
Let $G$ be a complex semisimple Lie group. Then $G$ can be equipped with incomplete holomorphic metrics.
\end{coro}

To finish this paper, let us consider again the previous family of holomorphic metrics. Albeit their geodesic flow not being complete in general,
it turns out to always being semicomplete as it will be seen below (Theorem~\ref{Thm:semicomplete}). In fact, as mentioned in the Introduction,
their Euler-Arnold vector fields give rise to a $2$-parameter family of iso-spectral quadratic vector fields on $\C^3$ which, in addition,
are semicomplete and naturally associated with certain elliptic surfaces birational to the complex projective plane, cf. \cite{Guillot}.

\begin{teo}\label{Thm:semicomplete}
Let $q$ be the holomorphic metric on $\mathfrak{sl}(2,\mathbb{C})$ defined by $q(X,Y) = B(\Phi X, Y)$, where $B$ is the Killing form
and $\Phi$ is a $B$-self-adjoint isomorphism. Then the associated Euler-Arnold vector field is semicomplete on $\C^3$.
\end{teo}

The argument used in the proof is adapted from an unpublished manuscript by J. Rebelo~\cite{Reb-manuscript}. We thank J. Rebelo for making his notes
available to us.

\begin{proof}
Again, the proof of this theorem is divided in the three cases (a), (b) and (c) listed in the proof of Theorem~\ref{Thm:incomplete-holomorphic-metrics}.
Let us begin with the diagonalizable case. As previously mentioned, the Euler-Arnold vector field is given by Formula~(\ref{eq:EAVF1}).
As already seen, if $abc = 0$, then the Euler-Arnold vector field is complete on $\C^3$ so that there is nothing to prove.
So, let us assume from now on that none of $a, \, b$ and $c$ is zero.

Consider then the meromorphic extension of $E$ to $\C \p(3)$ and let $\fol$ stand for the associated foliation. As previously seen, $\fol$
possesses seven singular points in $\Delta_{\infty}$. Three of these singular points are given by the intersection of
$\Delta_{\infty}$ with each of the coordinate axis of the initial coordinates $(z_1,z_2,z_3)$. These axes are, in fact, fully constituted
by singular points of $\fol$. The three projective lines determined by the mentioned coordinate axes along with the remaining 4 singular
points (denoted by $p_1, \, p_2, \, p_3$ and $p_4$ in Section~\ref{Sec:Case1}) in $\Delta_{\infty}$ make up for the singular set of $\fol$.
The three points of $\Delta_{\infty}$ determined by the coordinate axes of $(z_1,z_2,z_3)$ are denoted by $q_x, \, q_y$ and $q_u$.

Fix a leaf $L$ of $\fol$ not contained in $\Delta_{\infty}$. Recall that $L$ is a Riemann surface contained in an algebraic curve of
$\C \p(3)$. By following the calculations presented in Section~\ref{Sec:Case1}, we know that whenever the projection of $L$ on $\Delta_{\infty}$
approaches the singular points $q_x$ (resp. $q_y$, $q_u$), then the leaf $L$ itself accumulates at a singular point of $\fol$ lying in the coordinate
axis $z_1$ (resp. $z_2$, $z_3$) different from $q_x$ itself (resp. $q_y$, $q_u$). In particular, locally, the leaf $L$ remains away from infinity so
that the vector field restricted to the leaf $L$ is holomorphic around the singularity in question.

It remains to study the restriction of the Euler-Arnold vector field to $L$ nearby the singular points $p_i, \,  i=1..4$.
It is enough to consider the case of $p_1$ since the case for the other singular points are analogous. To begin with, recall from the calculations in
Section~\ref{Sec:Case1} that a leaf $L$ whose projection in $\Delta_{\infty}$ approaches $p_1$ defines a separatrix
of $\fol$ through $p_1$. Furthermore, since the eigenvalues of $\fol$ at $p_1$ are $2, \, 2, \, 1$, where the eigenvalue associated with
the direction transverse to the $\Delta_{\infty}$ is 1, there follows that the separatrix induced by $L$ is smooth and admits an irreducible
Puiseux parametrization $\sigma(t)$ of the form $\sigma(t) = (\sigma_1(t), \sigma_2(t), t)$, where $\sigma_1(0) = \sigma_2(0) = \sigma_1'(0)
= \sigma_2'(0) = 0$.

Consider next the restriction $E|_L$ of $E$ to $L$ and its pull-back $\sigma^{\ast} E|_L$ by $\sigma$. In particular, $\sigma^{\ast} E|_L$
is a vector field defined on a punctured neighborhood of $0$ in $\C$.

\bigbreak

\noindent {\it Claim:} The vector field $\sigma^{\ast} E|_L$ admits a holomorphic extension as a regular vector field to $0 \in \C$.

\begin{proof}[Proof of the Claim]
Let $\sigma^{\ast} E|_L = F(t) \partial /\partial t$. Owing to Riemann's Theorem, it suffices to prove that the limit of $F$ at $0$ exists
and is non-zero. Consider then local coordinates $(x,y,z)$ nearby $p_1$ where $p_1 \simeq (0,0,0)$ and $\Delta_{\infty} \simeq \{z = 0\}$.
It can be easily checked that in these coordinates, the vector field $E$ takes on the form
\[
\frac{1}{z} \left[ (2x + {\rm h.o.t.}) \frac{\partial}{\partial x} + (2y + {\rm h.o.t.}) \frac{\partial}{\partial y}
+ (z + {\rm h.o.t.}) \frac{\partial}{\partial z} \right] \, ,
\]
up to a (non-zero) multiplicative constant. Thus the pull-back of the restriction of $E$ to $L$ by $\sigma$ is given by
\[
(1 + {\rm h.o.t}) \frac{\partial}{\partial t}
\]
up to the same multiplicative constant. This means that $\sigma^{\ast} E|_L$ is regular at $0 \in \C$ as we intended to prove.
\end{proof}

The normalization $\widehat{L}$ of the compactification of $L$ in $\C\p(3)$ is a (compact) Riemann surface equipped with a globally defined holomorphic
vector field. This vector field if therefore complete on $\widehat{L}$. Since the restriction of $E$ to $L$ is identified with the restriction of the
mentioned vector field to the complement of finitely many points in $\widehat{L}$, it must be semicomplete as the restriction of a complete vector
field to an open set. This holds for an arbitrary leaf of $\fol$, we conclude that $E$ is semicomplete on $\C^3$.

Consider next the case where $\Phi$ has an eigenvalue $\lambda$ such that $m_a(\lambda) - m_g(\lambda) = 1$ and consider a $B$-pseudo-orthonormal
basis where $\Phi^{-1}$ takes on the form~(\ref{eq:Phi-1-3}), so that the Euler-Arnold vector field is written as~(\ref{sist_EA3})
\[
E = b z_3^2 \dd{z_{1}} - z_1 \left(a z_2 + b z_3\right) \dd{z_{2}} + a z_1 z_3 \dd{z_{3}} \, ,
\]
where $a=\eta-\nu$ and $b= \zeta \nu^2$. As previously seen, if $\Phi$ has an eigenvalue whose eigenspace has dimension at least 2, i.e. if $a = 0$,
then the Euler-Arnold vector field is complete on $\C^3$. Thus it is semicomplete on $\C^3$ as well. Assume then that $a \ne 0$.

Given the preceding construction, it suffices to show that whenever a leaf $L$ of $\fol$ induces a separatrix for a singularity of $\fol$
lying in $\Delta_{\infty}$, the restriction of $X$ to the resulting separatrix is holomorphic. In the present case, $\fol$ possesses four
singular points in $\Delta_{\infty}$, namely
\begin{itemize}
\item[(i)] $p_1 = \left( \sqrt{b/a}, -b/(2a), 0 \right)$ and $p_2 = \left(-\sqrt{b/a}, -b/(2a),0 \right)$ in the affine coordinates $(x_1,x_2,x_3)$.
The eigenvalues of $\fol$ at these points are $2, \, 2, \, 1$,  where $1$ is the eigenvalue associated to the direction transverse to $\Delta_{\infty}$.
These points can be treated exactly as the points with the same eigenvalues in the previous case.

\item[(ii)] The intersection of $\Delta_{\infty}$ with the axis $z_1$, denoted by $p_3$. The eigenvalues of $\fol$ at these points are $1, \, -1, \, 0$,
where $0$ is the eigenvalue associated to the direction transverse to $\Delta_{\infty}$. In particular, the foliation induced on $\Delta_{\infty}$ has
a saddle behavior. This implies that the leaves in the invariant planes defined by the two separatrices though $p_3$ are the unique leaves that can
accumulate at $p_3$. By expressing the Euler-Arnold vector field in the affine coordinates, it can easily be checked that a leaf $L$ in these invariant
planes never accumulates at $p_3$ unless it is totally contained in $\Delta_{\infty}$. Summarizing, no leaf of $\fol$ induces a separatrix
at $p_3$ so that this singular point plays no further role in the present discussion.

\item[(iii)] The intersection of $\Delta_{\infty}$ with the axis $z_2$, denoted by $p_4$. The foliation has a singular point of order~$2$ at $p_4$.
In particular all eigenvalues are zero. This case is discussed in detail below.
\end{itemize}

Consider the affine coordinates $(v_1,v_2,v_3)$ where $(v_1/v_2,1/v_2,v_3/v_2) = (z_1,z_2,z_3)$. The first integrals $I_1$ and
$I = \nu I_1 - I_2$ (cf. Section~\ref{Sec:Case3}) in these coordinates are given by
\[
I_1 = \frac{v_1^2 + 2v_3}{v_2^2} \qquad \text{and} \qquad I = \frac{-av_1^2 + bv_3^2}{v_2^2} \, .
\]
From $I_1$, we have that $v_3$ is a function of $v_1$ and $v_2$. In fact, $v_3 = (kv_2^2 - v_1^2)/2$ so that the corresponding invariant surfaces
for $\fol$ pass through the origin for every $k \in \C$ and they are tangent to the plane $\{v_3 = 0\}$ at the point in question. In turn, by substituting
$v_3 = v_3(v_1,v_2)$ we can check that, indeed, the leaves pass through the origin and that they are transverse to $\Delta_{\infty}$ (locally
given by $\{v_2 = 0\}$). The leaf $L$ admits then a Puiseux parametrization of the form $\sigma(t) = (t, \, \alpha t + {\rm h.o.t}, \, t^2 + {\rm h.o.t.})$,
for some $\alpha \in \C^{\ast}$. Noticing that the Euler-Arnold vector field in the affine coordinates $(v_1,v_2,v_3)$ is given by
\[
\frac{1}{v_2} \left[ \left( a v_1^{2} + b v_3(v_1^2 + v_3) \right) \dd{v_1} + v_1 v_2 (a + b v_3) \dd{v_2} + v_1 v_3 (2a + b v_3) \dd{v_3} \right] \, ,
\]
the pull-back of the restriction of $E$ to $L$ by $\sigma$ can be holomorphically extended to $0 \in C$ as a vector field of the form
$F(t) \partial /\partial t$, with $F(0) = 0$ and $F'(0) \ne 0$. We then conclude that $L$ may be compactified as an algebraic
Riemann surface equipped with a global holomorphic vector field, so that it is complete.

Finally, consider case (c) and the $B$-pseudo-orthonormal basis with respect to which the isomorphism $\Phi$ takes on the form presented in
Section~\ref{sec: LP_equations}. If $(z_1,z_2,z_3) \in \C^3$ stands for the coordinates of $z \in \slR$ in the mentioned $B$-pseudo-orthonormal
basis, the Euler-Arnold vector field is given in these coordinates by
\begin{equation}
E = \zeta \nu^2 \left[ z_3(-\zeta \nu z_3 + z_1) \dd{z_{1}} + (z_{2}z_{3} - z_{1}^{2} + \zeta \nu z_{1}z_{3}) \dd{z_{2}}
- z_{3}^{2} \dd{z_{3}} \right] \, .
\end{equation}
We keep the preceding notations so that $\fol$ stands for the corresponding singular holomorphic foliation in $\C\p(3)$. The singular set
of $\fol$ consists of
\begin{itemize}
\item[(i)] an isolated singular point $p \in \Delta_{\infty}$ given in the affine coordinates $(x_1,x_2,x_3)$ by $p = \left( \frac{\zeta \nu}{2},
-\frac{\zeta^2 \nu^2}{8} ,0 \right)$. The eigenvalues of $\fol$ at $p$ are $2, \, 2, \, 1$,  where $1$  is the eigenvalue associated to the direction
transverse to $\Delta_{\infty}$. This point can be treated exactly as the points with the same eigenvalues in the previous cases.

\item[(ii)] the projective line arising from the $z_2$-axis. The intersection point of this projective line with $\Delta_{\infty}$ will be denoted
by $q$ and coincides with the origin of the affine coordinates $(v_1,v_2,v_3)$. We detail below how to treat the leaves $L$ of $\fol$ yielding a
(local) separatrix for $\fol$ at $q$.
\end{itemize}

In the coordinates $(v_1,v_2,v_3)$ the foliation $\fol$ is determined by the (two independent) first integrals
\[
I_1 = \frac{v_1^2 + 2v_3}{v_2^2} \qquad \text{and} \qquad I_2 = \nu I_1 + \frac{\zeta^2 \nu^3 v_3^2 - 2\zeta \nu^2 v_1 v_3}{v_2^2} \, .
\]
In particular, if $L$ is a leaf of $\fol$ whose projection in $\Delta_{\infty}$ approaches $q$, then the actual leaf $L$ has to approach $q$.
Following the same argument used in case (b), a separatrix induced by $L$ is smooth at $q$ and admits an
irreducible Puiseux parametrization of the form $\sigma(t) = (t, \alpha t + {\rm h.o.t}, t^2 + {\rm h.o.t})$, for some $\alpha \in
\C^{\ast}$. Finally, by noticing that the Euler-Arnold vector field in the present coordinates takes on the form
\[
V = \frac{ba^2}{v_2} \left[ ( v_1^3 - ab v_1^2v_3 - ab v_3^2 ) \frac{\partial}{\partial v_1}
+ v_2 (v_1^2 - v_3 - ab v_1 v_3) \frac{\partial}{\partial v_2} + v_3 (v_1^2 - 2v_3 - ab v_1 v_3) \frac{\partial}{\partial v_3} \right] \ .
\]
there follows that $\sigma^{\ast} V|L$ admits a holomorphic extension to $0 \in \C$ of the form $F(t) \partial/\partial t$, for some $F$
such that $F(0) = F'(0) = 0$ and $F''(0) \ne 0$. Again, this shows that the restriction of the Euler-Arnold vector field to the normalization
of the compactification of $L$ is holomorphic and ends the proof of the theorem.
\end{proof}


\bigbreak

\bigbreak

\noindent {\bf Acknowledgements}.

The authors are grateful to Julio Rebelo for providing us with a copy of his unpublished manuscript~\cite{Reb-manuscript}
and would also like to thank Ilka Agricola, Miguel Sánchez and Abdelghani Zeghib for their valuable comments
on this manuscript.

The first author was financed by FCT - Funda\c{c}\~ao para a Ci\^encia e Tecnologia, I.P.  (Portugal) - through the PhD grant PD/BD/143019/2018.
The second author was partially supported by FCT through the sabbatical grant SFRH/BSAB/135549/2018 and through CMAT under the project UID/MAT/00013/2013.
The third author was partially supported by CMUP, member of LASI, which is financed by national funds through FCT under the project UIDB/00144/2020
and also by CIMI through the project ``Complex dynamics of group actions, Halphen and Painlev\'e systems''. Finally, all three authors benefited from
CNRS (France) support through the PICS project ``Dynamics of Complex ODEs and Geometry''.

\vspace*{6mm}

{\footnotesize

{\sc Ahmed Elshafei},  Centro de Matem\'atica da Universidade do Porto,  Centro de Matem\'atica da Universidade do Minho,
Portugal,  {\tt a.el-shafei@hotmail.fr}

{\sc Ana Cristina Ferreira}, Centro de Matem\'atica da Universidade do Minho,  Campus de Gualtar,  4710-057 Braga, Portugal,
{\tt anaferreira@math.uminho.pt}

{\sc Helena Reis} , Centro de Matem\'atica da Universidade do Porto, Faculdade de Economia da Universidade do Porto, Portugal, {\tt hreis@fep.up.pt}

}



\begin{thebibliography}{19}

\bibitem{Arnold-paper}
V. Arnold,  Sur  la  g\'eometrie  diff\'erentielle  des  groupes  de  Lie  de  dimension  infinie  et  ses applications \`a l'hydrodynamique des fluides parfaits. Ann. Inst. Fourier (1966), 319--361.


\bibitem{Arnold-book}
V. Arnold, \emph{Mathematical methods of classical mechanics},
Springer-Verlag, New York, second edition (1989).


\bibitem{B-D}
I. Biswas, S. Dumitrescu, \emph{Holomorphic Riemannian metric and fundamental group}, Bull. Soc. Math. Fr. 147, no. 3 (2019), 455--468.


\bibitem{B-M}
S. Bromberg, A. Medina, \emph{Geodesically complete Lorentzian metrics on some homogeneous 3 manifolds},
SIGMA Symmetry Integrability Geom. Methods Appl. 4 (2008), Paper 088, 13 pp.


\bibitem{Bourbaki}

N. Bourbaki, \emph{Lie groups and Lie Algebras}, Chapters 8-9, Elements of Mathematics, Springer (2005).


\bibitem{D}
S. Dumitrescu, \emph{Métriques riemanniennes holomorphes en petite dimension}, Ann. Inst. Fourier 51, no. 6 (2001), 1663--1690.


\bibitem{D-Z}
S. Dumitrescu, A. Zeghib, \emph{Global rigidity of holomorphic Riemannian metrics on compact complex 3-manifolds}, Math. Ann. 345, no.1 (2009), 53--81.


\bibitem{A_thesis} {\sc A. Elshafei}, On completeness of Halphen systems and of pseudo-Riemannian geodesics flows,
{\it PhD Thesis}, University of Porto, April 2022.


\bibitem{FRR}
A.C. Ferreira, J.C. Rebelo, H. Reis, \emph{Palais leaf-space manifolds and surfaces carrying holomorphic flows}. Mosc. Math. J. 19 (2019), no. 2, 275-305


\bibitem{gautier}
S. Gautier, \emph{Quadratic centers defining elliptic surfaces}, J. Differential Equations 245, no. 12 (2008), 3545–3569


\bibitem{G-L}
M. Guediri, J. Lafontaine, \emph{Sur la complétude des variétés pseudo-riemanniennes},
J. Geom. Phys. 15, no. 2 (1995), 150--158.


\bibitem{Guillot}
A. Guillot, \emph{Semicompleteness of homogeneous quadratic vector fields}, Ann. Inst. Fourier, 56, no. 5 (2006), 1583-1615


\bibitem{Guillot-CRAS}
A. Guillot, \emph{Sur les exemples de Lins Neto de feuilletages algébriques}, C. R. Acad. Sci. Paris, Ser I 334 (2002), 747-750


\bibitem{Knapp}
A. Knapp, \emph{Lie Groups Beyond an Introduction}, Birkhauser, second edition (2002).

\bibitem{Lafuente}
J. Lafuente L\'opez, \emph{A geodesic completeness theorem for locally symmetric Lorentz mani-
folds}, Rev. Mat. Univ. Complut. Madrid 1 (1988) 101--110.


\bibitem{LeBrun}
C. LeBrun, \emph{Spaces of complex null-geodesics in complex-Riemannian geometry},
Trans. Amer. Math. Soc. 278, no. 1 (1983), 209--231.


\bibitem{Lee}
D.H. Lee, \emph{The structure of complex Lie groups}, Chapman \& Hall/CRC (2002).

\bibitem{Marsden}
J. E. Marsden, \emph{On completeness of pseudo-Riemannian manifolds}, Indiana Univ. Math. J. 22
(1973), 1065--1066.


\bibitem{MR}
A. Medina, P. Revoy, \emph{Alg\'ebres de Lie et produit scalaire invariant},  Annales Scientifiques de l'\'E.N.S, 4e s\'erie, tome 18, no. 3 (1985), 553--561.





\bibitem{ON}
B. O'Neill, \emph{Semi-Riemannian geometry with applications to relativity}, Academic Press (1983).


\bibitem{Palais}
R. Palais, \emph{A global formulation of the Lie theory of transformation groups}, Mem. Amer. Math. Soc., no. 22 (1957)


\bibitem{Reb}
J.C. Rebelo, \emph{Singularit\'es des flots holomorphes}, Ann. Inst. Fourier, Vol. 46, no. 2 (1996), 411-428.


\bibitem{Reb-manuscript}
J.C. Rebelo, \emph{On the structure of singularities of holomorphic flows}, manuscript.


\bibitem{RR}
J.C Rebelo, H. Reis, \emph{Local theory of holomorphic foliations and vector fields}, available at arXiv:1101.4309 (2011).

\bibitem{RR_Applications}
J.C Rebelo, H. Reis, \emph{Uniformizing complex ODE'S and applications}, Revista Matemática Iberoamericana, Vol. 30, no. 3 (2014), 799–874


\bibitem{Sanchez}
M. S\'anchez, \emph{On the completeness of trajectories for some mechanical systems}, In: Chang, D., Holm, D., Patrick, G., Ratiu, T. (eds) Geometry, Mechanics, and Dynamics.  Springer, vol. 73 (2015), 343--372.

\bibitem{Tholozan}
N. Tholozan, \emph{Uniformisation des vari\'et\'es pseudo-riemanniennes localement homog\`enes},
PhD. Thesis, Univ. Nice-Sophia Antipolis (2014).



\end{thebibliography}
\end{document}